\documentclass[reqno]{amsart}
\usepackage{amssymb,amsmath}
\usepackage{color}
\usepackage{graphicx}
\usepackage{caption}
\usepackage{subcaption}
\usepackage{epstopdf}

\usepackage{tikz}
\usetikzlibrary{arrows,automata,backgrounds,calendar,mindmap,trees}
\usepackage{xcolor}

\begin{document}

\author[T.Q. Bao, K. Fellner, S. Rosenberger] {Tang Quoc Bao, Klemens Fellner, Stefan Rosenberger}

\address{Tang Quoc Bao$^{1, 2}$ \hfill\break
$^{1}$ Institute of Mathematics and Scientific Computing, University of Graz, Heinrichstra{\ss}e 36, 8010 Graz, Austria\hfill\break
$^{2}$ Faculty of Applied Mathematics and Informatics, Hanoi University of Science and Technology, 1 Dai Co Viet, Hai Ba Trung, Hanoi, Vietnam}
\email{quoc.tang@uni-graz.at} 

\address{Klemens Fellner \hfill\break
Institute of Mathematics and Scientific Computing, NAWI Graz, University of Graz, Heinrichstra{\ss}e 36, 8010 Graz, Austria}
\email{klemens.fellner@uni-graz.at}

\address{Stefan Rosenberger \hfill\break
Institute of Mathematics and Scientific Computing, University of Graz, Heinrichstra{\ss}e 36, 8010 Graz, Austria}
\email{stefan.rosenberger@edu.uni-graz.at}


\title[QSSA and numerics for a volume-surface reaction-diffusion system]{Quasi-steady-state approximation\\ and numerical simulation for a \\volume-surface reaction-diffusion system}

%\author{Klemens Fellner\thanks{Institute of Mathematics and Scientific Computing, NAWI Graz, University of Graz, Heinrichstra{\ss}e 36, 8010 Graz, Austria,  klemens.fellner@uni-graz.at}\and{Stefan Rosenberger\thanks{Institute of Mathematics and Scientific Computing, University of Graz, Heinrichstra{\ss}e 36, 8010 Graz, Austria,  stefan.rosenberger@edu.uni-graz.at}}\and{Bao Quoc Tang\thanks{Institute of Mathematics and Scientific Computing, University of Graz, Heinrichstra{\ss}e 36, 8010 Graz, Austria\; and \; Faculty of Applied Mathematics and Informatics, Hanoi University of Science and Technology, 1 Dai Co Viet, Hai Ba Trung, Hanoi, Vietnam, quoc.tang@uni-graz.at}}}

 %        \pagestyle{myheadings} \markboth{A reaction-diffusion model of asymmetric stem-cell divisions}{K. Fellner, S. Rosenberger, B.Q. Tang} \maketitle

\begin{abstract}
The asymmetric stem-cell division of Drosophila SOP precursor cells is driven by the asymmetric localisation of the key protein Lgl (Lethal giant larvae) during mitosis, when Lgl is phosphorylated by the kinase aPKC on a subpart of the cortex and subsequently released into the cytoplasm.

In this paper, we present a volume-surface reaction-diffusion system, which models the localisation of Lgl within the cell cytoplasm and on the cell cortex. We prove well-posedness of global solutions as well as regularity of the solutions. Moreover, we rigorously perform the fast reaction limit to a reduced quasi-steady-state approximation system, when phosphorylated Lgl is instantaneously expelled from the cortex. Finally, we apply a suitable first order finite element scheme to simulate and discuss interesting numerical examples, which illustrate {i) the influence of the presence/absence of surface-diffusion to the behaviour of the system and the complex balance steady state and ii) the dependency on the release rate of phosphorylated cortical Lgl.}
 
          \end{abstract}
          
\maketitle
\numberwithin{equation}{section}
\newtheorem{theorem}{Theorem}[section]
\newtheorem{lemma}[theorem]{Lemma}
\newtheorem{proposition}[theorem]{Proposition}
\newtheorem{definition}{Definition}[section]
\newtheorem{remark}{Remark}[section]
          
%\begin{keywords}
%Reaction-Diffusion Equations; Global Existence; Surface Diffusion; Quasi-steady-state Approximation; Asymmetric Stem Cell Division; Finite Element Method.
%\end{keywords}

% \begin{AMS}
% 		{35K57, 35B40, 92C45}
%\end{AMS}
         \section{Introduction}
         
In stem cells undergoing asymmetric cell division, particular proteins (so-called cell-fate determinants) are localised at the cortex of only one of the two daughter cells during mitosis. These cell-fate determinants trigger in the following the differentiation of one daughter cell into specific tissue while the other daughter cell remains a stem cell. 
         
In Drosophila, SOP precursor stem cells provide a well-studied biological example model of asymmetric stem cell division, see e.g. \cite{BMK,MEBWK,WNK} and the references therein. 
In particular, asymmetric cell division of SOP cells is driven by the asymmetric localisation of the key protein Lgl (Lethal giant larvae), which exists in two conformational states: a non-phosphorylated form which regulates the localisation of the cell-fate-determinants in the membrane of one daughter cell, and a phosphorylated form which is inactive.

The asymmetric localisation of Lgl during mitosis is the result of the activation of the kinase aPKC, which phosphorylates Lgl (as part of a highly evolutionary conserved protein complex) only on a subpart of the cortex, as well as the weakly reversible reaction/sorption dynamics of the two conformations of Lgl between cortex and cytoplasm. In particular, it is the (fast) irreversible release of phosphorylated Lgl from the cortex, which initiates the asymmetric localisation of Lgl upon the activation of aPKC.

While the asymmetric localisation of Lgl is essential for the asymmetric cell division of SOP cells, the subsequent biological machinery, where the asymmetric localisation of Lgl leads to the asymmetric localisation allocation of the adaptor protein Pon (Partner of Numb) and the cell-fate determinate Numb is currently not well enough understood to be considered here.

\medskip
         
In this paper, we shall present and study a volume-surface reaction-diffusion model system describing the evolution of Lgl in its non-phosphorylated and phosphorylated conformations both in the cytoplasm (i.e. in the cell volume) and at the cortex (i.e. the surface/membrane of the cell). 

More precisely, we shall denote by $L(t,x)$ and $P(t,x)$ the cytoplasmic concentrations of non-phosphorylated and phosphorylated Lgl
          within the bounded cell domain $\Omega\subset \mathbb R^n$, while $l(t,x)$ and $p(t,x)$ 
         denote the cortical concentrations of the non-phosphorylated and phosphorylated Lgl at the boundary $\Gamma:= \partial \Omega$, which is assumed sufficiently smooth (e.g. $C^{2+\alpha}$ with $\alpha >0$).
         
         The reaction kinetics between the species $L$, $P$, $l$ and $p$ are depicted in Figure \ref{LGLmodel} and summarise the following processes: i) a reversible reaction between $L$ and $P$ with rates $\alpha$ and $\beta$ on the domain $\Omega$, ii) a reversible exchange between $L$ and $l$ at the boundary $\Gamma$ with rates $\lambda$ and $\gamma$, iii) an irreversible 
         phosphorylation of $l$-Lgl into $p$-Lgl at the boundary $\Gamma$ with rate $\sigma$ and iv) an irreversible release of 
         $p$-Lgl from the boundary $\Gamma$ into the domain $\Omega$ with rate $\xi$. 
         
 \medskip
         
We emphasise that these reaction/sorption processes jointly conserve the total mass of Lgl (see the conservation law \eqref{cons} below). Moreover, the dynamics of Fig. \ref{LGLmodel} forms a so-called weakly-reversible or complex balance reaction network, for which the convergence towards a steady state and well as the structure of the steady states are significantly more subtle than for detailed balance models, see the discussion of the numerical examples in Section \ref{sec:3} or also \cite{FPT, Hor72} and the references therein.

         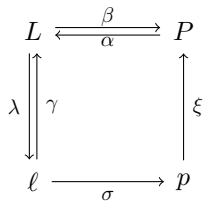
\begin{figure}[htp]
         \begin{center}\scalebox{1}[1]{
         \begin{tikzpicture}
         \node (a) {$L$} node (b) at (2,0) {$P$} node (c) at (0,-2) {$\ell$} node (d) at (2,-2)  {$p$};
         \draw[arrows=->] ([xshift =0.5mm,yshift=0.5mm]a.east) -- node [below] {\scalebox{.8}[.8]{$\alpha$}} ([xshift =-0.5mm,yshift=0.5mm]b.west);
         \draw[arrows=->] ([xshift =-0.5mm,yshift=-0.5mm]b.west) -- node [above] {\scalebox{.8}[.8]{$\beta$}} ([xshift =0.5mm,yshift=-0.5mm]a.east);
         \draw[arrows=->] ([xshift =0.5mm,yshift=0.5mm]c.north) -- node [right] {\scalebox{.8}[.8]{$\gamma$}} ([xshift =0.5mm,yshift=-0.5mm]a.south);
         \draw[arrows=->] ([xshift =-0.5mm,yshift=-0.5mm]a.south) -- node [left] {\scalebox{.8}[.8]{$\lambda$}} ([xshift =-0.5mm,yshift=0.5mm]c.north);
         \draw[arrows=->] ([xshift =0.5mm]c.east) -- node [below] {\scalebox{.8}[.8]{$\sigma$}} ([xshift=-0.5mm]d.west);
         \draw[arrows=->] ([yshift =0.5mm]d.north) -- node [right] {\scalebox{.8}[.8]{$\xi$}} ([yshift=-0.5mm]b.south);
         \end{tikzpicture}} \caption{The weakly-reversible reaction dynamics of $L, P, l$ and $p$.}\label{LGLmodel}
         \end{center}
         \end{figure}
         
In the following, we propose  a continuum model of partial differential equations, which describe the reactions and the diffusion processes of these species both on the domain $\Omega$ and on its surface $\Gamma$. 
The choice of a continuum model is based on the biological observation that protein concentrations in SOP cells are rather large and that 
stochastic effects in the concentrations can thus be neglected, see \cite{BMK}.  

The aim of the model is to qualitatively study, mathematically analyse and 
numerically simulate  
the reaction-diffusion dynamics between phosphorylated and un-phosphorylated Lgl 
in the cytoplasm and on the cortex with particular emphasis on the effect of surface diffusion, see Sections \ref{sec:3} and \ref{sec:conc}.

In SOP stem cells, the phosphorylation of Lgl occurs at the boundary $\Gamma$ by means of an atypical protein kinase aPKC, which is pre-located at a sub-part $\Gamma_2\subset\Gamma$. We shall thus 
         assume that $\Gamma$ is the union of two disjoint subsets $\Gamma = \Gamma_1\cup \Gamma_2$, in which $\Gamma_2$ is connected has a smooth boundary 
         $\partial\Gamma_2$. In case that $\Gamma_1 = \emptyset$, then $\Gamma_2 \equiv \Gamma$ is a surface in $\mathbb R^n$ without boundary. 
          
We consider the following equations for
         the volume concentrations $L$ and $P$: 
         \begin{equation}
         \begin{cases}
         L_t - d_L\Delta L = \alpha P - \beta L, \quad &\qquad x\in\Omega, \quad t>0,\\
         P_t - d_P\Delta P = -\alpha P + \beta L, \quad &\qquad x\in\Omega, \quad t>0,\\
         L(0, x) = L_0(x), \quad P(0,x) = P_0(x), &\qquad x\in\Omega,
         \end{cases}
         \label{f1}
         \end{equation}
         where $\Delta$ denotes the Laplacian on the domain $\Omega$, $d_L$, $d_P$ are positive volume-diffusion coefficients, $\alpha$, $\beta$ are positive and constant reaction rates,  and $L_0(x)$ and
         $P_0(x)$ are given initial concentrations.  
         
         The volume concentrations $L$ and $P$ are connected to the surface concentrations $l$ and $p$ in terms of Robin- and Neumann boundary conditions
         \begin{equation}
         \begin{cases}
         d_L\frac{\partial L}{\partial \nu}  = - \lambda L+\gamma l, \qquad &x\in\Gamma, \quad t>0,\\
         d_P\frac{\partial P}{\partial \nu} = \chi_{\Gamma_2}\xi p, \qquad &x\in\Gamma, \quad t>0,
         \end{cases}
         \label{f2}
         \end{equation}
         where $\nu(x)$ denotes the unit normal outward vector at $x\in \Gamma$, 
         $\gamma$, $\lambda$, and $\xi$ are positive and 
         constant reaction rates and $\chi_{\Gamma_2}$ denotes the characteristic function localising the aPKC-active part of the boundary $\Gamma_2$, 
         i.e. $\chi_{\Gamma_2}(x) = 1$ if $x\in \Gamma_2$ and $\chi_{\Gamma_2}(x) = 0$ otherwise. 
We remark that we consider smoothed versions of $\chi_{\Gamma_2}(x)$ biologically equally justified, yet mathematically less general and challenging.

         Thirdly, the surface concentrations $l$ and $p$ satisfy
         \begin{equation}
         \begin{cases}
         l_t - d_l\Delta_{\Gamma}l = \lambda L- (\gamma + \sigma \chi_{\Gamma_2})l, &\qquad x\in\Gamma,\quad t>0,\\
         p_t - d_p\Delta_{\Gamma_2}p = \sigma l- \xi p , &\qquad x\in\Gamma_2,\quad t>0,\\
         d_p\frac{\partial p}{\partial \nu_{\Gamma_2}} = 0, &\qquad x\in\partial\Gamma_2,\\
         l(0,x) = l_0(x), &\qquad x\in\Gamma,\\
         p(0,x) = p_0(x), &\qquad x\in\Gamma_2,
         \end{cases}
         \label{f3}
         \end{equation}
         where $\Delta_{\Gamma}$ and $\Delta_{\Gamma_2}$ are Laplace-Beltrami operators (see e.g. \cite{GT}) acting on the surfaces $\Gamma$ and $\Gamma_2$, respectively, $d_l$, $d_p$ are non-negative surface-diffusion coefficients and $\sigma>0$ is the positive and constant phosphorylation rate.  
         
         The considered evolution process conserves the total mass of Lgl, which is expressed in the following conservation law:
         \begin{equation}
         \int_{\Omega} (L(t,x)+P(t,x)) \, dx + \int_{\Gamma} l(t,x)\,dS + \int_{\Gamma_2} p(t,x)\,dS = M_0>0,\qquad \forall t>0 \label{cons}
         \end{equation}
         where $M_0$ is the initial mass, which is assumed to be positive,
         \begin{equation*}
         M_0 := \int_{\Omega} (L_0(x)+P_0(x)) \, dx + \int_{\Gamma} l_0(x)\,dS + \int_{\Gamma_2} p_0(x)\,dS>0.
         \end{equation*}
         \medskip
         
         Linear VSRDs like \eqref{f1}--\eqref{f3} appear recently in many related cell-biological models, such as signalling models, see e.g. \cite{FNR}, models for transcription and translation of genes, see e.g. \cite{TVO}, 
         models of the reversible transitions between various conformational states of proteins, see e.g. \cite{Mayor03}, models of the transitions between various folding states of RNA, see e.g. \cite{Bok03}, and models of "field-road" coupling in ecology, see e.g. \cite{B1, B2, B3}.
         \medskip
         
The {\bf content of this paper} is the following: In Section \ref{sec:2}, we first study the well-posedness of the volume-surface reaction-diffusion (VSRD for short) system \eqref{f1}--\eqref{f3} and the regularity of solutions. The well-posedness will be shown by reformulating the system in a variational form where the bilinear form is not coercive but satisfies instead a G{\aa}rdingÕs inequality. Global existence of solutions follows then from the conservation of mass and a suitable $L^2$ functional.          
%         We remark, that although the system is linear, the existence and uniqueness of global solution is not trivial due to i) the weakly-reversible system character, ii) the mixed boundary conditions and iii) the presence of surface diffusion. 
         { Previous related results of VSRD systems can be found, for instance, in \cite{ER, NGCRSS} for linear models and \cite{B1,B2,B3, BFL} for nonlinear models.}
\medskip
         
The next major part of this paper is the quasi-steady-state approximation (QSSA for short) for \eqref{f1}--\eqref{f3} which will be done in Section 3. The biological data for the model system \eqref{f1}--\eqref{f3} suggests that the release rate $\xi$ from cortical $p$-Lgl to cytoplasmic $P$-Lgl is much larger than the others reaction rates. We shall thus prove rigorously the QSSA for  \eqref{f1}--\eqref{f3} towards a reduced QSSA system (see Section \ref{sec:4}), which occurs when considering the limit $\xi\to\infty$ for the release rate of cortical phosphorylated $p$-Lgl. The QSSA thus leads to a reduced system, where the phosphorylation of $l$-Lgl at the cortex yields directly the inflow for $P$-Lgl and the cortical concentration $p$-Lgl no longer needs to be considered.

QSSAs of reactive systems occur commonly in chemical engineering and although applying QSSAs has been routinely done by chemical engineers since a long time, the mathematical theory is usually missing. 
 Recent however, a lot of mathematical attention has been paid to rigorously prove QSSAs (see e.g. \cite{BCD, Bothe1, BH, BP, BP1, CD} and references therein).
         
In the present paper, the novelty lies in the surface-volume coupling of the QSSA since the limiting parameter $\xi\to\infty$ appears as inhomogeneity in the Neumann boundary condition for $P$-Lgl, which introduces new technical difficulties and requires appropriate \emph{a priori} estimates in order to deal with the reactions connecting volume and surface.

Here, we are able present a first result concerning such a QSSA of a VSRD system in the case that no surface diffusion terms are present. The proof is based on a duality argument, which was already successfully applied to a nonlinear reaction-diffusion system in \cite{BP}, yet without surface-volume coupling. We remark that the QSSA of the weakly-reversible model system \eqref{f1}--\eqref{f3} with surface diffusion poses technique difficulties, which we were unable to overcome so far. The QSSA of a simpler, detailed balance VSRD system was recently studied in \cite{HT}.         
         
\medskip

In Section \ref{sec:3}, we present a suitable finite element method (FEM) discretisation of the model system \eqref{f1}--\eqref{f3} and discuss some numerical test cases, which illustrate the complexity of this four species system. The chosen examples present particular interesting aspects of the interplay between volume and surface diffusion and the reactions connecting cytoplasm and cortex.                

{ The numerical simulations demonstrate the remarkable influence of surface diffusion and the complex balance reaction kinetics to the behaviour of the system. On one hand, surface diffusion clearly helps to smooth jumps in of concentrations on the boundary. More interestingly, the presence/absence of surface diffusion strongly affects the shape of the attained complex balance equilibrium (see e.g. Fig \ref{figLP} and it's discussion for details). The numerical simulation also illustrates the asymptotic behaviour of the system as $\xi \rightarrow +\infty$, which confirms the theoretical QSSA done in Section \ref{sec:4}.

For relating results on the role of surface diffusion, we refer the interested reader to recent works of Berestycki {\it et al.} \cite{B1,B2,B3}, where the authors studied the influence of surface diffusion to the travelling wave speed in the domain for Fisher-KPP models.
}

Finally, in Section \ref{sec:conc}, we summarise the conclusions of the paper, while the Appendix \ref{sec:app} provides additional details.

         \section{Well-posedness of the system \eqref{f1}--\eqref{f3}}\label{sec:2}
    
         In this section, we first prove the existence of a unique weak solution to \eqref{f1}--\eqref{f3} by using abstract results for linear PDE systems. Since the bilinear form is not coercive but only satisfies a G\r{a}rding's inequality, the uniform bounds in time of solutions do not follow immediately. We will show that if the initial data is nonnegative, and consequently the solution is nonnegative, then the solution to \eqref{f1}--\eqref{f2} is bounded uniformly in time. Some regularities of the solution are studied at the end of the section.
         
         \noindent\textbf{Notations:}  
         Throughout the paper, we will denote by $(\cdot, \cdot)_{\Omega}$ and $\|\cdot\|_{\Omega}$ the inner product and its induced norm in $L^2(\Omega)$. Analogously, we will denote the inner products and norms in $L^2(\Gamma), L^2(\Gamma_1)$ and $L^2(\Gamma_2)$ (e.g. the norm in $L^2(\Gamma)$ is denoted by $\|\cdot\|_{\Gamma}$). The tangential derivatives on $\Gamma$ and $\Gamma_2$ are denoted by $\nabla_{\Gamma}$ and $\nabla_{\Gamma_2}$, respectively (see e.g. \cite[Chapter 16]{GT}.
         
We will denote by $C$ a generic constant, which only depends on the initial data, all diffusion and reaction rates. 
Moreover, $C_T$ denotes a generic constant, which additionally depends on the time interval size $T>0$. 
 
For any given $T>0$ and $q\geq 1$, we shall denote 
         $$
         \Omega_T:= [0,T]\times \Omega, \qquad \Gamma_T:=[0,T]\times \Gamma, \qquad 
         \Gamma_{2,T}:=[0,T]\times \Gamma_2.
         $$ 
         The spaces $L^q(\Omega_T)$, $L^q(\Gamma_T)$ or $L^q(\Gamma_{2,T})$ will be used with the usual norms, for example
         $$
         \|f\|_{L^2(\Omega_T)} = \biggl(\int_{0}^T\|f(t)\|_{L^2(\Omega)}^2dt\biggr)^{1/2}.
         $$
         
%         Moreover, for a function $u\in C([0,T]; L^p(\Omega))$, we shall use the notation $u\geq 0$ to denote that $u(t,x)\geq 0$ for all $t\in [0,T]$ and for a.e. $x\in\Omega$.
         
         \begin{definition}[Weak Solutions]\label{def:WeakSolution}\hfill\\
         	A quadruple $(L, P, l, p)$ is called a weak solution to system \eqref{f1}--\eqref{f3} on $(0, T)$ if
         	\begin{align*}
         	L, P & \in C([0,T]; L^2(\Omega))\cap L^2(0,T; H^1(\Omega)),\\
         	l & \in C([0,T]; L^2(\Gamma))\cap L^2(0,T; H^1(\Gamma)),\\ 
         	p & \in C([0,T]; L^2(\Gamma_2))\cap L^2(0,T; H^1(\Gamma_2)),
         	\end{align*}
         	and
         	\begin{equation*}
	         	L(x,0) = L_0(x), \qquad P(x, 0) = P_0(x), \qquad \ell(x,0) = \ell_0(x), \qquad p(x,0) = p_0(x),
         	\end{equation*}
         	and for all test functions $\varphi_1,\varphi_2\in H^1(\Omega)$, $\psi\in H^1(\Gamma)$ and $\psi_2\in H^1(\Gamma_2)$
         	we have
         	\begin{equation}
	         	\begin{aligned}
		         	\frac{d}{dt}\bigl[(L, \varphi_1)_{\Omega} + (P, \varphi_2)_{\Omega} + (\ell, \psi_1)_{\Gamma} + (p,\psi_2)_{\Gamma_2}\bigr]+ a(L, P, \ell, p; \varphi_1,\varphi_2, \psi_1, \psi_2) = 0
	         	\end{aligned}
         	\end{equation}
         	for a.e. $0<t<T$, where the bilinear form $a$ is defined as
         	\begin{equation}\label{bilinear}
	         	\begin{aligned}
		         	&a(L, P, \ell, p; \varphi_1, \varphi_2, \psi_1, \psi_2)\\
		         	&\qquad= d_L(\nabla L, \nabla \varphi_1)_{\Omega} + d_P(\nabla P, \nabla\varphi_2)_{\Omega} + d_{\ell}(\nabla_{\Gamma}\ell, \nabla_{\Gamma}\psi_1)_{\Gamma} + d_p(\nabla_{\Gamma_2}p, \nabla_{\Gamma_2}\psi_2)_{\Gamma_2}\\
		         	&\qquad\qquad + (\beta L - \alpha P, \varphi_1 - \varphi_2)_{\Omega} + (\lambda L - \gamma \ell, \varphi_1 - \psi_1)_{\Gamma}\\
		         	&\qquad\qquad\qquad+ (\sigma\ell, \psi_1 - \psi_2)_{\Gamma_2} + (\xi p, \psi_2 - \varphi_2)_{\Gamma_2}.
	         	\end{aligned}
         	\end{equation}
         \end{definition}
         
         \begin{lemma}[G{\aa}rding Inequality]\label{lem:Garding}\hfill\\
         	The bilinear form $a$ defined in \ref{bilinear} is continuous in $H^1(\Omega)\times H^1(\Omega)\times H^1(\Gamma)\times H^1(\Gamma_2)$ and satisfies a G{\aa}rding inequality, i.e. there exists $\delta_1>0$ and $\delta_2>0$ such that
         	\begin{multline*}
 %        	\begin{gathered}
         	a(L, P, \ell, p; L, P, \ell, p) + \delta_1(\|L\|_{\Omega}^2 + \|P\|_{\Omega}^2 + \|\ell\|_{\Gamma}^2 + \|p\|_{\Gamma_2}^2)\\
         	\geq \delta_2(\|L\|_{H^1(\Omega)}^2 + \|P\|_{H^1(\Omega)}^2 + \|\ell\|_{H^1(\Gamma)}^2 + \|p\|_{H^1(\Gamma_2)}^2)
   %      	\end{gathered}
   \end{multline*}
for all $(L,P,\ell,p)\in H^1(\Omega)\times H^1(\Omega)\times H^1(\Gamma)\times H^1(\Gamma_2)$.
         \end{lemma}
         \begin{proof}
         	The continuity of $a$ is standard so we omit it here. To prove the G\r{a}rding inequality we use \cite[Theorem 1.5.1.10]{Grisvard} that for any $\varepsilon>0$, there exists $C_{\varepsilon}>0$ such that
         	\begin{equation*}
         	\|f\|_{\Gamma}^2 \leq \varepsilon\|\nabla f\|_{\Omega}^2 + C_{\varepsilon}\|f\|_{\Omega}^2 \qquad \text{ for all } \quad f\in H^1(\Omega).
         	\end{equation*}
         	Hence, we can estimate
         	\begin{equation*}
         	\begin{aligned}
         	a(L,P,\ell,p; L, P,\ell, p)&\geq d_L\|\nabla L\|_{\Omega}^2 + d_P\|\nabla P\|_{\Omega}^2 + d_{\ell}\|\nabla_{\Gamma}\ell\|_{\Gamma}^2 + d_p\|\nabla_{\Gamma_2}p\|_{\Gamma_2}^2\\
         	&\qquad - (\alpha+\beta)(L, P)_{\Omega} - (\lambda + \gamma)(L,\ell)_{\Gamma} - \sigma(\ell,p)_{\Gamma_2} - \xi(P,p)_{\Gamma_2}\\
         	&\geq \frac{1}{2}\left(d_L\|\nabla L\|_{\Omega}^2 + d_P\|\nabla P\|_{\Omega}^2 + d_{\ell}\|\nabla_{\Gamma}\ell\|_{\Gamma}^2 + d_p\|\nabla_{\Gamma_2}p\|_{\Gamma_2}^2\right)\\
         	&\qquad - C(\|L\|_{\Omega}^2 + \|P\|_{\Omega}^2 + \|\ell\|_{\Gamma}^2 + \|p\|_{\Gamma_2}^2).
         	\end{aligned}
         	\end{equation*}
         \end{proof}
         % % % % % % % % % % % % % %
         % % % % % % % % % % % % % %
         % % % % % % % % % % % % % %
         % % % % % % % % % % % % % %
         % % % % % % % % % % % % % %
         % % % % % % % % % % % % % %
         
         Basing on Lemma \ref{lem:Garding}, the existence of unique weak solution to \eqref{f1}--\eqref{f3} follows from standard theory of linear parabolic equations, see e.g. \cite[XVIII \S3)]{DL92}.
         
         \medskip

		\begin{theorem}[Existence of a Unique Weak Global Solution]\label{thm:existence} \hfill\\
			For any $(L_0,P_0,\ell_0,p_0)\in L^2(\Omega)\times L^2(\Omega)\times L^2(\Gamma)\times L^2(\Gamma_2)$, the system \eqref{f1}--\eqref{f3} has a unique weak solution on $(0,T)$ for all $T>0$ which conserves the total mass, that is
	         \begin{multline*}
		        % \begin{gathered}
		         \int_{\Omega}L(x,t)dx + \int_{\Omega}P(x,t)dx + \int_{\Gamma}\ell(x,t)dS + \int_{\Gamma_2}p(x,t)dS\\
		         = \int_{\Omega}L_0(x)dx + \int_{\Omega}P_0(x)dx + \int_{\Gamma}\ell_0(x)dS + \int_{\Gamma_2}p_0(x)dS.
		         %\end{gathered}
	         \end{multline*}
			Moreover, if the initial data is nonnegative then the solution is nonnegative for all time $t>0$.
         \end{theorem}
         
         \medskip
         
         The uniform global bounds of the solution in $L^2$ for the system \eqref{f1}--\eqref{f3} does not follows immediately from Theorem \ref{thm:existence} since the bilinear $a$ is not coercive. In the sequel, we will show that such $L^2$-bounds can be obtained with the help of the mass conservation in Theorem \ref{thm:existence}. 

         \medskip
         
         \begin{theorem}[{Global $L^2$-bounds of Weak Solutions}]\label{theo:GlobalExistence}\hfill\\
         Assume that the initial data $(L_0,P_0,l_0,p_0)\ge0$ are non-negative. 
         Then, there exists a constant $C$, which depends only on the domain, the initial data, the reaction rates and the diffusion rates such that
         \begin{equation*}
         \forall t\ge 0:\qquad \|L(t)\|_{\Omega}^2 + \|P(t)\|_{\Omega}^2 + \|l(t)\|_{\Gamma}^2 + \|p(t)\|_{\Gamma_2}^2 \leq C,
         \end{equation*}
         i.e. the global solutions to system \eqref{f1}--\eqref{f3} are  bounded uniformly-in-time. 
         \end{theorem}
         \begin{proof}
%         We shall consider the case of non-negative initial data (and thus solutions) and prove uniform boundedness in time. The proof the existence of global solutions to general $L^2$ initial data follows by a similar argument when replacing the below $L^1$-interpolation and $L^1$-bound \eqref{globalL1} by analog estimates in $L^2$ and a Gronwall argument, which yields global existence yet without uniform-in-time boundedness.            
         Define a quadratic functional
         \begin{equation*}
         \mathcal H(t) = \frac{1}{2}\Bigl(\|L(t)\|_{\Omega}^2 + \|P(t)\|_{\Omega}^2 + \sigma\|l(t)\|_{\Gamma}^2 + \xi\|p(t)\|_{\Omega}^2\Bigr).
         \end{equation*}
         By calculating the time derivative of $\mathcal H$ along solutions of system \eqref{f1}--\eqref{f3}, we get
         \begin{equation}
         \begin{aligned}
         \frac{d\mathcal H}{dt} =& -d_L\|\nabla L\|_{\Omega}^2 - d_P\|\nabla P\|_{\Omega}^2 - \beta\|L\|_{\Omega}^2 - \alpha\|P\|_{\Omega}^2 - \lambda\|L\|_{\Gamma}^2\\
         &+(\alpha+\beta)(L,P)_{\Omega} + (\lambda\sigma + \gamma)(L, l)_{\Gamma} + \xi(P,p)_{\Gamma_2}\\
         &- \gamma\sigma\|l\|_{\Gamma}^2 - \sigma^2\|l\|_{\Gamma_2}^2 - \xi^2\|p\|_{\Gamma_2}^2 + \sigma\xi(p,l)_{\Gamma_2}.
         \end{aligned}
         \label{hh1}
         \end{equation}
         We recall two well-known interpolation estimates (see e.g. \cite[Theorem 1.3, Page 18]{BaVi} and \cite[Theorem 1.5.1.10]{Grisvard}): there exists for any $\varepsilon >0$ a constant $C_{\varepsilon}>0$ such that
         \begin{equation}\label{interp1}
	               \|u\|_{\Omega}^2 \leq \varepsilon\|\nabla u\|_{\Omega}^2 + C_\varepsilon\|u\|_{L^1(\Omega)}^2,
                  %\end{equation*}
                  \qquad\text{and}\qquad
                  %\begin{equation*}
                  \|u\|_{\Gamma}^2 \leq \varepsilon\|\nabla u\|_{\Omega}^2 + C_\varepsilon\|u\|_{L^1(\Omega)}^2,
         \end{equation}
         for all $u\in H^1(\Omega)$. Now, by using the Cauchy inequality and \eqref{interp1}, after standard computations we get from \eqref{hh1} that     
         \begin{equation*}
         \begin{aligned}
         \frac{d\mathcal H}{dt}\leq &-\frac{d_L}{2}\|\nabla L\|_{\Omega}^2 - \frac{d_P}{2}\|\nabla P\|_{\Omega}^2 - \beta\|L\|_{\Omega}^2 - \alpha\|P\|_{\Omega}^2
         - \frac{\gamma\sigma}{2}\|l\|_{\Gamma}^2 - \frac{\xi^2}{2}\|p\|_{\Gamma_2}^2\\ &+ C(\|L\|_{L^1(\Omega)}^2 + \|P\|_{L^1(\Omega)}^2).
         \end{aligned}
         %\label{hh6}
         \end{equation*}
         By defining $\eta := \frac{1}{2}\min\{2\beta, 2\alpha, \gamma, \xi\}$ and by using $\|L\|_{L^1(\Omega)}\leq M$ and $\|P\|_{L^1(\Omega)}\leq M$ (thanks to the mass conservation in Theorem \ref{thm:existence}), we have
         \begin{equation}
         \frac{d\mathcal H}{dt} + \eta \mathcal H \leq C.
%         \label{hh7}
         \end{equation}
         By Gronwall's inequality, we obtain in particular 
%         \begin{equation*}
$         \mathcal H(t)\leq e^{-\eta t}\mathcal H(0) + C$,
%         %\label{hh8}
%         \end{equation*}
%         or equivalently,
%         \begin{multline*}
%         \|L(t)\|_{\Omega}^2 + \|P(t)\|_{\Omega}^2 + \sigma\|l(t)\|_{\Gamma}^2 + \xi\|p(t)\|_{\Gamma_2}^2\\
%         \leq e^{-\eta t}\left(\|L_0\|_{\Omega}^2 + \|P_0\|_{\Omega}^2 + \sigma\|l_0\|_{\Gamma}^2 + \xi\|p_0\|_{\Gamma_2}^2\right) + C,
%         %\label{hh9}
%         \end{multline*}
         which completes the proof.
         \end{proof}
 \medskip         

         To study the regularity of solutions to \eqref{f1}--\eqref{f3}, we need the following Lemma:
         \medskip
         
         \begin{lemma}[Regularity of Parabolic Equations, \cite{Evans,KC06}]\label{regularity}\hfill \\
         		Assume that $\Omega\subset \mathbb{R}^n$ is a bounded domain with regular boundary $\Gamma:= \partial\Omega$ (says $\Gamma \in C^{2+\alpha}, \alpha>0$). Consider the following initial boundary value problem
         		\begin{equation}\label{heateq}
		         		\begin{cases}
				         		u_t - d\Delta u = f, &\quad x\in\Omega, \qquad t>0,\\
				         		d\partial_{\nu}u +\alpha u= g, &\quad x\in\Gamma, \qquad t>0,\\
				         		u(0,x) = u_0(x), &\quad x\in\Omega.
		         		\end{cases}
         		\end{equation}
         		Assume that $\alpha \geq 0$, $u_0\in L^2(\Omega)$, $g\in L^2(0,T;H^{\theta-1/2}(\Gamma))\cap H^{\theta}(0,T;H^{-1/2}(\Gamma))$ and $f\in L^2(0,T;(H^{1-\theta}(\Omega))^*)$ for some $\theta\in [0,1]$. 
Then, \eqref{heateq} has a unique weak solution satisfying, for any $0<\tau<T<+\infty$,
         		\begin{equation}\label{SecondReg}
         	u\in L^2(\tau,T;H^{1+\theta}(\Omega)) \cap H^{1}(\tau,T;(H^{1-\theta}(\Omega))^*).
         		\end{equation}
         \end{lemma}
%         \begin{proof}
%For $\theta=0$, i.e. $u_0\in L^2(\Omega)$, the existence of a unique weak solution $u\in C([0,T];L^2(\Omega))\cap L^2(0,T;H^1(\Omega))$ of \eqref{heateq} follows
%from \cite[\S 7.1.2]{Evans}.
%         		
%         		Let $0<\tau<T$ and $\varphi: [0,+\infty) \rightarrow [0,1]$ be a smooth function satisfying $\varphi \equiv 0$ on $[0,\tau/2]$ and $\varphi\equiv 1$ on $[\tau,+\infty)$. Set $v(t,x) = u(t,x)\varphi(t)$, then $v$ is the solution to
%         		\begin{equation*}
%		         		\begin{cases}
%				         		v_t - d\Delta v = f\varphi + u\varphi_t =: \tilde{f}, &x\in\Omega, \quad t>0,\\
%				         		d\partial_{\nu}v + \alpha v = g\varphi =: \tilde{g}, &x\in\partial\Omega, \quad t>0,\\
%				         		v(0,x) = 0, &x\in\Omega.
%		         		\end{cases}
%         		\end{equation*}
%        Since $\varphi$ is smooth and $u\in L^2(0,T;H^1(\Omega))$, $\tilde{f}$ and $\tilde{g}$ has the same regularity as $f$ and $g$. Note that the initial data of $v$ now is zero. Then, it follows from \cite[Theorem 3.3]{KC06} that
%         		$$
%        v\in L^2(0,T;H^{1+\theta}(\Omega)) \cap H^{1}(0,T;(H^{1-\theta}(\Omega))^*).
%         		$$
%         		Thus, \eqref{SecondReg} follows thanks to $u \equiv v$ on $[\tau,+\infty)$.
%         \end{proof}
         \medskip
         
         \begin{theorem}[Regularity of Solutions]\label{Reg}
                  For any $I = [\tau, T]$ with $0<\tau<T<+\infty$, the weak solution $(L, P, l, p)$ to \eqref{f1}--\eqref{f3} satisfies the following regularities
\begin{align*}
         L &\in L^2(I;H^2(\Omega)) \cap H^{1}(I;L^2(\Omega)),\\
         P &\in L^2(I;H^{3/2}(\Omega)) \cap H^{1}(I;H^{-1/2}(\Omega)),\\
        	l &\in L^2(I;H^2(\Gamma)) \cap H^{1}(I;L^2(\Gamma)),\\
%  \intertext{and}
           	p &\in L^2(I;H^2(\Gamma_2)) \cap H^{1}(I;L^2(\Gamma_2)).
\end{align*}
                  %Thus, $(L,P,l,p)$ is in fact a strong solution, which satisfies the system \eqref{f1}--\eqref{f3} in $L^2$ for almost every $t\in I$.
         \end{theorem}
         \begin{proof}
First, by considering the equation for $l$,
					\begin{equation*}
									l_t - d_l\Delta_{\Gamma}l = \lambda L - (\gamma + \chi_{\Gamma_2}\sigma)l \;\in \;L^2(0,T;L^2(\Gamma)), \qquad l(0) = l_0 \in L^2(\Gamma),
					\end{equation*}
we observe that $l\in L^2(I;H^2(\Gamma))\cap H^{1}(I;L^2(\Gamma))$ thanks to the regularity of solutions to heat equation on smooth Riemann manifolds $\Gamma$ without boundary (see, e.g. \cite[Chapter 6]{Taylorbook}). A similar result for heat equation on smooth Riemann manifold $\Gamma_2$ with smooth boundary $\partial\Gamma_2$ applies to
					\begin{equation*}
									p_t - d_p\Delta_{\Gamma_2}p = \sigma l - \xi p \; \in \; L^2(0,T;L^2(\Gamma_2)), \qquad 
									d_p\partial p/\partial \nu_{\Gamma_2} = 0, \qquad p(0) = p_0\in L^2(\Gamma_2),
					\end{equation*}
which yields $p\in L^2(I;H^2(\Gamma_2))\cap H^{1}(0,T;L^2(\Gamma_2))$. Turn to the equation for $L$,
					\begin{equation*}
							\begin{gathered}
								L_t - d_L\Delta L = -\beta L + \alpha P \in L^2(0,T;L^2(\Omega))\\
								d_L\partial L/\partial \nu + \lambda L = \gamma l\in L^2(I;H^{1/2}(\Gamma))\cap H^{1}(I;H^{-1/2}(\Gamma)),  \qquad
								L(0) = L_0\in L^2(\Omega),
							\end{gathered}
					\end{equation*}
%we observe that $L_0\in L^2(\Omega)$, $-\beta L + \alpha P \in $ and $\gamma l \in $ due to (the even better) regularity of $l$ stated above. 
Then, we can apply Lemma \ref{regularity} with $\theta = 1$ to obtain $L\in L^2(I;H^2(\Omega))\cap H^{1}(I;L^2(\Omega))$.
					
 It remains to show the regularity of $P$ which is the solution to
					\begin{equation*}
									P_t - d_P\Delta P = \beta L - \alpha P \in L^2(0,T;L^2(\Omega)),\qquad 
									d_P\partial P/\partial\nu = \chi_{\Gamma_2}\xi p,\qquad 
									P(0) = P_0\in L^2(\Omega).
					\end{equation*}
We cannot expect $P$ to have the same regularity as $L$ because of the low regularity of the characteristic function $\chi_{\Gamma_2}$ (e.g. $\chi_{\Gamma_2}\not\in H^{1/2}(\Gamma)$). However, since $\chi_{\Gamma_2}\in L^{\infty}(\Gamma)$ and, thus,  
$\chi_{\Gamma_2}\xi p\in L^2(0,T;L^2(\Gamma))$, we can apply Lemma \ref{regularity} with $\theta = 0$ to conlude
					\begin{equation*}
							P \in L^2(I;H^{3/2}(\Omega))\cap W^{1,2}(I;H^{-1/2}(\Omega)).
					\end{equation*}
         \end{proof} 
         \medskip
         
         \begin{remark} (Further Regularity of Solutions)
We remark that the relatively lower regularity of $P$ in the Theorem \ref{Reg} 
stems exclusively from the low regularity of the characteristic function 
$\chi_{\Gamma_2}$. It is thus a mathematical consequence of the modelling choice 
of a bounded cut-off function $\chi_{\Gamma_2} \in L^{\infty}$ for the activity range of aPKC. 
%which is the consequence of artificial modeling when we "cut off" the activated part of the boundary by the characteristic function of $\Gamma_2$. 
 
It seems biologically equally justified to replace $\chi_{\Gamma_2}$ by a smoothed cut-off function $\chi_{\Gamma_2}^{\epsilon}:\Gamma_2 \rightarrow [0,1]$, for sufficiently small $\epsilon>0$, satisfying that $\chi_{\Gamma_2}^{\epsilon}$ vanished on $\partial\Gamma_2$.
Then, we obtain the full regularity $P\in L^2(I;H^2(\Omega))\cap H^{1}(I;L^2(\Omega))$. In this case, we can further bootstrap to obtain arbitrary high Sobolev regularity of $l$, $p$, $L$ and $P$ provided sufficient regularity of the boundaries $\Gamma$ and $\partial\Gamma_2$.

\end{remark}

         \section{Quasi-Steady-State Approximation}\label{sec:4}
         In this section, we study the Quasi-Steady-State Approximation (QSSA) for the system \eqref{f1}--\eqref{f3} as $\xi \rightarrow +\infty$. The limit $\xi \rightarrow +\infty$ can be interpreted as the instantaneous release of phosphorylated Lgl from the cell cortex into the cell cytoplasm. For technical reasons (see Lemma \ref{lem1} and Remark \ref{remark:QSSAsurfacediffusion}), we shall restrict our analysis to the case without boundary diffusion, i.e. $d_l=0=d_p$.  The QSSA for system \eqref{f1}--\eqref{f3} with surface diffusion constitutes currently an open problem.

         Quasi-Steady-State Approximations for (bio)chemical reaction systems have long been studied in terms of asymptotic expansions, but it was not until recently that rigorous results were obtained for the corresponding fast-reaction limits 
         (see e.g. \cite{Bothe1, BH, BP, BP1, BCD, CD} and references therein). 
         
         %Roughly speaking, QSSA means that you can simplify a system be neglecting (in a certain sense) some substances if their rates of reaction are much more bigger compare to the others. In the case of this work, we consider system \eqref{f1}-\eqref{f3} as $\xi \rightarrow +\infty$, which means we consider the fast-reaction limit of $p \xrightarrow{\xi} P$ on the boundary $\Gamma_2$. This limit is motivated by the biologists which expect that the release of phosphorylated Lgl $p$ from the cortex to cytoplasm to happen very fast.

         \begin{figure}[htp]
         \begin{center}\scalebox{1}[1]{
         \begin{tikzpicture}
         \node (a) {$L$} 	node (b) at (2,0) {$P$} 	node (c) at (0,-2) {$\ell$} 	node (d) at (2,-2)  {$p$} 		node (aa) at (7,0) {$L$} 	node (bb) at (9,0) {$P$} 	node (cc) at (7,-2) {$\ell$};
         \draw[arrows=->] ([xshift =0.5mm,yshift=0.5mm]a.east) -- node [below] {\scalebox{.8}[.8]{$\alpha$}} ([xshift =-0.5mm,yshift=0.5mm]b.west);
         \draw[arrows=->] ([xshift =-0.5mm,yshift=-0.5mm]b.west) -- node [above] {\scalebox{.8}[.8]{$\beta$}} ([xshift =0.5mm,yshift=-0.5mm]a.east);
         \draw[arrows=->] ([xshift =0.5mm,yshift=0.5mm]c.north) -- node [right] {\scalebox{.8}[.8]{$\gamma$}} ([xshift =0.5mm,yshift=-0.5mm]a.south);
         \draw[arrows=->] ([xshift =-0.5mm,yshift=-0.5mm]a.south) -- node [left] {\scalebox{.8}[.8]{$\lambda$}} ([xshift =-0.5mm,yshift=0.5mm]c.north);
         \draw[arrows=->] ([xshift =0.5mm]c.east) -- node [below] {\scalebox{.8}[.8]{$\sigma$}} ([xshift=-0.5mm]d.west);
         \draw[arrows=->] ([yshift =0.5mm]d.north) -- node [right] {\scalebox{.8}[.8]{$\xi$}} ([yshift=-0.5mm]b.south);
         \draw[arrows=->] ([xshift =0.5mm,yshift=0.5mm]aa.east) -- node [below] {\scalebox{.8}[.8]{$\alpha$}} ([xshift =-0.5mm,yshift=0.5mm]bb.west);
         \draw[arrows=->] ([xshift =-0.5mm,yshift=-0.5mm]bb.west) -- node [above] {\scalebox{.8}[.8]{$\beta$}} ([xshift =0.5mm,yshift=-0.5mm]aa.east);
         \draw[arrows=->] ([xshift =0.5mm,yshift=0.5mm]cc.north) -- node [right] {\scalebox{.8}[.8]{$\gamma$}} ([xshift =0.5mm,yshift=-0.5mm]aa.south);
         \draw[arrows=->] ([xshift =-0.5mm,yshift=-0.5mm]aa.south) -- node [left] {\scalebox{.8}[.8]{$\lambda$}} ([xshift =-0.5mm,yshift=0.5mm]cc.north);
         \draw[arrows=->] ([xshift =0.5mm]cc.east) -- node [below] {\scalebox{.8}[.8]{$\sigma$}} ([xshift=-0.5mm]bb.south);
         \draw[arrows=->] ([xshift =6mm, yshift=-1cm]b.east) -- node [above] {\scalebox{.8}[.8]{$\xi \rightarrow +\infty$}} ([xshift =-6mm, yshift=-1cm]aa.west);
         \end{tikzpicture}}
         \end{center}
         \caption{The original Lgl model system (left) and the reduced QSSA system (right). The reduced QSSA model, like the full model, is still a weakly-reversible system.}
         \label{fig:QSSA}
         \end{figure}
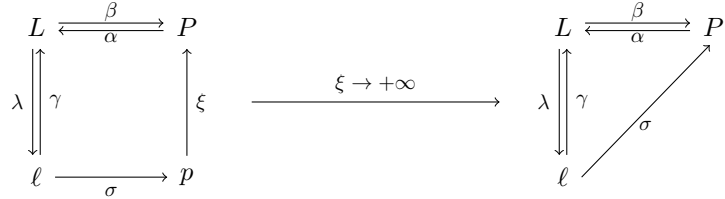
         
         The QSSA we shall study in this Section can be illustrated as the passage of the left reaction
         diagram towards the right reaction diagram in Figure \ref{fig:QSSA}.
         Without the surface diffusion terms, the system \eqref{f1}--\eqref{f3} rewrites as
         \begin{equation}
         \begin{cases}
         L_t- d_L\Delta L = -\beta L + \alpha P, &\qquad x\in\Omega,\quad t> 0,\\
         d_L\frac{\partial L}{\partial \nu} + \lambda L = \gamma l, &\qquad x\in\Gamma,\quad t> 0,\\
         L(0,x) = L_0(x), &\qquad x\in\Omega,
         \end{cases}
         \label{newf1}
         \end{equation}
         \begin{equation}
         \begin{cases}
         P_t - d_P\Delta P = \beta L - \alpha P, &\qquad x\in\Omega,\quad t> 0,\\
         d_P\frac{\partial P}{\partial \nu} = \chi_{\Gamma_2}\xi p, &\qquad x\in\Gamma,\quad t> 0,\\
         P(0,x) = P_0(x), &\qquad x\in\Omega,
         \end{cases}
         \label{newf2}
         \end{equation}
         %and
         \begin{equation}
         \begin{cases}
         l_t  = \lambda L - (\gamma + \chi_{\Gamma_2}\sigma)l, &\qquad x\in\Gamma,\quad t> 0,\\
         l(0,x) = l_0(x), &\qquad x\in\Gamma,\\
         p_t = \sigma l - \xi p, &\qquad x\in\Gamma_2,\quad t> 0,\\
         p(0,x) = p_0(x), &\qquad x\in\Gamma_2.
         \end{cases}
         \label{newf3}
         \end{equation}
         
         Intuitively and according to the numerical example Fig. \ref{fig3a}, we expect from the second equation in \eqref{newf3} that in the limit $\xi \to +\infty$ the concentration $p(t,x)$ of phosphorylated Lgl on the boundary $\Gamma_2$ tends to zero for any positive time since all the $p$-Lgl on the active part of the cell cortex part is instantaneously released into the cytoplasm and becomes $P$-Lgl. 
         
         However, if the initial $p$-Lgl concentration is non-zero, i.e. $p_0(x)\neq 0$, an initial layer at $t=0$ will be forming in the limit $\xi\to +\infty$, which expresses the transfer of initial mass of $p_0$ into $P_0$ (see also Figure \ref{InitialMass} for a numerical  example). 
         
         Thus, the expected limiting system has the following form:
         \begin{equation}
         \begin{cases}
         L_t - d_L\Delta L = -\beta L + \alpha P,&\qquad x\in\Omega,\quad t>0,\\
         d_L\frac{\partial L}{\partial \nu} = -\lambda L + \gamma l,&\qquad x \in\Gamma,\quad t>0,\\
         L(0,x) = L_0(x), &\qquad x\in\Omega,
         \end{cases}
         \label{ff1}
         \end{equation}
         \begin{equation}
         \begin{cases}
         P_t - d_P\Delta P = \beta L - \alpha P, &\qquad x\in\Omega,\quad t>0,\\
         d_P\frac{\partial P}{\partial \nu} = \chi_{\Gamma_2}\sigma l, &\qquad x\in\Gamma,\quad t>0,\\
         P(0,x) = P_0(x) + P^*(x), &\qquad x\in\Omega,
         \end{cases}
         \label{ff2}
         \end{equation}
         and
         \begin{equation}
         \begin{cases}
         l_t = \lambda L - (\gamma + \sigma\chi_{\Gamma_2})l, &\qquad x\in\Gamma,\quad t>0,\\
         l(0,x) = l_0(x), &\qquad x\in\Gamma,
         \end{cases}
         \label{ff3}
         \end{equation}
         where we emphasise that $P^*$ is the unique function in $L^2(\Omega)$, which satisfies
         \begin{equation}\label{ff3a}
         %\boxed{
         {\int_{\Omega}P^*\varphi \, dx = \int_{\Gamma_2}p_0\varphi \,dS, \qquad  \forall\varphi\in H^1(\Omega).}
         %}
         \end{equation}
         Note that the system \eqref{ff1}--\eqref{ff3} corresponds to the reaction dynamics represented by the right diagram in Figure \ref{fig:QSSA}. Its reaction kinetics is still weakly-reversible. 
         
         \medskip
         
         \begin{remark}[Common Stationary States of Full System and QSSA]\label{rem41}\hfill\\
         We point out that the system \eqref{newf1}--\eqref{newf3} and the QSSA system \eqref{ff1}--\eqref{ff3} with condition \eqref{ff3a} share the same stationary state $(L_{\infty}, P_{\infty}, l_{\infty})$. This is a consequence of the fact that the systems  \eqref{newf1}--\eqref{newf3} and \eqref{ff1}--\eqref{ff3} satisfying the same stationary state system (see \eqref{SS} below) and that 
         condition \eqref{ff3a} ensure identical initial total mass. 
         
         Indeed, it follows from \eqref{newf3} that $\xi p_{\infty} = \sigma l_{\infty}$ and $\lambda L_{\infty} = (\gamma + \sigma\chi_{\Gamma_2})l_{\infty}$. Inserting these two relations into  \eqref{newf1}--\eqref{newf2} yields the stationary state system
         \begin{equation}
         \begin{cases}
         -d_L\Delta L_{\infty} = -\beta L_{\infty} + \alpha P_{\infty},&\qquad x\in\Omega,\\
         - d_P\Delta P_{\infty} = \beta L_{\infty} - \alpha P_{\infty}, &\qquad x\in\Omega,\\
         d_L\frac{\partial L_{\infty}}{\partial \nu} = 
         - \frac{\sigma \lambda}{\gamma + \sigma} L_{\infty},&\qquad x \in\Gamma_2,\\
         d_L\frac{\partial L_{\infty}}{\partial \nu} = 0,&\qquad x \in\Gamma\setminus\Gamma_2,\\
         d_P\frac{\partial P_{\infty}}{\partial \nu} = \frac{\sigma \lambda}{\gamma + \sigma} L_{\infty}, &\qquad x\in\Gamma_2,\\
         d_P\frac{\partial P_{\infty}}{\partial \nu} = 0, &\qquad x\in\Gamma\setminus\Gamma_2,
         \end{cases}
         \label{SS}
         \end{equation}
         which is also the stationary state system of the QSSA system \eqref{ff1}--\eqref{ff3}. In fact, by solving \eqref{SS} the stationary concentration $l_{\infty}$ and $p_{\infty}$ are afterwards calculated from $L_{\infty}$ and $P_{\infty}$ for both systems  \eqref{newf1}--\eqref{newf3} and \eqref{ff1}--\eqref{ff3}. 
         
         Moreover, the stationary state system \eqref{SS} can be solved by observing that 
         \begin{equation}
         \begin{cases}
         -\Delta (d_L L_{\infty} + d_P P_{\infty}) = 0,&\qquad x\in\Omega,\\
         \frac{\partial}{\partial \nu}\bigl(d_L L_{\infty} + d_P P_{\infty}\bigr) = 0,&\qquad x \in\Gamma.
         \end{cases}
         \label{SS1}
         \end{equation}
         Thus, the sum $d_L L_{\infty} + d_P P_{\infty}=C$ equals a constant $C$ for all $x\in\Omega$ and the stationary state concentrations $L_{\infty}$ or $P_{\infty}$, respectively are obtained by solving 
         an inhomogeneous linear elliptic boundary value problem with mixed Neumann/Robin boundary data. For instance, the equilibrium concentration 
         $L_{\infty}$ satisfies
         \begin{equation}
         \begin{cases}
         -d_L\Delta L_{\infty} +\bigl(\beta +\alpha\frac{d_L}{d_P}\bigr) L_{\infty} = \alpha \frac{C}{d_P},&\qquad x\in\Omega,\\
         d_L\frac{\partial L_{\infty}}{\partial \nu} = -\frac{\sigma \lambda}{\gamma+\sigma} L_{\infty},&\qquad x \in\Gamma_2,\\
         d_L\frac{\partial L_{\infty}}{\partial \nu} = 0,&\qquad x \in\Gamma\setminus\Gamma_2,\\
         \end{cases}
         \label{SSLinf}
         \end{equation}
         
         By standard computations, the stationary state $L_{\infty}$ is unique for each constant $C$.
%         , since the 
%         difference $\hat L_{\infty}$ of two such steady state solutions satisfies a homogeneous version of \eqref{SSLinf}
%         with $C=0$, which has only the trivial solution. Indeed, when testing \eqref{SSLinf} for $C=0$ with $\hat L_{\infty}$, we see that any solution $\hat L_{\infty}$ has to satisfy
%         $$
%         d_L\int_{\Omega} \bigl|\nabla \hat L_{\infty}\bigr|^2\,dx + 
%         \frac{\sigma \lambda}{\gamma+\sigma} \int_{\Gamma_2} \hat L_{\infty}^2\,dS +
%         \biggl(\beta +\alpha\frac{d_L}{d_P}\biggr) \int_{\Omega} \hat L_{\infty}^2\,dx =0,
%         $$
%         which implies $\hat L_{\infty}=0$.
%         
         Moreover, the constant $C$ is itself determined by the conserved initial total mass.
         As a consequence, since the system \eqref{newf1}--\eqref{newf3} and its QSSA system\eqref{ff1}--\eqref{ff3} with the condition \eqref{ff3a} share by construction the same initial total  mass,  
         the corresponding stationary states are identical. See Figure \ref{SteadyState} for numerical illustration.
         \end{remark}
         
         \medskip
         
         In the following, we will show that solutions to \eqref{newf1}--\eqref{newf3}
         converge towards solutions of the QSSA system \eqref{ff1}--\eqref{ff3} 
         as $\xi\to+\infty$.         
         We remark here that all generic constants $C$ and $C_T$ are \emph{independent of the reaction rate $\xi$}.
                  
         The following Lemma provides some crucial {\it a priori} estimates, which will  allow to pass to the limit $\xi \rightarrow +\infty$.
         
         \medskip
         
         \begin{lemma}[Uniform in $\xi$ Boundedness of Solutions to the Original System]\label{lem1}\hfill\\
         For any $T>0$, the solution $(L, P, l, p)$ to system \eqref{newf1}--\eqref{newf3} satisfies the following estimates: 
         \begin{equation}
         \|L\|_{L^2(\Omega_T)} + \|P\|_{L^2(\Omega_T)} + \|l\|_{L^2(\Gamma_T)}\leq C_T,
         \label{g1}
         \end{equation}
         and
         \begin{equation}\label{g1a}
         \|L\|_{L^2(0,T;H^1(\Omega))} \leq C_T,
         %\|L\|_{L^2(0,T;H^1(\Omega))} + \|P\|_{L^2(0,T;H^1(\Omega))} \leq C_T.
         \end{equation}
the constants $C_T$ do not depend on the reaction rate $\xi$.
         \end{lemma}
         \begin{proof}
         We will use the duality method in \cite{BP} with suitable modifications to adapt to the case of volume-surface coupling. 
         By setting $Z = L + P$ and $W = d_LL + d_PP$, we get from the non-negativity of $L$ and $P$ that
         \begin{equation*}
         0< \min\{d_L,d_P\} \leq \frac{W}{Z} \leq \max\{d_L, d_P\} < +\infty.
         \end{equation*}
         It follows from \eqref{newf1}--\eqref{newf2} that
         \begin{equation}
         \begin{cases}
         Z_t - \Delta W = 0,&\qquad x\in\Omega,\quad t>0,\\
         \frac{\partial W}{\partial \nu} = -\lambda L + \gamma l + \chi_{\Gamma_2}\xi p,&\qquad x\in\Gamma,\quad t>0.
         \end{cases}
         \label{g2}
         \end{equation}
         We integrate the first equation in \eqref{g2} over $(0,t)$ and take then the inner product with $W(t)$ in $L^2(\Omega)$ to get
         \begin{multline}
         \int_{\Omega}W(t)[Z(t) - Z(0)]\,dx + \int_{\Omega}\nabla W(t)\cdot \nabla \int_{0}^{t}W(s)\,dsdx\\
         = \int_{\Gamma}W(t)\int_{0}^{t}(-\lambda L + \gamma l + \chi_{\Gamma_2}\xi p)\,dsdS.
         \label{g3}
         \end{multline}
         From \eqref{newf3} we have $-\lambda L + \gamma l + \chi_{\Gamma_2}\xi p = - (l + \chi_{\Gamma_2}p)_t$, thus it follows from \eqref{g3} that
%         In order to estimate the right hand side of \eqref{g3}, we observe from \eqref{newf3} that 
%         \begin{equation}\label{g3a}
%          .
%         \end{equation}
%         Thus, from \eqref{g3} and \eqref{g3a}, we have 
         \begin{multline}
         \int_{\Omega}W(t)[Z(t)-Z(0)]\,dx + \int_{\Omega}\nabla W(t)\cdot \nabla\int_{0}^tW(s)\,dsdx\\
         = -\int_{\Gamma}W(t)[l(t)-l(0)]\,dS - \int_{\Gamma_2}W(t)[p(t)-p(0)]\,dS.
         \label{g7}
         \end{multline}
         
         In the following, we shall denote by $\phi(t,x) := \int_0^tW(s,x)ds$, which implies $\partial_t\phi(t) = W(t)$. Therefore, we calculate  
         \begin{equation*}
         \int_{\Omega}\nabla W(t)\cdot\nabla \int_0^tW(s)\,dsdx = 
         \int_{\Omega}\nabla \partial_t\phi(t)\cdot\nabla \phi(t)\,dx
         = \frac{1}{2}\int_{\Omega}\frac{\partial}{\partial t}|\nabla \phi(t)|^2\,dx.
         \end{equation*}
         As a consequence, integration of \eqref{g7} in $t$ over $(0,T)$ yields
         \begin{multline}
         \int_{\Omega_T}WZ\,dsdx + \int_{\Gamma_T}Wl\,dsdS + \int_{\Gamma_{2,T}}Wp\,dsdS +\frac{1}{2}\int_{\Omega}\bigl|\nabla\phi(T)\bigr|^2\,dx\\
         = \int_{\Omega_T}WZ(0)\,dsdx + \int_{\Gamma_T}Wl(0)\,dsdS + \int_{\Gamma_{2,T}}Wp(0)\,dsdS,
         \label{g8}
         \end{multline}
         where we have used that
         $
         \int_0^T\int_{\Omega}\frac{\partial}{\partial t}|\nabla \phi(t)|^2\,dxdt  
         = \int_{\Omega}\bigl|\nabla\phi(T)\bigr|^2dx
         $.
         
         Next, by Young's inequality, we have
         \begin{equation}\label{star}
         \|\phi(T)\|_{\Omega}^2 = \int_{\Omega}\biggl|\int_{0}^{T}W(s,x)ds\biggr|^2dx \leq T\int_{\Omega}\int_{0}^{T}|W(s,x)|^2\,dsdx = T\|W\|_{L^2(\Omega_T)}^2.
         \end{equation}
         Considering the right hand side of \eqref{g8}, we estimate in the following by Cauchy's, Young's, a Trace inequalities and \eqref{star} that
         \begin{equation}
         \int_{\Omega_T}WZ(0)\,dsdx \leq \sqrt{T}\|W\|_{L^2(\Omega_T)}\|Z(0)\|_{\Omega},
         \label{g9}
         \end{equation}
         and
         \begin{equation}
         \begin{aligned}
         \int_{\Gamma_T}Wl(0)\,dsdS&= (\phi(T),l(0))_{\Gamma}\leq \|l(0)\|_{\Gamma}\|\phi(T)\|_{\Gamma} \\
         &\leq \|l(0)\|_{\Gamma}(C\|\nabla \phi(T)\|_{\Omega} + \|\phi(T)\|_{\Omega})\\
         &\leq  \frac{1}{8}\|\nabla \phi(T)\|_{\Omega}^2 + C\|l(0)\|_{\Gamma}^2+\sqrt{T}\|l(0)\|_{\Gamma}\|W\|_{L^2(\Omega_T)},
         \end{aligned}
         \label{g10}
         \end{equation}
         and similarly,
         \begin{equation}
         \int_{\Gamma_{2,T}}Wp(0)\,dsdS \leq \frac{1}{8}\|\nabla \phi(T)\|_{\Omega}^2 + C\|p(0)\|_{\Gamma_2}^2+ \sqrt{T}\|p(0)\|_{\Gamma_2}\|W\|_{L^2(\Omega_T)}.
         \label{g11}
         \end{equation}
         
         Hence, by the non-negativity of $L, P, l, p$, we obtain from \eqref{g8}--\eqref{g11}
         \begin{equation}
         \int_{\Omega_T}WZ\,dsdx %+ \frac{1}{4}\int_{\Omega}\bigl|\nabla\phi(T)\bigr|^2\,dx\\
         \leq C\sqrt{T}\|W\|_{L^2(\Omega_T)} + C.%\sqrt{T}(\|Z(0)\|_{\Omega}+\|l(0)\|_{\Gamma}+\|p(0)\|_{\Gamma_2})\|W\|_{L^2(\Omega_T)} + C(\|l(0)\|_{\Gamma}^2+\|p(0)\|_{\Gamma_2}^2)
         %\leq C\|W\|_{2,T,\Omega} + C.
         \label{g12}
         \end{equation}
         It follows then from $W\le \max\{d_L,d_P\}Z$ and \eqref{g12} that
%         \begin{equation}
%         \|W\|_{L^2(\Omega_T)}^2 \leq \max\{d_L, d_P\}\int_{\Omega_T}WZ\,dsdx \leq C_T\|W\|_{L^2(\Omega_T)} + C,
%         \label{g13}
%         \end{equation}
%         and the linearity of the right hand side implies by Young's inequality that 
         \begin{equation}
         \|W\|_{L^2(\Omega_T)}\leq C_T.
         \label{g14}
         \end{equation}
         Therefore, by the non-negativity of $L$ and $P$ and by keeping in mind that $W = d_LL + d_PP$, we conclude that
         \begin{equation}
         \|L\|_{L^2(\Omega_T)} + \|P\|_{L^2(\Omega_T)} \leq C_T.
         \label{g15}
         \end{equation}
         
         Next,  by testing \eqref{newf1} with $L$, we estimate
         \begin{equation}
         \begin{aligned}
         \frac{1}{2}\frac{d}{dt}\|L\|_{\Omega}^2 + d_L\|\nabla L\|_{\Omega}^2 &= -\beta\|L\|_{\Omega}^2 + \alpha\int_{\Omega}LP\,dx - \lambda\|L\|_{\Gamma}^2 + \gamma\int_{\Gamma}Ll\,dS\\
         &\leq C\|P\|_{\Omega}^2 + C\|l\|_{\Gamma}^2 - \frac{\lambda}{2}\|L\|_{\Gamma}^2.
         \end{aligned}
         \label{g16}
         \end{equation}
         On the other hand, we get from \eqref{newf3} 
         \begin{equation}
         \frac{1}{2}\frac{d}{dt}\|l\|_{\Gamma}^2 = -\gamma\|l\|_{\Gamma}^2 - \sigma\|l\|_{\Gamma_2}^2 + \lambda\int_{\Gamma}Ll\,dS
         \leq C\|l\|_{\Gamma}^2 + \frac{\lambda}{2}\|L\|_{\Gamma}^2.
         \label{g17}
         \end{equation}
         Summing \eqref{g16} and \eqref{g17} yields
         \begin{equation}
         \frac{d}{dt}(\|L\|_{\Omega}^{2}+\|l\|_{\Gamma}^2) + 2d_L\|\nabla L\|_{\Omega}^2 \leq C\|P\|_{\Omega}^2 + C \|l\|_{\Gamma}^2.
         \label{g18}
         \end{equation}
         Therefore, by $\|L\|_{L^2(\Omega_T)}+\|P\|_{L^2(\Omega_T)}\leq C_T$, we conclude that
%         \begin{equation*}
%         
%         \end{equation*}
%         and 
         \begin{equation}\label{lCT}
         \|\nabla L\|_{L^2(\Omega_T)} \leq C_T \quad \text{ and } \quad \|l\|_{L^2(\Gamma_T)}\leq C_T\, e^{CT} \le C_T.
         \end{equation}
         %Finally, we note that \eqref{g12} implies $\int_{\Omega}\bigl|\nabla\phi(T)\bigr|^2\,dx\le C \|\nabla W\|_{2,T,\Omega}{\color{red}\leq C???}$, then combining with $\|\nabla L\|_{2,T,\Omega}\leq C$ we have
         %\begin{equation*}
         %\|L\|_{L^2(0,T;H^1(\Omega))} + \|P\|_{L^2(0,T;H^1(\Omega))} \leq C_T.
         %\end{equation*}
         This finishes the proof.
         \end{proof}
         \medskip
         
         \begin{remark}\label{remark:QSSAsurfacediffusion}
         The proof applied in Lemma \ref{lem1} fails when trying to include one of the surface diffusion terms $\Delta_\Gamma l$ or $\Delta_{\Gamma_2}p$ because in these cases, the formulation \eqref{g8} would have additional terms $\int_{\Gamma}\nabla_{\Gamma}W(t)\cdot\nabla \int_{0}^tl(s)\,dsdS$ or $\int_{\Gamma_2}\nabla_{\Gamma_2}W(t)\cdot\nabla_{\Gamma_2}\int_{0}^{t}p(s)\,dsdS$, for which we do not know a sign or suitable {\it a priori} estimates. The problem of the QSSA for system 
         \eqref{f1}--\eqref{f3} with surface diffusion remains open for future work.
         \end{remark}   
         
         \medskip
         
         From now on, we always denote  the solution to \eqref{newf1}--\eqref{newf3} by $(L^{\xi}, P^{\xi}, l^{\xi}, p^{\xi})$ in order to emphasise the dependency on the reaction rate $\xi$.
         \medskip
         
         In the next Lemma, we will show that the concentration $p^{\xi}$ of the phosphorylated Lgl on the active boundary $\Gamma_2$ tends to zero as $\xi \rightarrow +\infty$ in $L^2(\Gamma_{2,T})$. Since $p^{\xi}(0) = p_0 \in L^2(\Gamma_2)$, we cannot expect that $p^{\xi} \rightarrow 0$ in $C([0,T]; L^2(\Gamma_2))$. Nevertheless, we will be able to show that $p^{\xi} \rightarrow 0$ in $C([t_0,T]; L^2(\Gamma_2))$ for all $t_0>0$.
         
         \medskip
         
         \begin{lemma}\label{lem3}
         For al $T>0$ and $T> t_0>0$, we have 
         $$
         p^{\xi} \xrightarrow{\xi \rightarrow +\infty} 0 \quad\text{ in }\quad L^2(\Gamma_{2,T})\cap C([t_0,T]; L^2(\Gamma_2)).
         $$
         \end{lemma}
         \begin{proof}
         By multiplying the equation \eqref{newf3} of $p^{\xi}$, i.e.
         $
         \partial_tp^{\xi} + \xi p^{\xi} = \sigma l^{\xi}
         $
         with $\xi p^{\xi}$ and integrating over $\Gamma_2$, we estimate
         \begin{equation*}
         \frac{\xi}{2}\frac{d}{dt}\|p\|_{\Gamma_2}^2 + \xi^2\|p^{\xi}\|_{\Gamma_2}^2 = (\sigma l^{\xi}, \xi p^{\xi})_{\Gamma_2} \leq \frac{\sigma^2}{2}\|l^{\xi}\|_{\Gamma_2}^2 + \frac{\xi^2}{2}\|p^{\xi}\|_{\Gamma_2}^2.
         \end{equation*}
         Therefore,
         \begin{equation*}
         \xi\frac{d}{dt}\|p^{\xi}\|_{\Gamma_2}^2 + \xi^2\|p^{\xi}\|_{\Gamma_2}^2 \leq \sigma^2\|l^{\xi}\|_{\Gamma_2}^2 \leq \sigma^2\|l^{\xi}\|_{\Gamma}^2,
         \end{equation*}
         and integration over $(0,T)$ yields
         \begin{equation}\label{xirate}
         \frac{1}{\xi}\|p^{\xi}(T)\|_{\Gamma_2}^2 + \int_{0}^{T}\|p^{\xi}(s)\|_{\Gamma_2}^2\,ds \leq \frac{1}{\xi}\|p_0\|_{\Gamma_2}^2 + \frac{\sigma^2}{\xi^2}\int_{0}^{T}\|l^{\xi}(s)\|_{\Gamma}^2\,ds.
         \end{equation}
         This implies that $\|p^{\xi}\|_{\Gamma_{2,T}}^2=O(\xi^{-1})$ and $p^{\xi} \rightarrow 0$ in $L^2(\Gamma_{2,T})$ as $\xi\rightarrow +\infty$ since $\{l^{\xi}\}_{\xi >0}$ is uniformly bounded in $L^2(\Gamma_T)$ according to Lemma \ref{lem1}.
         
         From \eqref{newf3}, we get
         \begin{equation*}
         \frac{d}{dt}\|p^{\xi}\|_{\Gamma_2}^2 + 2\xi\|p^{\xi}\|_{\Gamma_2}^2 = 2\sigma(l^{\xi}, p^{\xi})_{\Gamma_2}.
         \end{equation*}
         Hence, for any fixed $0<t_0\le t \le T$, we have 
         \begin{align*}
         \|p^{\xi}(t)\|_{\Gamma_2}^2 &\leq e^{-2\xi t}\|p_0\|_{\Gamma_2}^2 + 2\sigma e^{-2\xi t}\int_{0}^{t}e^{2\xi s}(l^{\xi}, p^{\xi})_{\Gamma_2}ds\\
%         &\leq e^{-2\xi t_0}\|p_0\|_{\Gamma_2}^2  + 2\sigma\int_{0}^{T}(l^{\xi}, p^{\xi})_{\Gamma_2}\,ds\\
         &\leq e^{-2\xi t_0}\|p_0\|_{\Gamma_2}^2  + 2\sigma\|l^{\xi}\|_{L^2(\Gamma_T)}\|p^{\xi}\|_{L^2(\Gamma_{2,T})}.
         \end{align*}
         In the limit $\xi \rightarrow +\infty$ and by using $\|p^{\xi}\|_{\Gamma_{2,T}} \rightarrow 0$ and $\{l^{\xi}\}_{\xi>0}$ is bounded in $L^2(\Gamma_T)$, we have thus $p^{\xi} \rightarrow 0$ in $C([t_0,T];L^2(\Gamma_2))$ for all $t_0>0$.
         \end{proof}
         
         \medskip
         
         \begin{lemma}\label{lem:limits} 
         There exists $L \in L^{2}(\Omega_T)$ and $l\in L^2(\Gamma_T)$ such that, when $\xi \rightarrow +\infty$
         \begin{equation}
         L^{\xi} \xrightarrow{\xi \rightarrow +\infty} L \quad\text{ in }\quad L^2(\Omega_T)\qquad \text{ and } \qquad  l^{\xi} \xrightarrow{\xi \rightarrow +\infty} l \quad\text{ in }\quad L^2(\Gamma_T). 
         \label{ff12}
         \end{equation}
%         and
%         \begin{equation}
%         l^{\xi} \xrightarrow{\xi \rightarrow +\infty} l \quad\text{ in }\quad L^2(\Gamma_T). 
%         \label{ff12_1}
%         \end{equation}
         \end{lemma}
         \begin{proof}
         By Lemma \ref{lem1}, we have that $\{L^{\xi}\}_{\xi >0}$ is bounded in $L^2(0,T;H^1(\Omega))$. Thus, by using \eqref{newf1}, we have $\{\partial_tL^{\xi}\}_{\xi >0}$ is bounded in $L^2(0,T;H^{-1}(\Omega))$ and the Aubin-Lions compactness lemma implies that $\{L^{\xi}\}_{\xi >0}$ is precompact in $L^2(\Omega_T)$. Thus, 
         \begin{equation*}
         L^{\xi} \xrightarrow{\xi \rightarrow +\infty} L \quad\text{ in }\quad L^{2}(\Omega_T)
         \end{equation*}
         for some $L\in L^2(\Omega_T)$ and up to a subsequence. By using that $\{L^{\xi}\}_{\xi >0}$ is bounded in $L^2(0,T;H^1(\Omega))$ and by a standard Trace Theorem (see e.g. \cite{GT}), we have $\{L^{\xi}|_{\Gamma}\}_{\xi >0}$ is also bounded in $L^2(0,T;H^{1/2}(\Gamma))$. Therefore, it follows from $l^{\xi}_t = \lambda L^{\xi} - (\gamma+\sigma\chi_{\Gamma_2})l^{\xi}$
%         \begin{equation*}
%         l^{\xi}_t = \lambda L^{\xi} - (\gamma+\sigma\chi_{\Gamma_2})l^{\xi}
%         \end{equation*}
         that $\{l^{\xi}\}_{\xi >0}$ is bounded in $L^2(0,T;H^{1/2}(\Gamma))$ and $\{l^{\xi}_t\}_{\xi >0}$ is bounded in $L^2(\Gamma_T)$. Using again the Aubin-Lions compactness lemma, we have
         \begin{equation*}
         l^{\xi} \xrightarrow{\xi \rightarrow +\infty} l \quad\text{ in }\quad L^2(\Gamma_T)
         \end{equation*}
         for some $l \in L^2(\Gamma_T)$ and up to a subsequence.
         \end{proof}
         \medskip
         
         By Lemma \ref{lem:limits}, we have so far established the strong convergence of $L^{\xi}, l^{\xi}$ and $p^{\xi}$ in $L^2(\Omega_T)$, $L^2(\Gamma_T)$ and $L^2(\Gamma_{2,T})$, respectively.
         
         The convergence of $P^{\xi}$ constitutes a more difficult problem due to the singularity of the boundary flux $d_P\partial P^{\xi}/\partial\nu = \chi_{\Gamma_2}\xi p^{\xi}$. 
         As an example to illustrate the, we may attempt a similar approach like in
         Lemma \ref{lem1}, which succeeded in proving the bound \eqref{g1a}: 
         By testing
         \begin{equation*}\label{PPP}
%         \begin{cases}
         	P^{\xi}_t - d_P\Delta P^{\xi} = \beta L^{\xi} - \alpha P^{\xi}, \qquad 
         	d_P\frac{\partial P^{\xi}}{\partial\nu} = \chi_{\Gamma_2}\xi p^{\xi},
%         \end{cases}
         \end{equation*}
         with $P^{\xi}$, we get after direct computations that
         \begin{equation}\label{PP}
         d_P\int_{0}^{T}\|\nabla P^{\xi}\|_{\Omega}^2\,ds \leq \|P_0\|_{\Omega}^2 + \int_{0}^{T}(\beta L^{\xi} - \alpha P^{\xi},P^{\xi})_{\Omega}\,ds + \xi\int_{0}^{T}(P^{\xi},p^{\xi})_{\Gamma_2}\,ds.
         \end{equation}
         Due to the boundedness of $L^{\xi}, P^{\xi}$ in $L^2(\Omega_T)$, we would need the uniform boundedness of $\xi\int_{0}^{T}(P^{\xi},p^{\xi})_{\Gamma_2}ds$
         in order to prove a uniform control of the left hand side of \eqref{PP}.
         A uniform bound on Neumann data of \eqref{PPP} seems thus to requires a uniform bound of $\xi\|p^{\xi}\|_{L^2(\Gamma_{2,T})}$ or equivalently $\|p^{\xi}\|_{L^2(\Gamma_{2,T})} \rightarrow 0$ when $\xi\rightarrow+\infty$ with the rate $1/\xi$.  However, Lemma \ref{lem3} implies only the decay rate of $\|p^{\xi}\|_{\Gamma_{2,T}}=O(1/\sqrt{\xi})$.
         
         Nevertheless, we notice that \eqref{newf3} implies
         \begin{equation}
         \begin{aligned}
         \xi \| p^{\xi}\|_{L^1(\Gamma_{2,T})} &= \int_{0}^{T}\int_{\Gamma_2}\xi p^{\xi}(t)\,dSdt= \int_{0}^{T}\int_{\Gamma_2}(\sigma l^{\xi} - \partial_tp^{\xi})\,dSdt\\
%         &= \|p_0\|_{L^1(\Gamma_2)} - \|p^{\xi}(T)\|_{L^1(\Gamma_2)} +\sigma\int_{0}^T\|l^{\xi}\|_{L^1(\Gamma_2)}dt\\
         &\leq \|p_0\|_{L^1(\Gamma_2)}  + C\sigma\|l^{\xi}\|_{L^2(\Gamma_T)}
         \le C_T.
         \end{aligned}\label{pl1}
         \end{equation}
         and the uniform $L^1$-bound \eqref{pl1} will be used in Lemma \ref{Pcompact} below
         to obtain the compactness of $\{P^{\xi}\}_{\xi >0}$ in $L^1(\Omega_T)$ and even in $L^1(0,T;W^{1,1}(\Omega))$. 
         The proof of Lemma \ref{Pcompact} is based on Lemma \ref{compactness}, which is similar to results given in \cite{PP} and \cite{BP}, yet for homogeneous boundary conditions. The proof of Lemma \ref{compactness} is based on a duality argument and will given in the Appendix of the sake of completeness.
         \medskip

         \medskip
         
         \begin{lemma}\label{compactness}
         The mapping $\mathfrak T: (w_0, \Theta, g)\rightarrow (w, \nabla w)$, where $w$ is the solution of
         \begin{equation}\label{dual}
         \begin{cases}
         w_t - d_P\Delta w = \Theta, &\qquad x\in\Omega,\quad t>0,\\
         d_P\partial w/\partial\nu = g, &\qquad x\in\Gamma,\quad t>0,\\
         w(0,\cdot) = w_0, &\qquad x\in\Omega,
         \end{cases}
         \end{equation}
         is compact from $L^1(\Omega)\times L^1(\Omega_T)\times L^1(\Gamma_T)$ into $L^1(\Omega_T)\times (L^1(\Omega_T))^N$.
         \end{lemma}
         
         \medskip
         
         Applying Lemma \ref{compactness} to $w = P^{\xi}$, $\Theta = \beta L^{\xi} - \alpha P^{\xi}$ and $g = \chi_{\Gamma_2}\xi p^{\xi}$ leads to
         \begin{lemma}\label{Pcompact}
         The sequence $\{P^{\xi}\}_{\xi>0}$ is pre-compact in $L^1(0,T;W^{1,1}(\Omega))$. In other words, there exists $P\in L^1(0,T;W^{1,1}(\Omega))$ such that up to a subsequence
         $$
         P^{\xi} \xrightarrow{\xi \rightarrow +\infty} P  \quad\text{ strongly in }\quad L^1(0,T;W^{1,1}(\Omega)) 
         \quad\text{ and weakly in }\quad L^2(\Omega_T).
         $$
%         due to Lemma \ref{lem1}.
         \end{lemma}
         
%         \begin{remark}
%         The proof of Lemma \ref{compactness} as stated in the Appendix shows actually 
%         that the mapping $\mathfrak T$ is indeed compact from $L^1(\Omega)\times L^1(\Omega_T)\times L^1(\Gamma_T)$ into $L^r(\Omega_T)\times (L^s(\Omega_T))^N$ for any $r<\frac{N+2}{N}$ and $s<\frac{N+2}{N+1}$. Thus, depending on the space dimension $N$, the convergence in 
%         Lemma \ref{Pcompact} could be somewhat improved. 
%         \end{remark}
         
         The following Theorem is the main result of this section.
         
         \medskip
         
         \begin{theorem}[Convergence of the QSSA]\label{QSSA}\hfill\\
         For any $(L_0,P_0,l_0,p_0)\in L^2(\Omega)\times L^2(\Omega)\times L^2(\Gamma)\times L^2(\Gamma_2)$ and all $T>0$, $t_0>0$, we have
         %\begin{equation*}
         %L^{\xi} \xrightarrow{\xi \rightarrow +\infty} L \quad\text{ in }\quad L^2(\Omega_T),
         %\end{equation*}
         \begin{equation*}
         L^{\xi} \xrightarrow{\xi \rightarrow +\infty} L  \quad\text{ strongly in }\quad 
         L^2(\Omega_T) 
         \quad\text{ and weakly in }\quad L^2(0,T;H^{1}(\Omega)),
         \end{equation*}
         \begin{equation*}
         P^{\xi} \xrightarrow{\xi \rightarrow +\infty} P  \quad\text{ strongly in }\quad L^1(0,T;W^{1,1}(\Omega)) 
         \quad\text{ and weakly in }\quad L^2(\Omega_T),
         \end{equation*}
         %\begin{equation*}
         %P^{\xi} \xrightarrow{\xi \rightarrow +\infty} P \quad\text{ in }\quad L^1(\Omega_T),
         %\end{equation*}
         \begin{equation*}
         l^{\xi} \xrightarrow{\xi \rightarrow +\infty} l  \quad\text{ strongly in }\quad L^2(\Gamma_T), 
         \end{equation*}
         and
         \begin{equation*}
         p^{\xi} \xrightarrow{\xi \rightarrow +\infty} 0 \quad\text{ strongly in }\quad L^2({\Gamma_{2}}_T)\cap C([t_0,T];L^2(\Gamma_2)), \quad\forall\, 0<t_0<T,
         \end{equation*}
         up to a subsequence, where $(L,P,l)$ is the unique weak solution to \eqref{ff1}--\eqref{ff3}.
         \end{theorem}
         
         \medskip
         
         \begin{remark}
         The well-posedness of system \eqref{ff1}--\eqref{ff3} can be shown in the same way as for system \eqref{f1}--\eqref{f3} in Section \ref{sec:2}.
         \end{remark}
         \begin{proof}
         All the limits are  already proven in the Lemmata \ref{lem3}, \ref{lem:limits} and \ref{Pcompact}. It remains to show that the limit $(L,P,l)$ in Lemma \ref{lem:limits} is the unique solution of system \eqref{ff1}--\eqref{ff3}. Indeed, by testing 
         \begin{equation*}
         \begin{cases}
         L^{\xi}_t - d_L\Delta L^{\xi} = -\beta L^{\xi} + \alpha P^{\xi}, &\qquad x\in\Omega,\quad t>0,\\
         d_L\frac{\partial L^{\xi}}{\partial \nu} = -\lambda L^{\xi} + \gamma l^{\xi}, &\qquad x\in\Gamma,\quad t>0,\\
         L^{\xi}(0,x) = L_0(x), &\qquad x\in\Omega,
         \end{cases}
         \end{equation*}
         with $\varphi \in C^1([0,T];H^1(\Omega))$, $\varphi(T)=0$ and by integration over $\Omega_T$, we have
         \begin{multline}
         %\begin{gathered}
         -\int_{0}^T(L^{\xi}, \varphi_t)_{\Omega}\,ds + d_L\int_{0}^{T}(\nabla L^{\xi}, \nabla\varphi)_{\Omega}\,ds\\
         = (L_0, \varphi(0))_{\Omega} + \int_{0}^{T}(-\lambda L^{\xi} + \gamma l^{\xi},\varphi)_{\Gamma}\,ds + \int_{0}^{T}(-\beta L^{\xi} + \alpha P^{\xi}, \varphi)_{\Omega}\,ds.
         %\end{gathered}
         \label{weakLxi}
         \end{multline}
%         Since $L^{\xi} \rightarrow L$ and $P^{\xi} \rightharpoonup P$ in $L^2(\Omega_T)$ as $\xi\rightarrow +\infty$, we have
%         \begin{equation}
%         -\int_{0}^{T}(L^{\xi}, \varphi_t)_{\Omega}\,ds \xrightarrow{\xi \rightarrow +\infty} -\int_{0}^{T}(L, \varphi_t)_{\Omega}\,ds
%         \label{ff13}
%         \end{equation}
%         and
%         \begin{equation}
%         \int_{0}^{T}(-\beta L^{\xi} + \alpha P^{\xi}, \varphi)_{\Omega}\,ds \xrightarrow{\xi \rightarrow +\infty} \int_{0}^{T}(-\beta L + \alpha P, \varphi)_{\Omega}\,ds.
%         \label{ff14}
%         \end{equation}
         By Lemma \ref{lem1}, $\{L^{\xi}\}_{\xi>0}$ is bounded in $L^2(0,T; H^1(\Omega))$ and together with \eqref{ff12}, we get
         \begin{equation}
         L^{\xi} \rightharpoonup L \quad\text{ in }\quad L^2(0,T;H^1(\Omega))
         \label{LPlimH1}
         \end{equation}
         up to a subsequence. 
%         Thus,
%         \begin{equation}
%         d_L\int_{0}^{T}(\nabla L^{\xi}, \nabla \varphi^{\xi})_{\Omega}ds \xrightarrow{\xi \rightarrow +\infty} d_L\int_{0}^{T}(\nabla L, \nabla\varphi)_{\Omega}ds.
%         \label{ff15}
%         \end{equation}
         By using the Trace Theorem and \eqref{LPlimH1}, we have
         \begin{equation}\label{ff16}
         L^{\xi} \rightharpoonup L \quad\text{ in }\quad L^2(\Gamma_T).
      	  \end{equation}
%         In combination with \eqref{ff12_1}, this yields
%         \begin{equation}
%         \int_{0}^{T}(-\lambda L^{\xi} + \gamma l^{\xi}, \varphi)_{\Gamma}ds \xrightarrow{\xi \rightarrow +\infty} \int_{0}^{T}(-\lambda L + \gamma l,\varphi)_{\Gamma}ds.
%%         \label{ff16}
%         \end{equation}
         From \eqref{LPlimH1}--\eqref{ff16}, $L^{\xi} \rightarrow L$ and $P^{\xi} \rightharpoonup P$ in $L^2(\Omega_T)$, $l^{\xi} \rightarrow l$ in $L^2(\Gamma_T)$, we can pass to the limit in \eqref{weakLxi} as $\xi\rightarrow +\infty$ and obtain 
         \begin{multline}
         %\begin{gathered}
         -\int_{0}^T(L, \varphi_t)_{\Omega}\,ds + d_L\int_{0}^{T}(\nabla L, \nabla\varphi)_{\Omega}\,ds\\
         = (L_0, \varphi(0))_{\Omega} + \int_{0}^{T}(-\lambda L + \gamma l,\varphi)_{\Gamma}\,ds + \int_{0}^{T}(-\beta L + \alpha P, \varphi)_{\Omega}\,ds
         %\end{gathered}
         \label{ff17}
         \end{multline}
         or equivalently that $L$ is a weak solution of \eqref{ff1}.
%         \begin{equation}
%         \begin{cases}
%         L_t - d_L\Delta L = -\beta L + \alpha P, &\qquad x\in\Omega,\quad t>0,\\
%         d_L\frac{\partial L}{\partial \nu} = -\lambda L + \gamma l, &\qquad x\in\Gamma,\quad t>0,\\
%         L(0,x) = L_0(x), &\qquad x\in\Omega.
%         \end{cases}
%         \label{ff17_1}
%         \end{equation}
         
         Next, by taking the inner product of $p^{\xi}_t = \sigma l^{\xi} - \xi p^{\xi}$
%         \begin{equation*}
%         p^{\xi}_t = \sigma l^{\xi} - \xi p^{\xi}
%         \end{equation*}
         with a test function $\psi\in C^1(0,T;L^2(\Gamma_{2}))$ satisfying $\psi(T) = 0$, we get
         \begin{equation}
         -\int_{0}^{T}(p^{\xi}, \psi_t)_{\Gamma_2}\,ds= (p_0, \psi(0))_{\Gamma_2} + \int_{0}^{T}(\sigma l^{\xi}, \psi)_{\Gamma_2}\,ds - \int_{0}^{T}(\xi p^{\xi}, \psi)_{\Gamma_2}\,ds.
         \label{ff18}
         \end{equation}
         In order to pass to the limit $\xi \rightarrow +\infty$ in \eqref{ff18}, we apply Lemma \ref{lem3} and Lemma \ref{lem:limits} and obtain
         \begin{equation} 
         \lim_{\xi\rightarrow+\infty}\int_{0}^{T}(\xi p^{\xi}, \psi)_{\Gamma_2}ds = (p_0, \psi(0))_{\Gamma_2} + \int_{0}^{T}(\sigma l, \psi)_{\Gamma_2}ds.
         \label{ff19}
         \end{equation}
         
         In the following, we consider equation for $P^{\xi}$ in the weak form, i.e. 
         \begin{multline}
         %\begin{gathered}
         -\int_{0}^{T}(P^{\xi}, \varphi_t)_{\Omega}\,ds + d_L\int_{0}^{T}(\nabla P^{\xi}, \nabla \varphi)_{\Omega}\,ds\\
         = (P_0, \varphi(0))_{\Omega} + \int_{0}^{T}(\xi p^{\xi}, \varphi)_{\Gamma_2}\,ds + \int_{0}^{T}(\beta L^{\xi} - \alpha P^{\xi}, \varphi)_{\Omega}\,ds
         %\end{gathered}
         \label{ff20}
         \end{multline}
         for test-functions $\varphi \in C^1(0,T;C^1(\Omega))$ with $\varphi(T)=0$. 
%         The following limits
%         \begin{equation}
%         -\int_{0}^{T}(P^{\xi}, \varphi_t)_{\Omega}\,ds \xrightarrow{\xi\rightarrow+\infty} -\int_{0}^{T}(P, \varphi_t)_{\Omega}\,ds,
%         \label{ff21}
%         \end{equation}
%         \begin{equation}
%         d_L\int_{0}^{T}(\nabla P^{\xi}, \nabla \varphi)_{\Omega}\,ds \xrightarrow{\xi\rightarrow+\infty} d_L\int_{0}^{T}(\nabla P, \nabla \varphi)_{\Omega}\,ds
%         \label{ff22}
%         \end{equation}
%         and
%         \begin{equation}
%         \int_{0}^{T}(\beta L^{\xi} - \alpha P^{\xi}, \varphi)_{\Omega}\,ds \xrightarrow{\xi\rightarrow+\infty} \int_{0}^{T}(\beta L - \alpha P, \varphi)_{\Omega}\,ds
%         \label{ff23}
%         \end{equation}
%         are due to 
         By using \eqref{ff19}, the limits $L^{\xi} \rightarrow L$ and $P^{\xi} \rightarrow P$ in $L^1(0,T;W^{1,1}(\Omega))$, we can pass to the limit $\xi\rightarrow +\infty$ in \eqref{ff20}, we obtain 
         \begin{multline}
         %\begin{gathered}
         -\int_{0}^{T}(P, \varphi_t)_{\Omega}\,ds + d_L\int_{0}^{T}(\nabla P, \nabla \varphi)_{\Omega}\,ds\\
%         = (P_0, \varphi(0))_{\Omega} + (p_0, \varphi(0))_{\Gamma_2} + \int_{0}^{T}(\sigma l, \varphi)_{\Gamma_2}\,ds + \int_{0}^{T}(\beta L - \alpha P, \varphi)_{\Omega}\,ds\\
         = {(P_0 + P^*, \varphi(0))_{\Omega}}  + \int_{0}^{T}(\sigma l, \varphi)_{\Gamma_2}\,ds + \int_{0}^{T}(\beta L - \alpha P, \varphi)_{\Omega}\,ds,
         %\end{gathered}
         \label{ff24}
         \end{multline}
         where we have use \eqref{ff3a}.
         This means that $P$ is a weak solution of \eqref{ff2}.
%         \begin{equation}
%         \begin{cases}
%         P_t - d_L\Delta P = \beta L - \alpha P, &\qquad x\in\Omega,\quad t>0,\\
%         d_L\frac{\partial P}{\partial \nu} = \chi_{\Gamma_2}\sigma l, &\qquad x\in\Gamma,\quad t>0,\\
%         P(0,x) = P_0(x) + P^*(x), &\qquad x\in\Omega.
%         \end{cases}
%         \label{ff25}
%         \end{equation}
         
         Finally, by testing $l^{\xi}_t = \lambda L^{\xi} - (\gamma + \chi_{\Gamma_2}\sigma)l^{\xi}$ with $\psi \in C^1([0,T];L^2(\Gamma))$ satisfying $\psi(T) = 0$ 
%         taking test-function $\psi \in C^1([0,T];L^2(\Gamma))$ with $\psi(T) = 0$ for $l^{\xi}_t = \lambda L^{\xi} - (\gamma + \chi_{\Gamma_2}\sigma)l^{\xi}$
%         \begin{equation*}
%         l^{\xi}_t = \lambda L^{\xi} - (\gamma + \chi_{\Gamma_2}\sigma)l^{\xi},
%         \end{equation*}
         we get
         \begin{equation}
         -\int_0^T(l^{\xi}, \psi_t)_{\Gamma}\,ds = (l_0, \psi(0))_{\Gamma} + \int_{0}^{T}(\lambda L^{\xi}, \psi)_{\Gamma}\,ds + \int_{0}^{T}[(\gamma l^{\xi}, \psi)_{\Gamma} + (\sigma l^{\xi}, \psi)_{\Gamma_2}]\,ds.
         \label{ff26}
         \end{equation}
         We use $L^{\xi} \rightharpoonup L$ and $l^{\xi} \rightharpoonup l$ in $L^2(\Gamma_T)$
         to pass the limit $\xi\rightarrow +\infty$ in \eqref{ff26} and obtain
         \begin{equation}
         -\int_0^T(l, \psi_t)_{\Gamma}\,ds = (l_0, \psi(0))_{\Gamma} + \int_{0}^{T}(\lambda L, \psi)_{\Gamma}\,ds + \int_{0}^{T}[(\gamma l, \psi)_{\Gamma} + (\sigma l, \psi)_{\Gamma_2}]\,ds
         \label{ff27}
         \end{equation}
         or equivalently $l$ is a weak solution of \eqref{ff3}.
%         \begin{equation}
%         \begin{cases}
%         l_t = \lambda L - (\gamma + \chi_{\Gamma_2}\sigma)l, &\qquad x\in\Gamma,\quad t>0,\\
%         l(0,x) = l_0(x), &\qquad x\in\Gamma.
%         \end{cases}
%         \label{ff28}
%         \end{equation}
%         In conclusion, from \eqref{ff17_1}, \eqref{ff25} and \eqref{ff28}, $(L, P, l)$ is the unique weak solution to the system \eqref{ff1}--\eqref{ff3}, where the 
%         well-posedness of weak solutions to  \eqref{ff1}--\eqref{ff3} follows in the same way as for the full system \eqref{f1}--\eqref{f3} as shown in Section \ref{sec:2}.
         \end{proof}
         %\vskip 1cm
         %\noindent{\bf A numerical test in 2D:} We consider the domain $\Omega = \{(x,y): x^2 + y^2 = 1\} = \{(r,\theta): r\leq 1\}$ and $\Gamma_2 = \{(r,\theta): r = 1, \pi \leq \theta \leq \frac{3\theta}{2}\}$. The parameters
         %% dL = .1; dP = .1;   % diffusion rates
         %% alpha = 1; beta = 2; gamma = 2; lambda = 2; sigma = 3;  % reaction rates
         %\begin{equation*}
         %d_L = 0.1,\qquad d_P = 0.1,\qquad \alpha = 1,\qquad \beta = 2,\qquad \gamma = 2,\qquad \lambda = 2,\qquad \sigma = 3,
         %\end{equation*}
         %and the initial data
         %% L0 = @(x,y) x+y;            
         %% P0 = @(x,y) x+y;
         %% l0 = @(x,y) 2+0*x;
         %% p0 = @(x,y) 1+0*x;
         %\begin{equation*}
         %L_0(x,y) = x + y,\qquad P_0(x, y) = x + y,\qquad l_0(x,y) = 2,\qquad p_0(x,y) = 1.
         %\end{equation*}
         %The following figure shows the change in time of 
         %$$f(\xi) = \|L^{\xi}-L\|_{2,T,\Omega}^2 + \|P^{\xi}-P\|_{2,T,\Omega}^2 + \|l^{\xi}-l\|_{2,T,\Gamma}^2 + \|p^{\xi}-p\|_{2,T,\Gamma_2}^2.$$
         % \begin{figure}[!ht]
         % \centering
         % \includegraphics[scale=0.2]{L2convergence.jpg}
         % \caption{The decay of $f(\xi)$}%\label{fig:}
         % \end{figure}

 \section{Numerical discretisation and qualitative discussions}\label{sec:3}
 In this section, we shall present prototypical numerical examples of the model system \eqref{f1}--\eqref{f3} in order highlight certain qualitative features. As the kinetic parameters for SOP precursor cells are neither known in-vivo or in-vitro, we shall use what we expect to be generic system parameters i.e. typical reaction- and diffusion rates, for which the model exhibits the expected behaviour. Thus, the aim of the following can only be a \emph{discussion of interesting qualitative features} and not a quantitative simulation of Lgl localisation in SOP precursor cell.
 
%         In this section, we first present a basic numerical finite element scheme for system \eqref{f1}--\eqref{f3}. As a prototypical example, we consider the particular case when $\Omega\subset \mathbb R^2$ is the unit ball and the active boundary part is located at 
%         $\Gamma_2 = \{(1, \theta): \pi\leq \theta \leq 3\pi/2\}$.
         
         {
         The test cases discussed later in this Section serve to illustrate in particular 
         i) the role of surface diffusion in the model \eqref{f1}--\eqref{f3} in smoothing jumps of concentrations on the boundary (see Fig. \ref{fig3}), 
         ii) the influence of the presence/absence of surface diffusion to the system's behaviour and the attained complex balance equilibrium (this happens in combination with the cell-geometry and the discontinuity of the systems's boundary parameters, see Figs. \ref{figLP} and \ref{figBigDiff}),
         and iii) the asymptotic behaviour of the system as $\xi\to\infty$, which confirms the theoretical QSSA in Section \ref{sec:4} (see Figs. \ref{fig3a} and \ref{SteadyState}).
         }
        %i) the role of surface diffusion in the model system 
         %\eqref{f1}--\eqref{f3} and ii) 
         %to investigate the behaviour of the system subject to changes in the release rate $\xi$  of cortical Lgl into the cytoplasm, i.e. $p \xrightarrow{\xi} P$.
         %\subsection{Numerical scheme}
         
As a numerical method, we use implicit Euler as time-discretisation and a piecewise linear finite elements as space-discretisation. In particular, we approximate the boundary $\Gamma$ as a polygon, which allows to calculate the discretisation of the Laplace-Beltrami operator as directional derivatives. Moreover, the used triangulation mesh is strongly refined near the boundary of the active boundary part $\partial\Gamma_2$, which are discontinuous points of the system. The details of the numerical scheme, a plot of the used triangulation mesh and the discretisation of the Laplace-Beltrami operator are given in the Appendix \ref{appnum}. See also \cite{EFPT} for the numerical analysis of such volume-surface reaction- diffusion systems including discrete entropy structures and uniform in time convergence rates.

%         \subsection{Discussion of the numerical results}\label{subsec}
         
As a prototypical geometry of a cell, we shall consider $\Omega\subset \mathbb R^2$ being the unit ball. The active boundary part shall be the intersection of the unit circle with the negative quadrant, i.e.   $\Gamma_2 = \{(1, \theta): \pi\leq \theta \leq 3\pi/2\}$.

%         The following numerical examples shall highlight certain qualitative features of the model system \eqref{f1}--\eqref{f3}.
%         Due to the lack of in-vivo or in-vitro parameters, which are in general unknown  for SOP precursor cells,
%         we shall use what we believe to be generic system parameters i.e. typical reaction- and diffusion rates, for which 
%         the model exhibits the expected behaviour. Thus, the aim of the following  
%         can only be a \emph{discussion of interesting qualitative features} and
%         not a quantitative simulation of Lgl localisation in SOP precursor cell.
         
         As generic parameters, we use the following reaction rates 
         \begin{equation}\label{reac}
         \alpha = 1, \qquad \beta = 2, \qquad \gamma = 2, \qquad \lambda = 4, \qquad \sigma = 3,
         \end{equation}
         the following volume diffusion rates
         \begin{equation}\label{diff}
         d_L = 0.01,\quad d_P = 0.02,
         \end{equation}
         and the following constant initial concentrations
         \begin{equation}\label{InitialData}
         	L_0(x, y) \equiv 0.8,\qquad P_0(x,y) \equiv 0.6, \qquad l_0(x,y) \equiv 0.3, \qquad  p_0(x,y) \equiv 0.4.
         \end{equation}
         The value of $\xi$ will be chosen differently during the discussion of the numerical examples. Also the surface diffusion rates $d_l$ and $d_p$ will be specified later.
         
         \subsection{The effects of surface diffusion}\label{subsec:diff}

         \begin{figure}[htp]
         \centering
         \begin{subfigure}{.5\textwidth}
           \centering
           \includegraphics[width=1\linewidth]{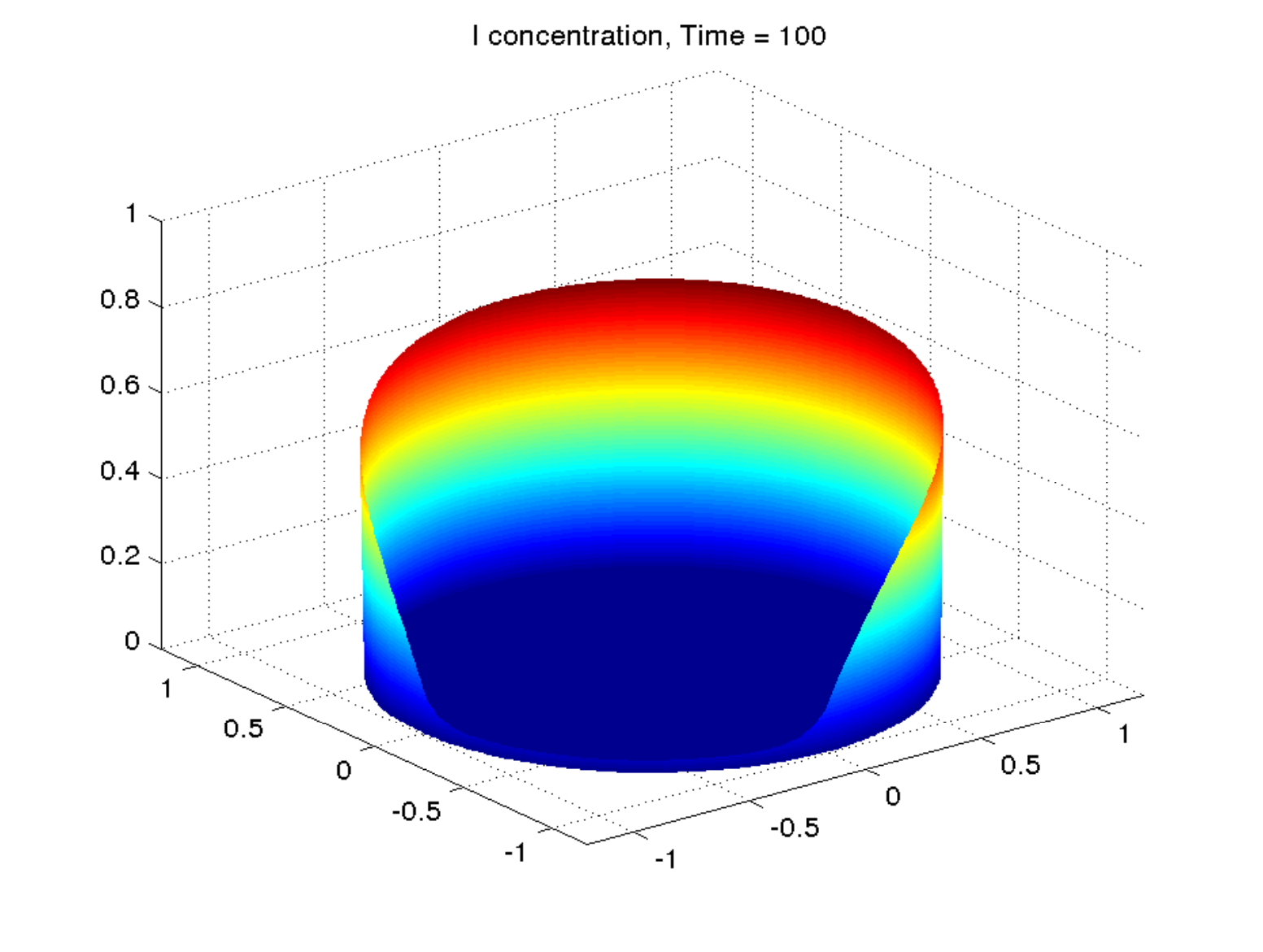}
           \caption{$l$-Lgl with surface diffusion}\label{Figldiff}
           %\label{fig:sub1}
         \end{subfigure}%
         \begin{subfigure}{.5\textwidth}
           \centering
           \includegraphics[width=1\linewidth]{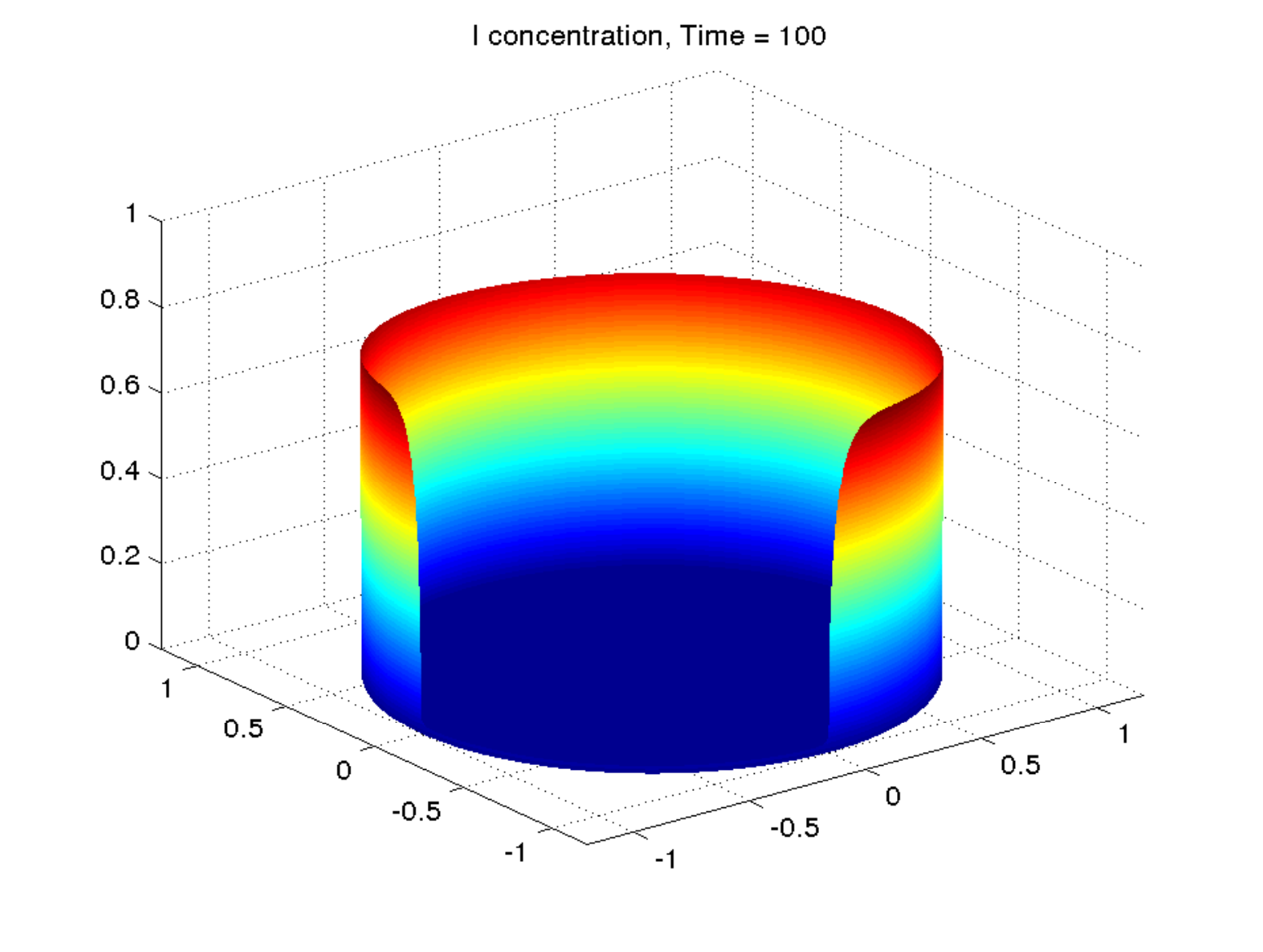}
           \caption{$l$-Lgl without surface diffusion}\label{Figlnodiff}
           %\label{fig:sub2}
         \end{subfigure}
         %\caption{Surface diffusion (A) and without surface diffusion (B)}
         \begin{subfigure}{.5\textwidth}
           \centering
           \includegraphics[width=1\linewidth]{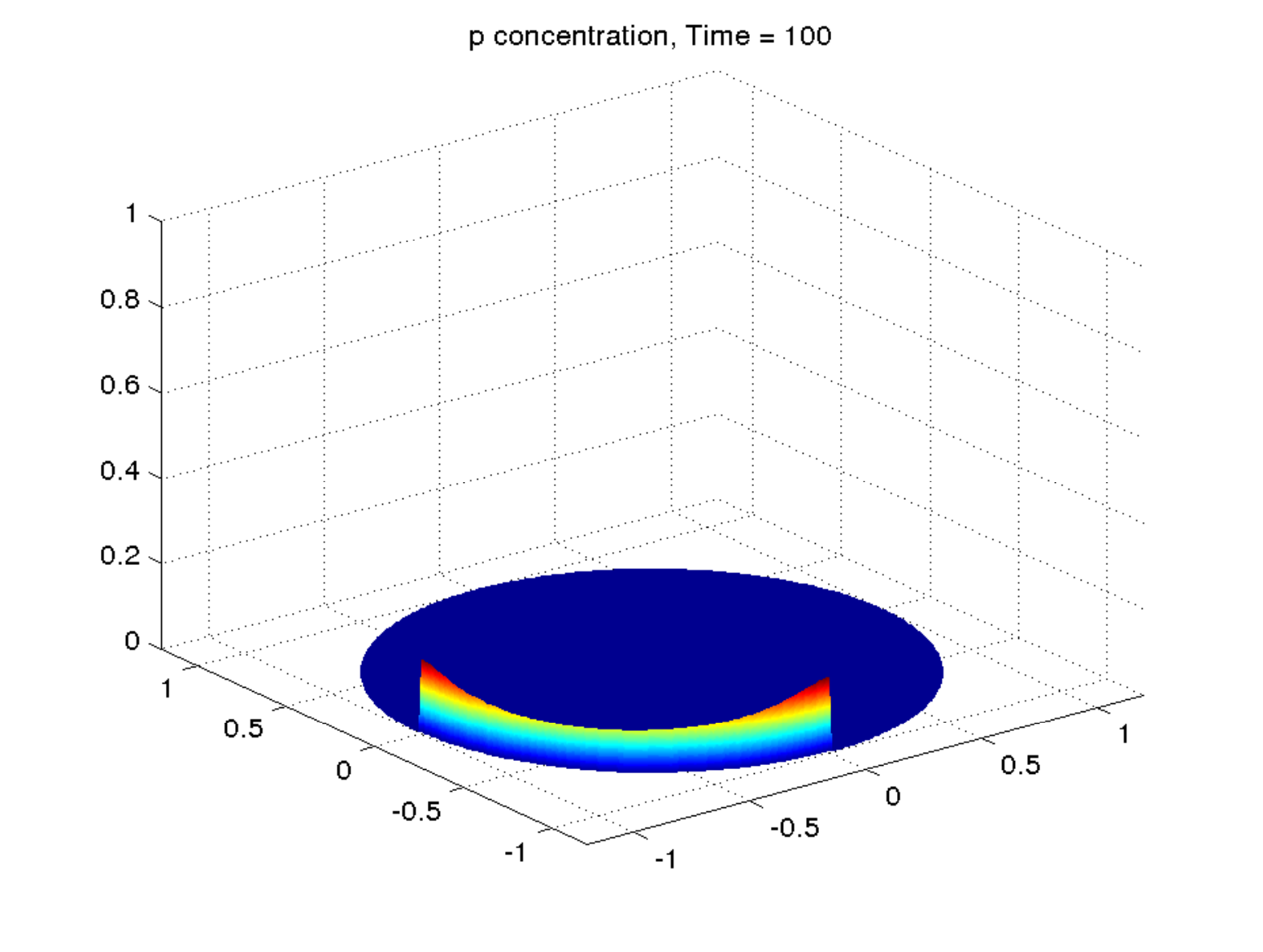}
           \caption{$p$-Lgl with surface diffusion}\label{Figpdiff}
           %\label{fig:sub1}
         \end{subfigure}%
         \begin{subfigure}{.5\textwidth}
           \centering
           \includegraphics[width=1\linewidth]{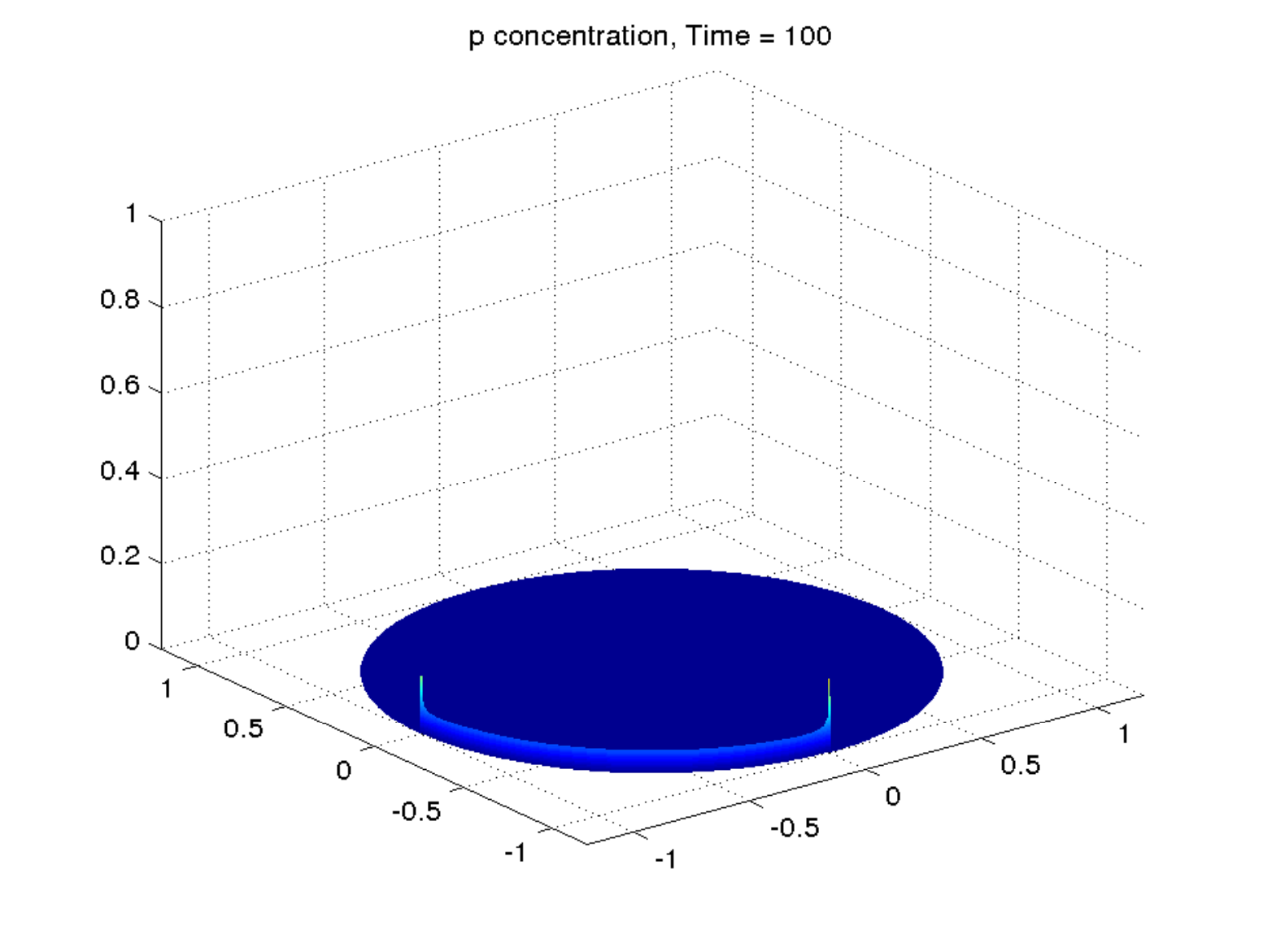}
           \caption{$p$-Lgl without surface diffusion}\label{Figpnodiff}
           %\label{fig:sub2}
         \end{subfigure}
         \caption{Comparison of the numerical stationary states of $l$ and $p$  on a unit-circular cell (i.e. the blue circle centred at the origin) with surface diffusion rates $d_l = 0.02$, $d_p = 0.04$ (Figs. \ref{Figldiff} and \ref{Figpdiff}) and without surface diffusion $d_l = d_p = 0$ (Figs. \ref{Figlnodiff} and \ref{Figpnodiff}) for the parameters \eqref{reac}, \eqref{diff} and initial data \eqref{InitialData}: The $l$-Lgl concentration strongly decreases towards the active boundary part $\Gamma_2$ located within the negative quadrant in the front (Figs. \ref{Figldiff} and \ref{Figlnodiff}). In Fig. \ref{Figldiff}, the decrease of the $l$-Lgl is clearly smoothed by the surface diffusion. In Fig. \ref{Figlnodiff} the decay profile is steep yet still smoothed by an indirect surface diffusive effect caused by volume diffusion and reversible reaction. This is confirmed by Figs. \ref{Figpdiff} and \ref{Figpnodiff} plotting the subsequent concentrations of phosphorylated $p$-Lgl. Fig \ref{Figpnodiff} shows in particular the spikes of $p$-Lgl produced near $\partial\Gamma_2$ due to the indirect surface diffusion effect.}
         \label{fig3}
         \end{figure}
         
         The first three numerical test examples Figs. \ref{fig3}, \ref{figLP} and \ref{figBigDiff} illustrate the role of surface diffusion by comparing the numerical stationary state solutions of the system for two cases: i) with surface diffusion rates $d_l = 0.02$, $d_p = 0.04$ and ii) without surface diffusion, i.e. $d_l = d_p = 0$.
        
        \medskip
         
         Figure \ref{fig3} plots the resulting numerically stationary state concentrations of non-phos\-phorylated cortical Lgl $l$ and phosphorylated cortical Lgl $p$ (for the generic parameters \eqref{reac}, \eqref{diff} and the initial data \eqref{InitialData}). 
         
         In the case with surface diffusion, Figure \ref{Figldiff} shows a smoothly decaying profile of $l$ around the boundary points $\partial\Gamma_2$, i.e. around the points $(-1,0)$ and $(0,-1)$, where the lower concentration of $l$ on $\Gamma_2$ is the result of $l$ being converted into $p$. The corresponding 
         numerical steady state concentration of $p$ on $\Gamma_2$ is plotted in 
         Figure \ref{Figpdiff}. The increase of $p$ towards the points $(-1,0)$ and $(0,-1)$ corresponds to the increasing values of $l$ over and beyond these boundary points. 
         
         As comparison, the Figures \ref{Figlnodiff} and \ref{Figpnodiff} show the numerical stationary state concentrations of $l$ and $p$ without surface diffusion. Due to the absence of surface diffusion, Figure \ref{Figlnodiff} depicts a significantly sharper profile of $l$ around the boundary points $\partial\Gamma_2$. However, the profile in $l$ is still smooth and so is the corresponding profile of the stationary state concentration of $p$ on $\Gamma_2$, which is shown in \ref{Figpnodiff} using a very highly refined mesh to eliminate potential numerical artefacts. 
In our understanding, these sharp yet smooth profiles of $l$ and $p$ 
are the combined effect of the volume diffusion of $L$ and $P$ and the reversible reactions between $L$ and $l$, 
which transfer a diffusive effect from the volume $\Omega$ onto the 
boundary $\Gamma$. In the context of reaction-diffusion systems with partially degenerate diffusion, such an indirect diffusion effect has been analytically shown  first in \cite{DF07} and for general linear weakly-reversible systems recently in \cite{FPT}.

         \medskip

         %This is due to the diffusion of $l$-Lgl along $\Gamma_2$, which somehow transports materials into domain to create a "hump" of $L$-Lgl.

         \begin{figure}[htp]
         	\begin{subfigure}{.45\textwidth}
         		\centering
         		\includegraphics[width=1\linewidth]{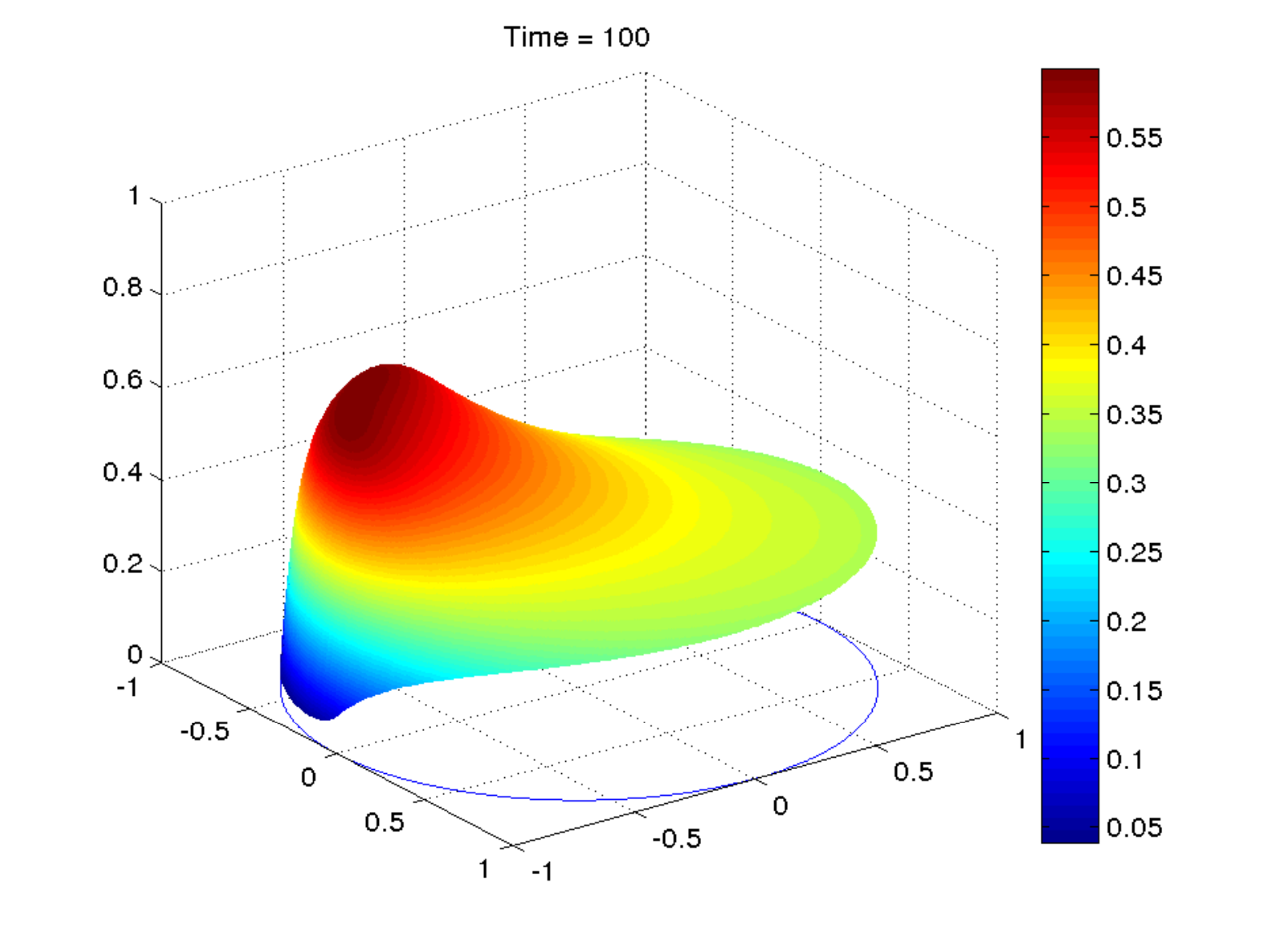}
         		\caption{$L$-Lgl with surface diffusion}\label{FigLdiff}
         	\end{subfigure}
         	\begin{subfigure}{.45\textwidth}
         		\centering
         		\includegraphics[width=1\linewidth]{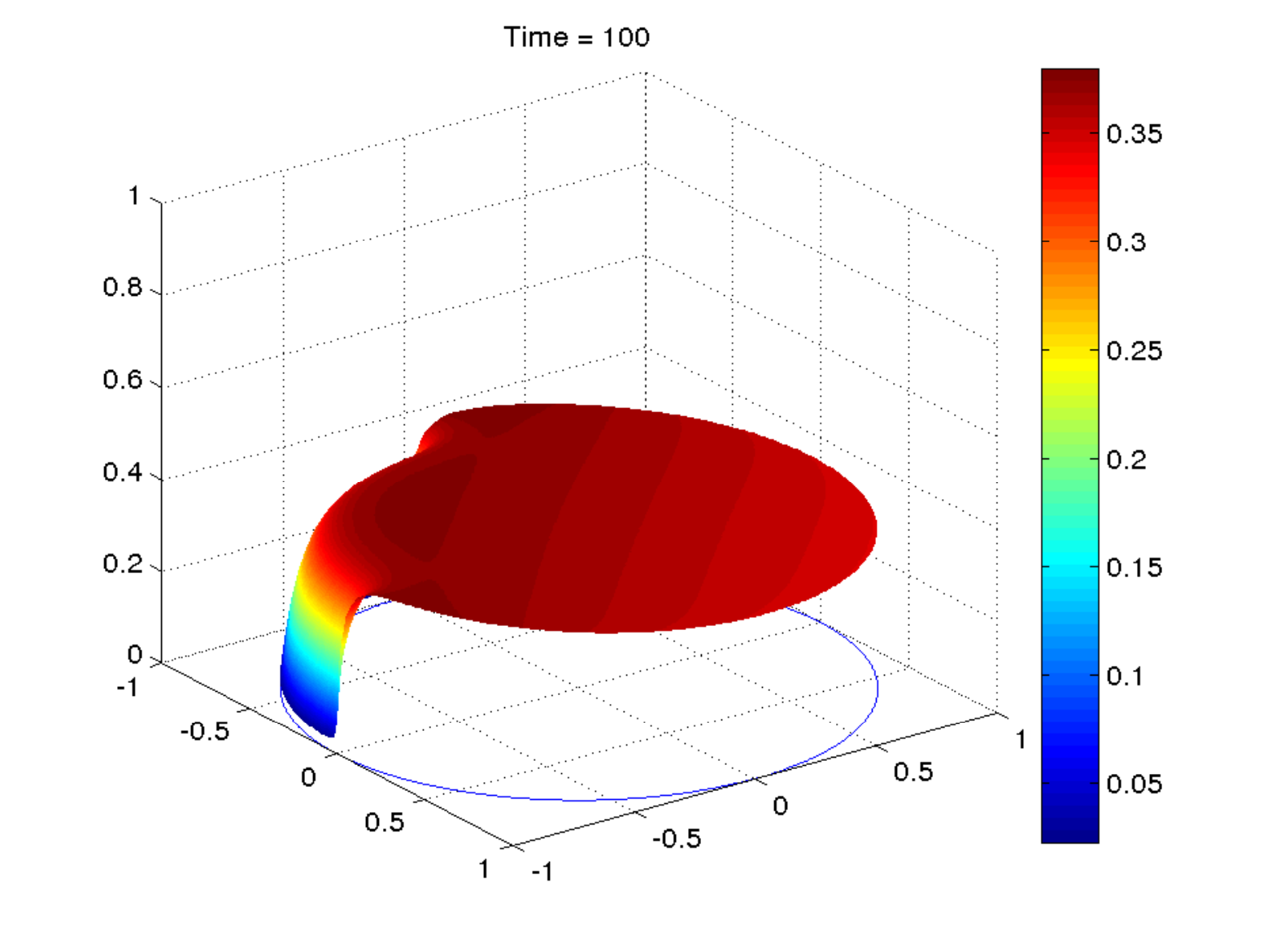}
         		\caption{$L$-Lgl without surface diffusion}\label{FigLnodiff}
         	\end{subfigure}
         	\begin{subfigure}{.45\textwidth}
         		\centering
         		\includegraphics[width=1\linewidth]{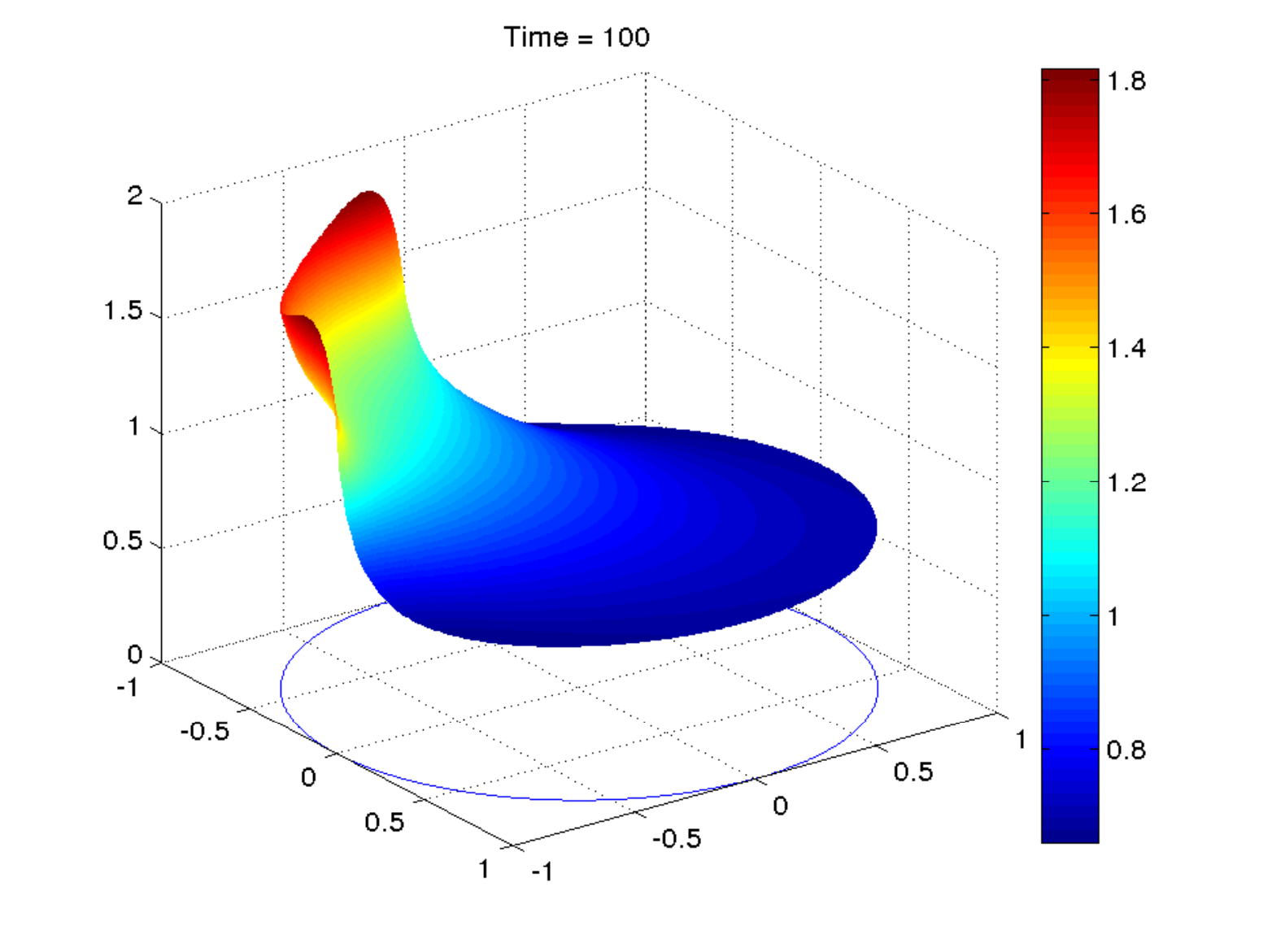}
         		\caption{$P$-Lgl with surface diffusion}\label{FigPdiff}
         	\end{subfigure}	
         	\begin{subfigure}{.45\textwidth}
         		\centering
         		\includegraphics[width=1\linewidth]{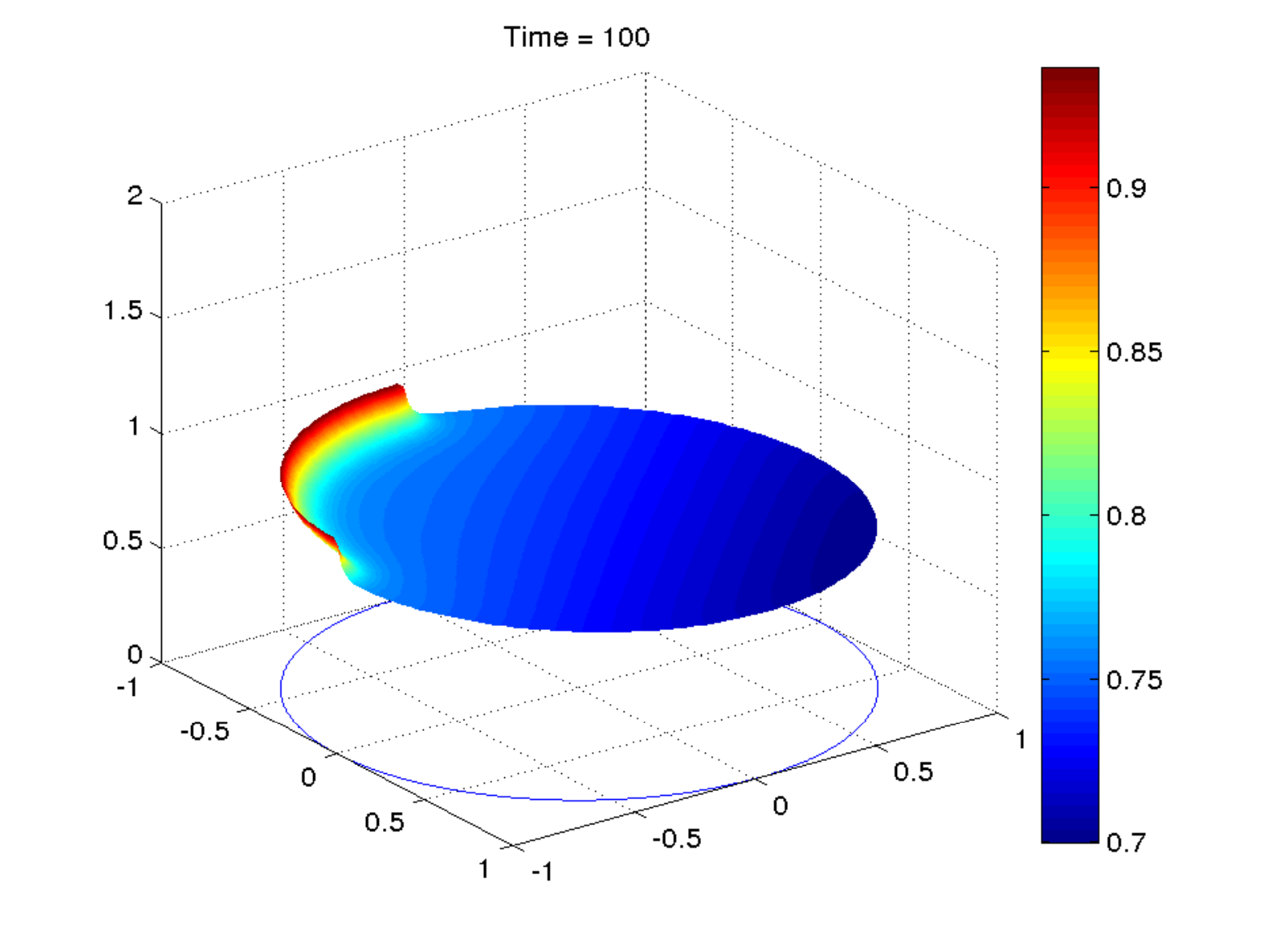}
         		\caption{$P$-Lgl without surface diffusion}\label{FigPnodiff}
         	\end{subfigure}	
         	\caption{Concentrations of the numerical stationary state of $L$ and $P$ on a unit-circular cell with surface diffusion rates $d_l = 0.02, d_p = 0.4$ (Figs. \ref{FigLdiff} and \ref{FigPdiff}) and without surface diffusion $d_l = d_p = 0$ (Figs. \ref{FigLnodiff} and \ref{FigPnodiff}) for the parameters \eqref{reac}, \eqref{diff} and initial data \eqref{InitialData}. Figs. \ref{FigLdiff} and \ref{FigPdiff} show how surface diffusion allows to a lateral flow of $l$-Lgl towards the active boundary $\Gamma_2$ (localised at the negative quadrant on the left side of the plot),
where it becomes phosphorylated and released into the cytoplasm creating high $P$-Lgl concentrations near $\Gamma_2$, see Fig. \ref{FigPdiff}. Subsequently, the reaction $P\xrightarrow{\alpha}L$ leads to a strong "hump" of $L$-Lgl away from the boundary $\Gamma_2$ within $\Omega$. Fig. \ref{FigLnodiff} and \ref{FigPnodiff} confirm that without surface diffusion, no such "hump" can occur. This examples also underlies that complex balance systems like model \eqref{f1}--\eqref{f3} behave significantly more intricate than detailed balance systems.}
         	\label{figLP}
         \end{figure}
         
         Figure \ref{figLP} plots the volume concentrations $L$ and $P$
         corresponding to Figure \ref{fig3}. 
         In the case with surface diffusion, Figure \ref{FigLdiff} shows very interestingly and somewhat surprisingly a "hump" in the numerical stationary state concentration of $L$ near the active boundary part $\Gamma_2$.
         This "hump" is not visible in Figure \ref{FigLnodiff} in the case without surface diffusion. The corresponding volume concentrations of $P$ in the Figures 
         \ref{FigPdiff} (with diffusion) and \ref{FigPnodiff} (without diffusion)
         allow to explain this "hump" as the effect of surface diffusion leading to a significant additional transport of $l$-Lgl (compared to the case without surface diffusion) along $\Gamma$ to the active boundary part $\Gamma_2$, where $l$-Lgl becomes phosphorylated into $p$-Lgl, which 
         is subsequently released from the cortex and thus results into a much higher concentration of $P$ along $\Gamma_2$ as depicted in \ref{FigPdiff}. 
         The "hump" in $L$ is then the consequence of the reversible reaction between 
         $L$ and $P$ and the volume diffusion of $L$. 

\medskip

\begin{remark}        
We remark that while Figures \ref{fig3} and \ref{figLP} cannot be viewed as a realistic simulation of the asymmetric localisation of cell-fate determinants, 
the qualitative behaviour of the asymmetric Lgl localisation nevertheless suggests that surface diffusion might play an important role.         
In particular, surface diffusion might help to explain an experimentally observed gap between $\Gamma_2$ the localisation of cell-fate determinant Numb, 
which is observed in neuroblast cell, yet not in SOP cells, see \cite{MEBWK}.
\end{remark}
         
         \medskip
         
                  \begin{figure}[htp]
         	\begin{subfigure}{.45\textwidth}
         		\centering
         		\includegraphics[width=1\linewidth]{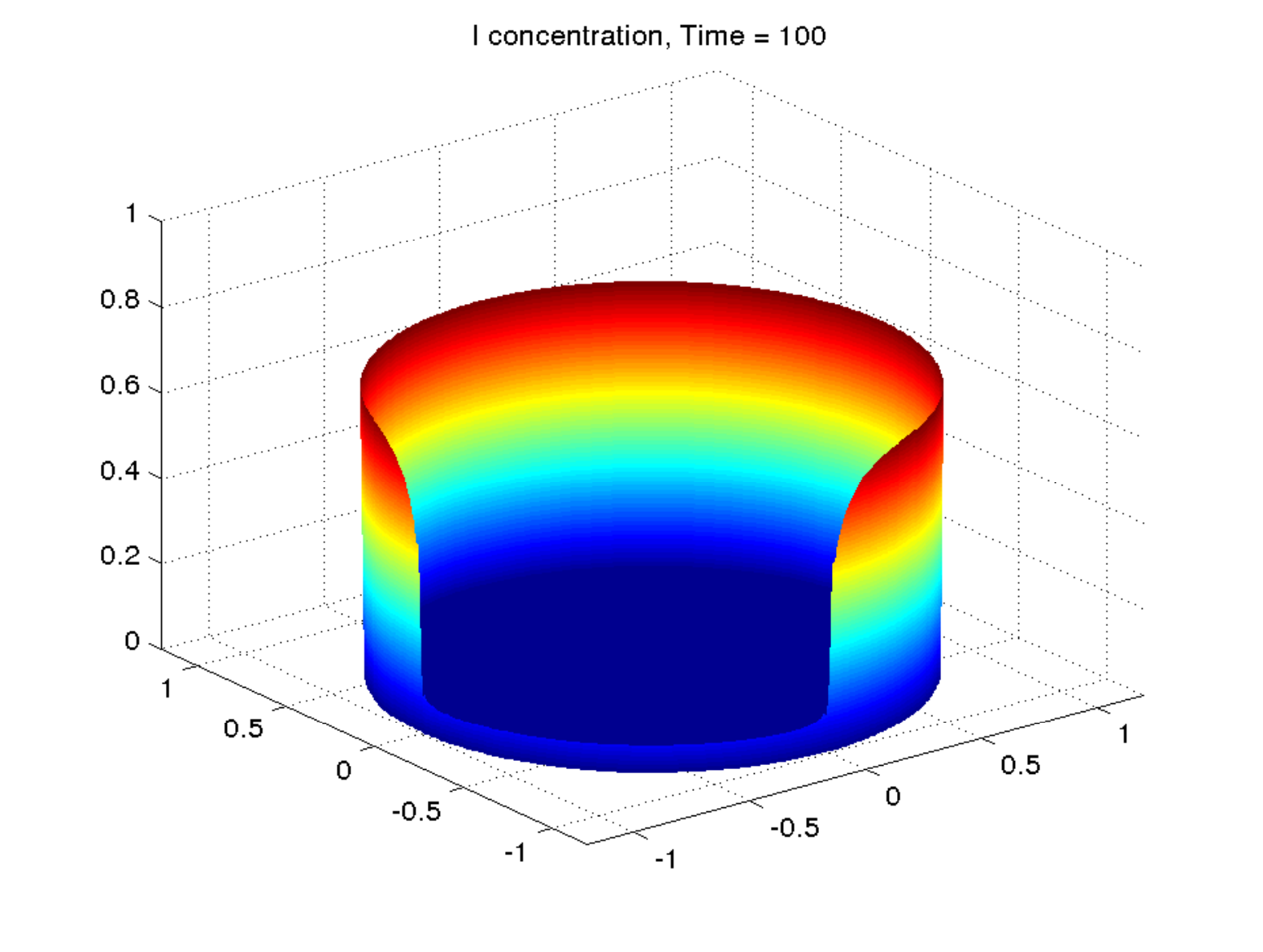}
         		\caption{$l$-Lgl without surface diffusion with big volume diffusion rate $d_L = 0.1$}\label{FiglBigdL}
         	\end{subfigure}	
         	\begin{subfigure}{.45\textwidth}
         		\centering
         		\includegraphics[width=1\linewidth]{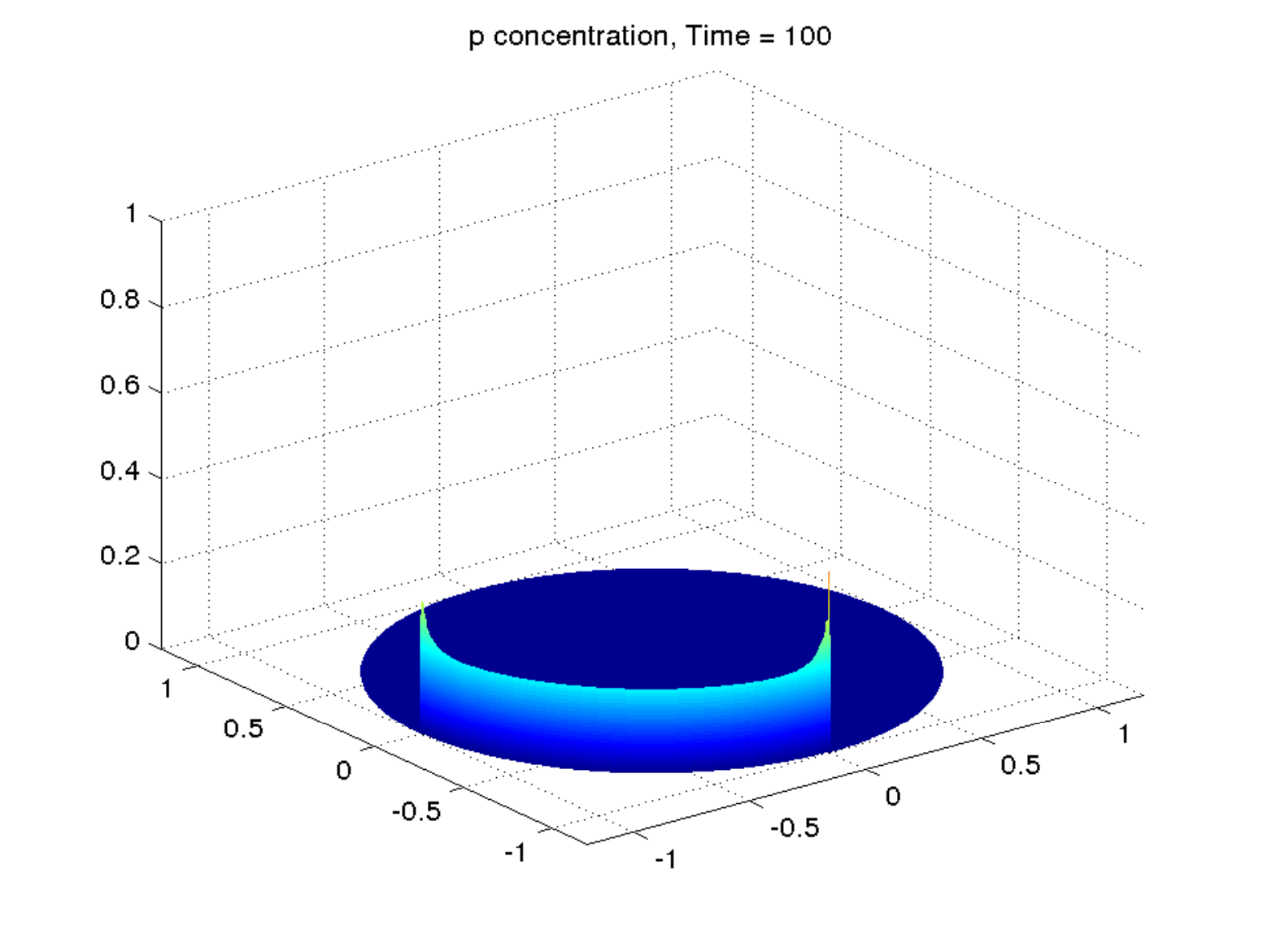}
         		\caption{$p$-Lgl without surface diffusion with big volume diffusion rate $d_L = 0.1$}\label{FigpBigdL}
         	\end{subfigure}	
         	\caption{Concentrations of the numerical stationary state of $l$-Lgl and $p$-Lgl on a unit-circular cell without surface diffusion yet with ten-fold volume-diffusion rate $d_L = 0.1$ while $d_P = 0.02$, $d_l = d_p = 0$ and initial data \eqref{InitialData} remain unchanged. Confirming the explanation of Fig. \ref{fig3} that volume diffusion and reversible reactions induce an indirect surface diffusion effect on $l$ and $p$, Figs. \ref{FiglBigdL} and \ref{FigpBigdL} plot the increased indirect surface diffusion effect on $l$ and even $p$ by increasing tenfold only the volume diffusion rate $d_L$.}
         	\label{figBigDiff}
         \end{figure}

         In Figure \ref{figBigDiff}, we investigate further the case without  surface diffusion by increasing the volume diffusion rate $d_L$ ten-folds compare to \eqref{diff}. More precisely, we set $d_l = d_p = 0$, $d_L = 0.1$ and $d_P = 0.02$.
         %We will compare the stationary profiles of $l$ and $p$, which stated in Figure \ref{figBigDiff},  to those in the case $d_L = 0.2$ which stated in Figure \ref{fig3}.
         
         Figure \ref{FiglBigdL} shows the effects of the increased volume 
         diffusion $d_L = 0.1$. It leads to a certain widening and flattening of the profile 
         of $l$-Lgl over the boundary points of $\Gamma_2$. This profile of $l$ 
         is less steeper than the profile of $l$ in Figure \ref{Figlnodiff} without surface diffusion but still steeper 
         than the profile of $l$ in Figure \ref{Figldiff} with surface diffusion. Thus, the observed effect is consistent and confirms 
         the above explanation that the profile of $l$ and $p$ in 
         cases without surface diffusion are a combined effect of 
         volume diffusion $d_L$ and the reversible reaction of $L \rightleftharpoons l$:
         An increased volume diffusion $d_L$ leads thus to an increased 
         diffusive effect at the surface $\Gamma$. 
         
         The corresponding stationary state profile of $p$-Lgl plotted in Figure \ref{FigpBigdL} shows that also the profiles of $p$-Lgl near the boundary points of $\Gamma_2$ are widened, yet still steeper that the profiles of $p$ in Figure \ref{Figpdiff} with surface diffusion.
         
         The increase of volume diffusion rate $d_L$ does not only affect to profile of $p$ and $l$ around $\Gamma_2$ as discussed above but also changes the absolute value of stationary states of $p$ and $l$ on $\Gamma_2$. More precisely, by comparing Figure \ref{FiglBigdL} and Figure \ref{Figlnodiff} (or Figure \ref{FigpBigdL} and Figure \ref{Figpnodiff}) we see that the absolute value of $p$ and $l$ on $\Gamma_2$ in the case $d_L = 0.1$ are higher than that in the case $d_L = 0.01$.
         
         \subsection{Asymptotic decay of $p$ for large $\xi$}
         
         In SOP cells, the reaction $p \xrightarrow{\xi} P$ of cortical Lgl $p$ to cytoplasmic Lgl $P$ is suggested to be significantly faster  than the other reactions.
         That means that the expulsion rates $\xi$ is expected to be much larger 
         than the generic reaction rates in \eqref{reac}.
         We are thus interested to study the qualitative behaviour for increasing reaction rates $\xi$ while keeping the reaction rates \eqref{reac} fixed.
         
         Intuitively, one expects that when $\xi$ becomes larger and larger, the concentration of $p$-Lgl will decay to zero since the $p$-Lgl is released 
         more and more rapidly to $P$-Lgl.

         \begin{figure}[htp]
         \centering
         \begin{subfigure}{.45\textwidth}
           \centering
           \includegraphics[width=1\linewidth]{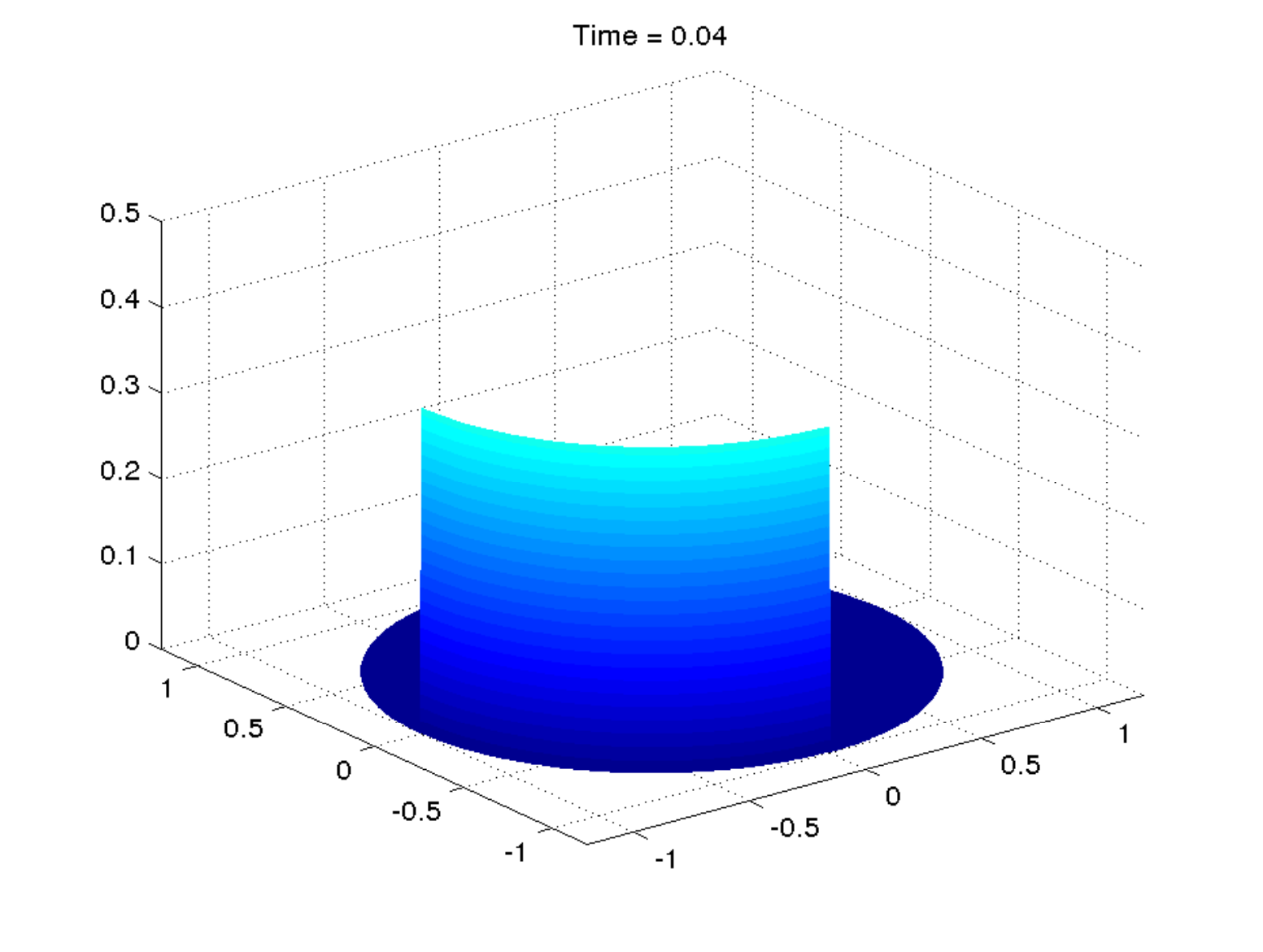}
           \caption{$\xi = 10$}
         \end{subfigure}
         \begin{subfigure}{.45\textwidth}
           \centering
           \includegraphics[width=1\linewidth]{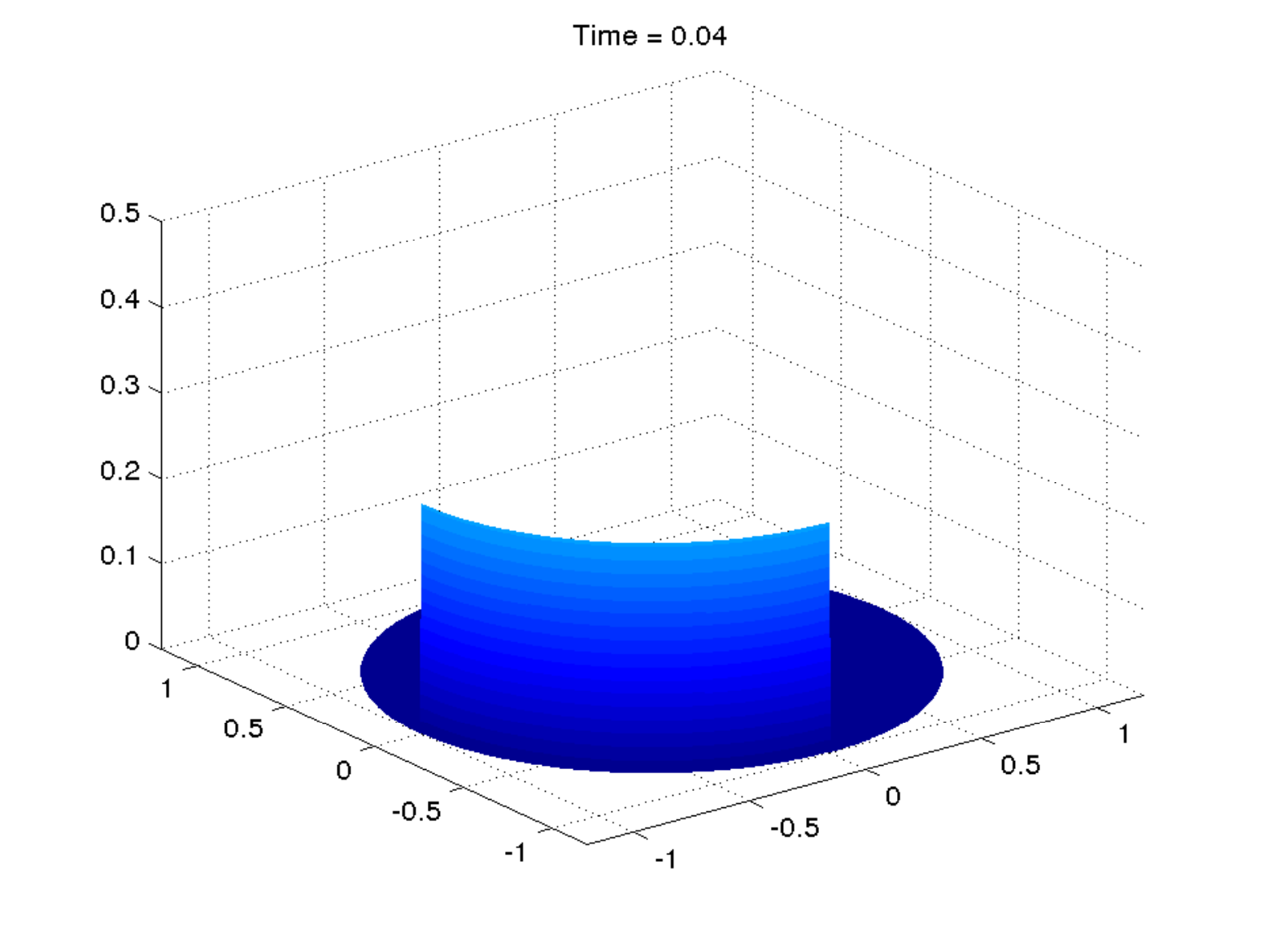}
           \caption{$\xi = 20$}
         \end{subfigure}
         
         \begin{subfigure}{.45\textwidth}
           \centering
           \includegraphics[width=1\linewidth]{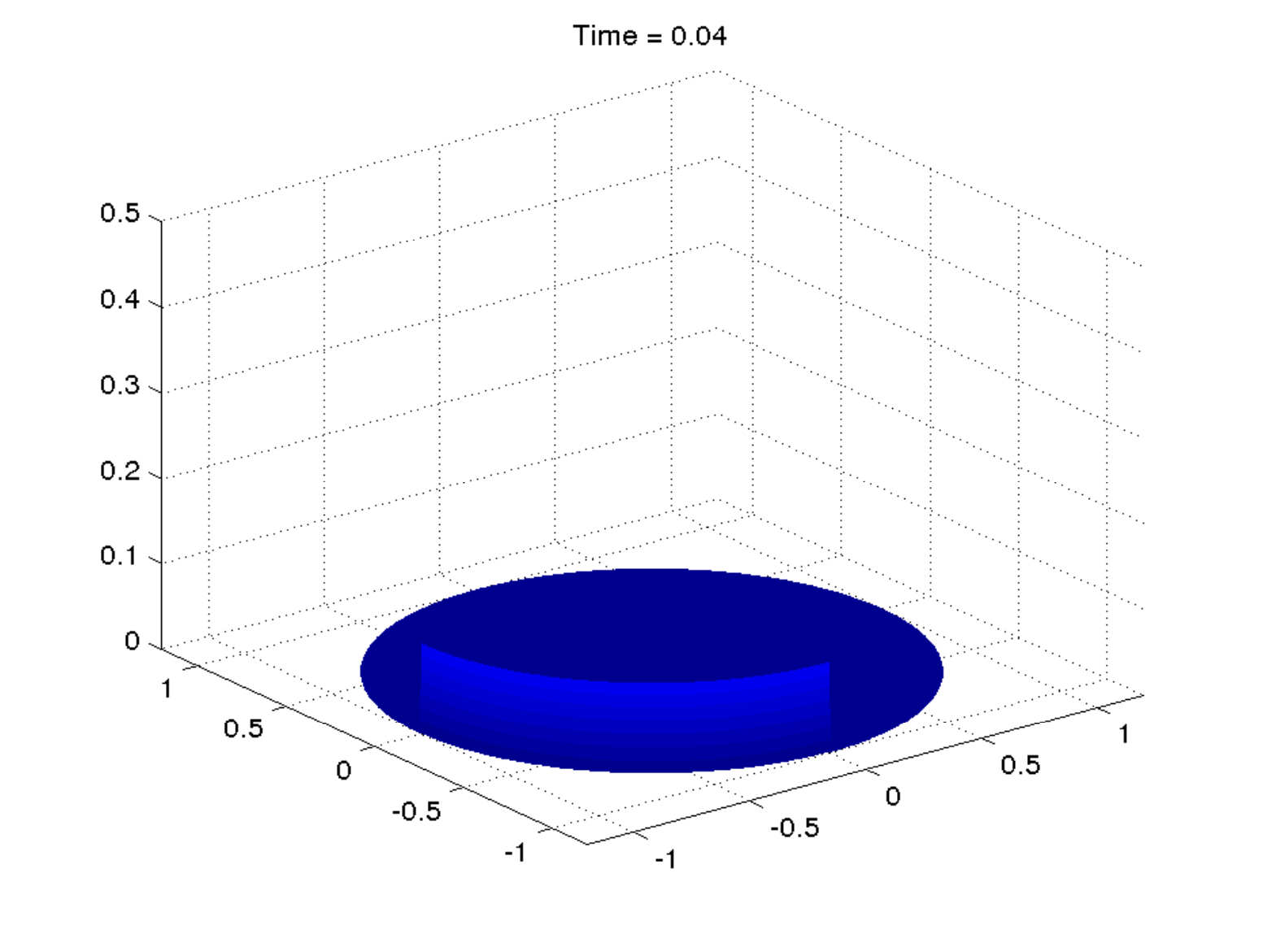}
           \caption{$\xi = 50$}
         \end{subfigure}
         \begin{subfigure}{.45\textwidth}
           \centering
           \includegraphics[width=1\linewidth]{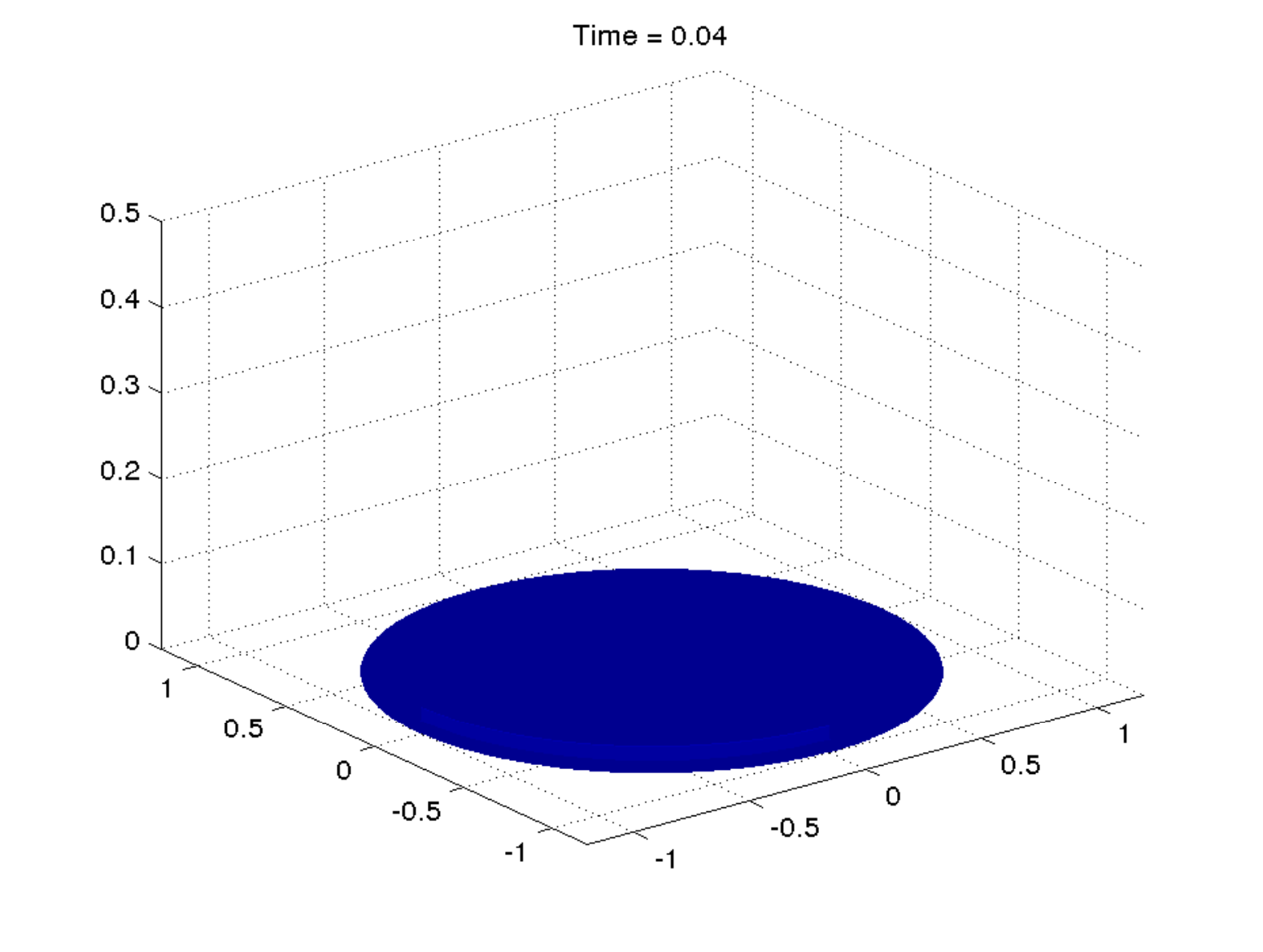}
           \caption{$\xi = 100$}
         \end{subfigure}
         \caption{Comparison of $p$ for $\xi=10,20,50,100$ at time $t = 0.04$ and for the generic parameters \eqref{reac}, \eqref{diff} and initial data \eqref{InitialData} on a unit-circular cell. The plots clearly confirms that $p$  decreases when $\xi$ increases.}
         \label{fig3a}
         \end{figure}
         
         \medskip
         
         In Figure \ref{fig3a}, we compare $p(t,x)$ on $\Gamma_2$ at an early time $t = 0.04$ for four different values of $\xi$ being $10, 20, 50$ and $100$.
         The numerical results show how a larger reaction rate $\xi$ leads 
          to a decay of $p\searrow 0$ on $\Gamma_2$. This happens already at the very small time $t=0.04$ and even more so for larger times (data not shown).
         Observing this fact suggests that  the system \eqref{f1}--\eqref{f3} for large $\xi$
         may be well approximated by a reduced quasi-steady-state approximation (QSSA) without $p$, which
         is formally obtained by letting $\xi\to+\infty$. This QSSA was  rigorously performed in Section \ref{sec:4}.
         
         \subsection{Initial-boundary layers in $P$ for large $\xi$}
         
         The following Figures \ref{InitialMass} and \ref{SteadyState} continue to numerically investigate the system behaviour for small and large $\xi$, i.e. for slow and fast 
         release of cortical $p$. Note that, in this part, we consider the case of no surface diffusion $d_l = d_p = 0$.

         \begin{figure}[htp]
         \centering
         \begin{subfigure}{.4\textwidth}
           \centering
           \includegraphics[width=1\linewidth]{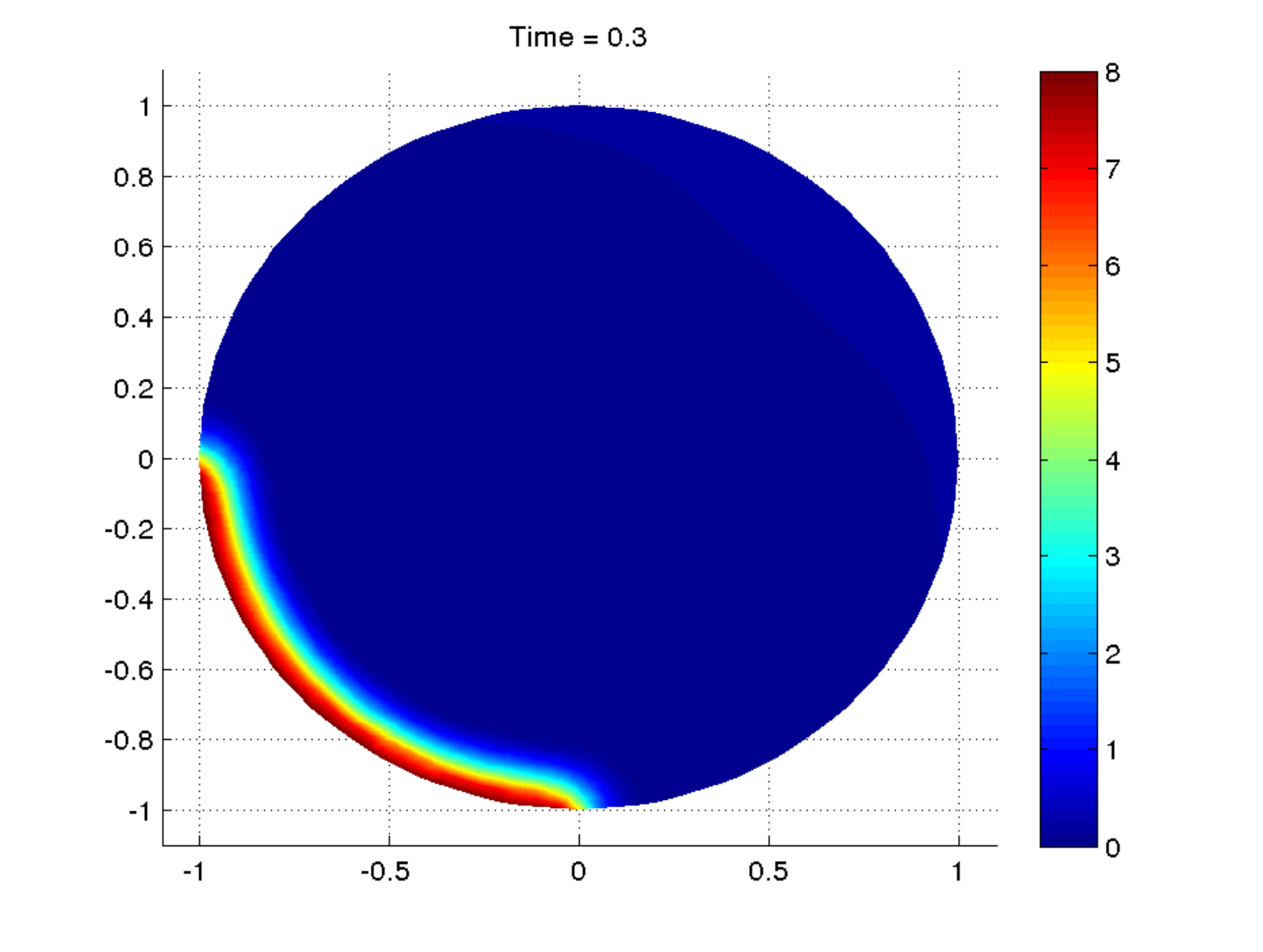}
           \caption{$\xi = 1000$}\label{FigPfast}
           %\label{fig:sub1}
         \end{subfigure}%
         \begin{subfigure}{.4\textwidth}
           \centering
           \includegraphics[width=1\linewidth]{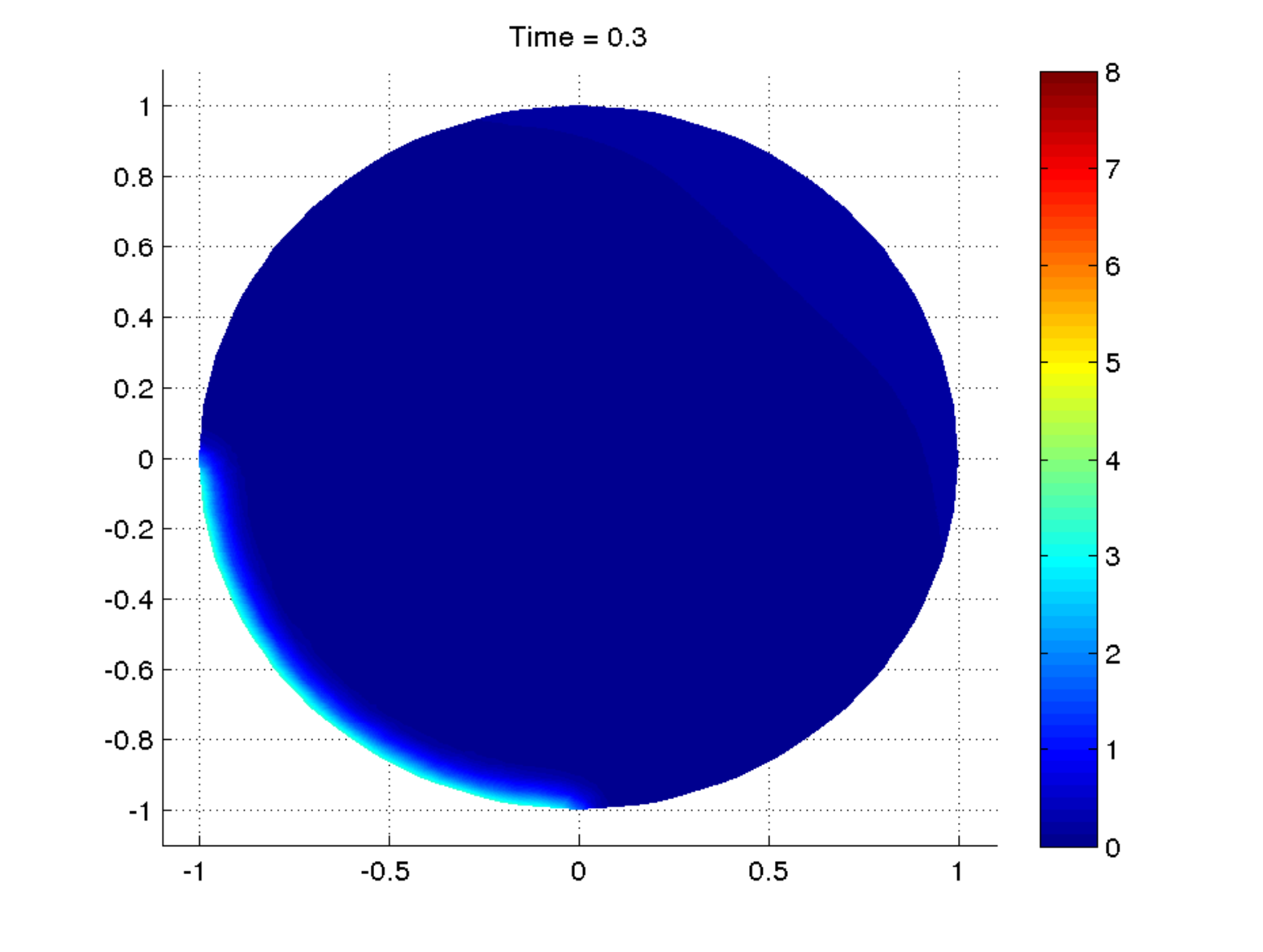}
           \caption{$\xi = 1$}\label{FigPslow}
           %\label{fig:sub2}
         \end{subfigure}
         \caption{Initial-boundary layer in $P$ for $\xi=1000$ and $\xi=1$ at time $t = 0.3$ and for the generic parameters \eqref{reac}, \eqref{diff}, $d_l=d_p=0$ and initial data \eqref{InitialData} on a unit-circular cell. Fig. \ref{FigPfast} for large $\xi$ shows a strongly increased boundary layer near $\Gamma_2$ compared to Fig. \ref{FigPslow}.}
         \label{InitialMass}
         \end{figure}
         
         Figure \ref{InitialMass} compares the cytoplasmic concentration of 
         phosphorylated $P$-Lgl  for two values $\xi=1000$ and $\xi=1$ 
         at the smallish time $t=0.3$ and for the specified, constant initial data \eqref{InitialData}. 
         In particular, Figure \ref{FigPfast} illustrates that the fast reaction $p \xrightarrow{\xi} P$ 
         for $\xi=1000$ leads to much larger values of $P$ near the boundary $\Gamma_2$ as compare to $\xi = 1$. We thus observe the 
         formation of an initial-boundary layer near $\Gamma_2$ in Figure \ref{FigPfast}
         compared to Figure \ref{FigPslow}, which plots $P$ being formed by the slow reaction with $\xi=1$. 
         
         %Figure In the second study case, we are interested in initial layer of $P$ as $\xi \rightarrow +\infty$. 
         %Although $p$ tends to $0$ in any positive time as $\xi\rightarrow+\infty$, the initial mass $p_0$ could not be avoided. Due to the reaction $p \longrightarrow P$, we guess that the initial mass $p_0$ will transfer to the initial mass $P_0$ of concentration $P$ (see Section 4 for the proof). 
         
         \begin{figure}[htp]
         \centering
         \begin{subfigure}{.4\textwidth}
           \centering
           \includegraphics[width=1\linewidth]{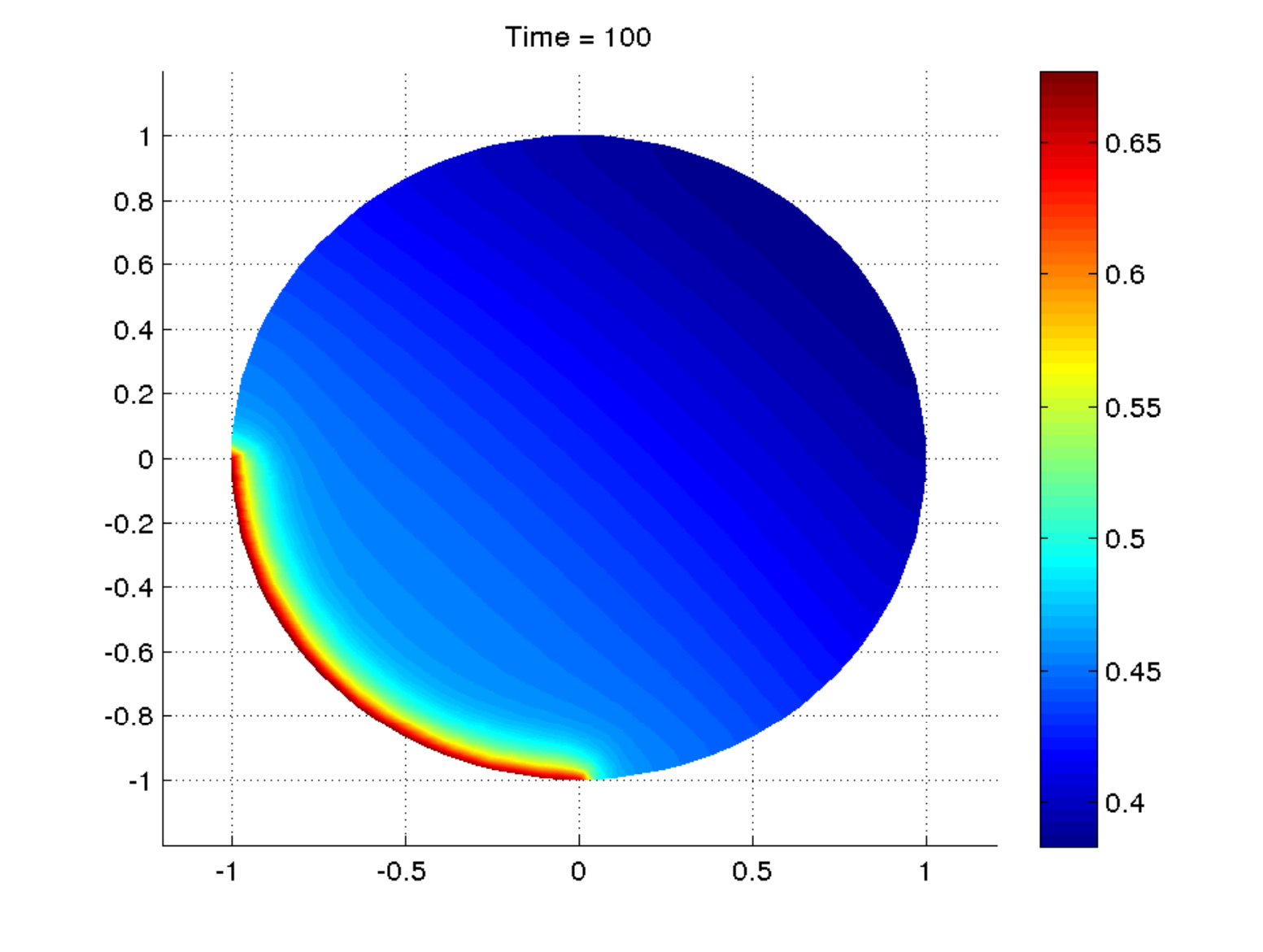}
           \caption{$\xi = 1000$}
           %\label{fig:sub1}
         \end{subfigure}%
         \begin{subfigure}{.4\textwidth}
           \centering
           \includegraphics[width=1\linewidth]{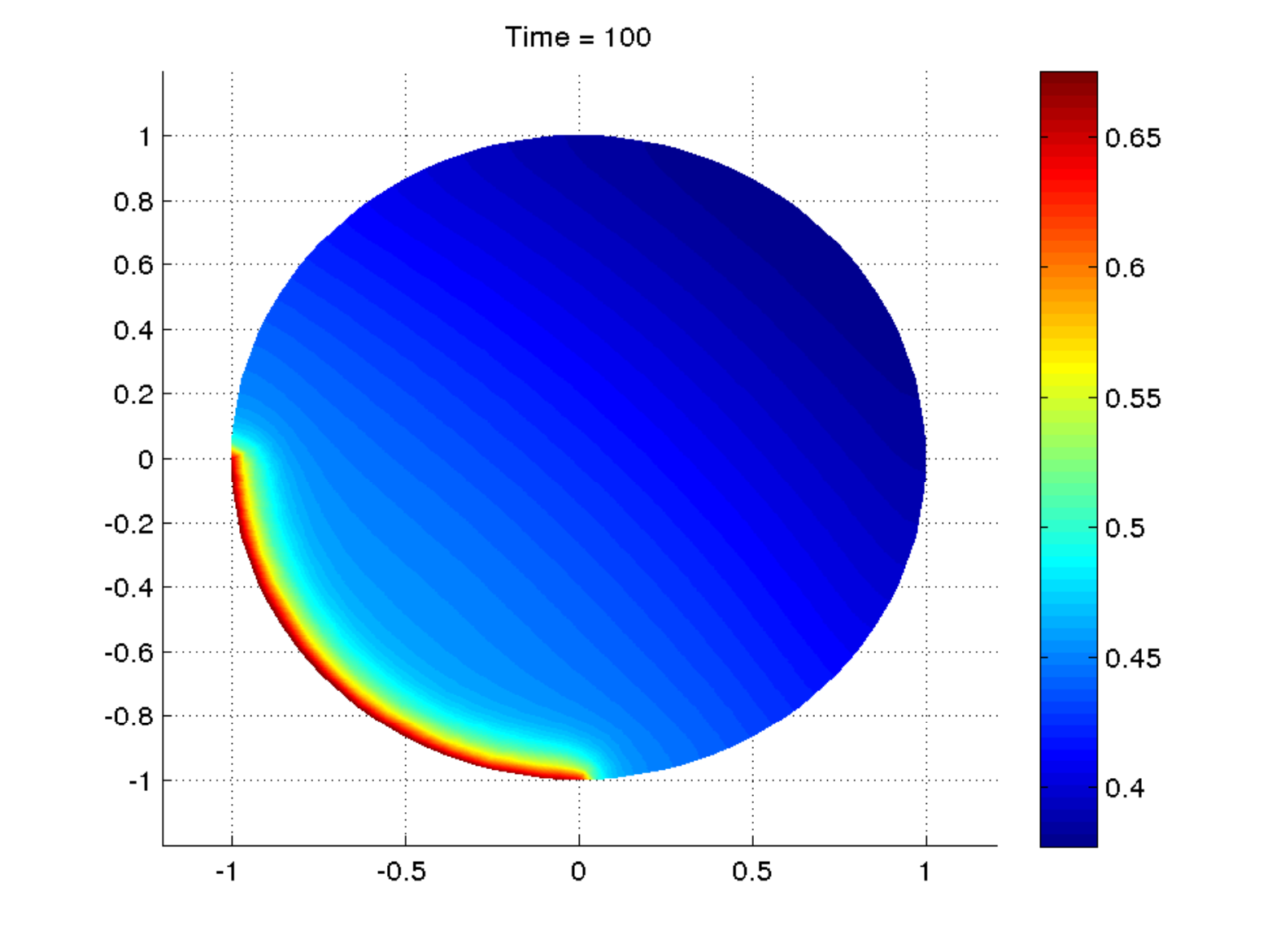}
           \caption{$\xi = 1$}
           %\label{fig:sub2}
         \end{subfigure}
         \caption{Identical numerical stationary state concentrations of $P$ for $\xi=1000$ and $\xi=1$ at the time $t = 100$ and for the generic parameters \eqref{reac}, \eqref{diff}, $d_l=d_p=0$ and initial data \eqref{InitialData} on a unit-circular cell. The plot confirms that the magnitude of $\xi$ only influences the transient behaviour yet not the stationary state, see also Remark \ref{rem41}.}
         \label{SteadyState}
         \end{figure}
         
         \medskip
         
         Finally, Figure \ref{SteadyState} plots the numerical steady state 
         concentrations of $P$ for $\xi=1000$ and $\xi=1$ at the time $t = 100$. We observe that the stationary states appear to be identical and that
         the boundary layer in Figure \ref{InitialMass} is indeed an initial-boundary layer for large $\xi$ and no longer present in the stationary states, which features much lower values of $P$ near the boundary $\Gamma_2$. In fact, 
         we have demonstrated in Remark \ref{rem41} in Section \ref{sec:4}, that the 
         stationary states of system \eqref{f1}--\eqref{f3} without boundary diffusion terms
         \emph{do not depend on the rate $\xi$} and are unique 
         for fixed total initial mass in the mass conservation law \eqref{cons}. 
         Thus, Figure \ref{SteadyState} plots indeed that same stationary state. 
         
%         \newpage
                 
\section{Conclusions}\label{sec:conc}
         
In this paper, we studied a VSRD system, which appears as a model for the asymmetric localisation of Lgl in Drosphila SOP precursor cell upon the activation of the kinase aPKC during mitosis. The challenges of the model system \eqref{f1}--\eqref{f3} lie in the volume-surface coupling as well as in the weakly-reversible reaction/sorption dynamics between the four considered conformations of the key protein Lgl. 

\medskip

The major aim of the paper is to provided the basic mathematical theory concerning such VSRD models: well-posedness, numerical discretisation and rigorous quasi-steady-state approximation with respect to a biologically large parameter. 
We expect that many of the results and methods of this paper can be carried over to similar VSRD models.
\smallskip

In performing the QSSA, mathematical difficulty arises from the volume-surface coupling, in particular from the singularity occurring in the Neumann boundary condition of $P$. In the case without surface diffusion terms, we were able to prove
in Section \ref{sec:4} the necessary a-priori estimates and  rigorously perform the QSSA to a reduced model system. For the full VSRD system  \eqref{f1}--\eqref{f3}, the weakly-reversible structure of \eqref{f1}--\eqref{f3} leads to technical difficulties in finding suitable a-priori estimates and thus prevents so far a rigorous proof of the QSSA.
\medskip

The VSRD model system \eqref{f1}--\eqref{f3} describes the asymmetric localisation of Lgl during mitosis, which subsequently leads to the asymmetric localisation of the cell-fate determinate Numb during the asymmetric stem cell division of Drosophila SOP precursor cells. 
Despite the lack of experimentally measured parameters, the numerical simulation of \eqref{f1}--\eqref{f3} allows nevertheless to study the qualitative behaviour of the model. 
\smallskip

Two such qualitative questions present themselves from the biological background: First, the role of surface diffusion and secondly, the fact the 
$p$-Lgl release rate $\xi$ is large. 
\smallskip

The first point is related to the experimental observation of 
a cortical gap between the Par complex, i.e. the active boundary part $\Gamma_2$ and the resulting localisation of the cell-fate determinants in neuroblasts cells, but not in SOP cells, see \cite{MEBWK}. It has been suggested by the experimentalists, that surface diffusion may play a role in 
creating this spectral gap. 

The numerical simulation performed in Fig. \ref{Figldiff} shows clearly that 
surface diffusion will spread out the $l$-Lgl profile. This can suggest an explanation for the cortical gap in neuroblast cells,
by assuming, for instance, that neuroblast cells need a higher $l$-Lgl concentration to localise Numb than SOP cells, where no spectral gap is observed. 
Moreover, larger surface diffusion rates would lead to further flattened $l$-Lgl profiles and, thus, an even larger cortical gap. 
\smallskip

However, our simulations also present two counter-indications: 
Firstly, even without surface diffusion, the model will always feature an indirect surface diffusion effect as a consequence of volume diffusion and reversible sorption of Lgl between cytoplasm and cortex, yet the magnitude of the indirect surface diffusion effect will always be significantly lower, see Section \ref{subsec:diff} and Fig. \ref{fig3}. Nevertheless, the indirect surface diffusion effect will increase if just one of the volume diffusion rates is increased, see Fig. \ref{figBigDiff} with tenfold increased $d_L$. 

Secondly, Fig. \ref{figLP} shows that surface diffusion 
together with the weakly-reversible structure of the model system  \eqref{f1}--\eqref{f3} can lead to an unexpected maximum of $L$-Lgl in the cytoplasm, 
which seems biologically unlikely. 
\smallskip

As a summary, one could conjecture that while surface diffusion may play an important role in the asymmetric localisation of Lgl and in the subsequent asymmetric stem cell division, the magnitude of the surface diffusion can not be seen independently from the parameters of the weakly-reversible Lgl kinetics and the volume diffusion rates. 
\medskip

Concerning the fast release of $p$-Lgl, Figs. \ref{fig3a} confirms the natural intuition that larger $\xi$ will lead to a faster release of cortical $p$-Lgl and the formation of an initial-boundary layer of $P$-Lgl, see Fig.  \ref{InitialMass}.
However, for the steady state of the system \eqref{f1}--\eqref{f3} and also for the reduced QSSA system \eqref{ff1}--\eqref{ff3}, Fig. \ref{SteadyState}
shows that the steady state of these two weakly-reversible systems is independent of  $\xi$ in the case with zero surface diffusion, see Remark \ref{rem41} (and also essentially independent of $\xi$ in cases with sufficiently small surface diffusion, data not shown). 
\smallskip

A biological interpretation of this observation has to recall that non-phosphorylated cortical Lgl, i.e. $l$-Lgl is active in the localisation of Pon and Numb. The concentration of $l$-Lgl, however, is not a direct results of the fast release of
$p$-Lgl, but the consequence of the dynamics of the entire weakly-reversible system \eqref{f1}--\eqref{f3}. 

For very large release rates $\xi$, it thus follows that the other reaction/sorption rates of   \eqref{f1}--\eqref{f3} will be rate limiting in converging to a steady state and the overall process of asymmetric Lgl localisation will no further be speeded up. 
Nevertheless, for sufficiently large  $\xi$, the QSSA obtained in Section \ref{sec:4} will certainly provide a good approximation, which will even predict the identical stationary state compared to \eqref{f1}--\eqref{f3}, see Remark \ref{rem41}.

\section{Appendix}\label{sec:app}
         
         \subsection{The proof of Lemma \ref{compactness}, cf. \cite{BP,PP}.}
         \begin{proof}
         The prove of the Lemma is based on a duality argument. We shall denote by 
         $$
         \mathfrak T^*: (\Phi_i)_{0\leq i\leq N}\in C_0^{\infty}(\Omega)\times (C^{\infty}_0(\Omega_T))^N \rightarrow (z(0), z, z|_{\partial\Omega})
         $$
         the adjoint operator $\mathfrak T^*$ of $\mathfrak T$, where $z$ is the solution of
         \begin{equation}\label{adjoint}
         \begin{cases}
         -z_t - d_P\Delta z = \Phi_0 - \sum_{i=1}^{N}\partial_{x_i}\Phi_i,\\ 
         d_P \partial z/\partial\nu = 0, \qquad \qquad          z(T) = 0.
         \end{cases}
         \end{equation}
         By integration by parts: for $\Phi = (\Phi_i)_{1\leq i\leq N}$
         \begin{align*}
                  \langle \mathfrak T^*(\Phi_0, \Phi), (w_0, \Theta, g)\rangle &=\langle (z(0), z, z|_{\Gamma}), (w_0, \Theta, g)\rangle\\
%                  = \int_{\Omega}z(0)w_0 + \int_{\Omega_T}z\Theta + \int_{\Gamma_T}zg\\
                  &= \int_{\Omega}z(0)w_0+\int_{\Omega_T}z(w_t - d_P\Delta w)  + \int_{\Gamma_T}zg\\
                  %&= \int_{\Omega}z(0)w_0+\left(\int_{\Omega}z(t)w(t)dx\right)\biggl|_{t=0}^{t=T} - \int_{\Omega_T}wz_t\\
                  %&\quad\quad + d_P\int_{\Omega_T}\nabla z \nabla w-\int_{\Gamma_T}zg+ \int_{\Gamma_T}zg\\
%                  &= -\int_{\Omega_T}wz_t + d_P\int_{\Omega_T}\nabla w\nabla z\\
                  &= \int_{\Omega_T}-w(z_t + d_P\Delta z)
                  = \int_{\Omega_T}-w(-\Phi_0 + \nabla\cdot\Phi)\\
                  &= \int_{\Omega_T}(\Phi_0w + \Phi\nabla w)
                  = \langle (\Phi_0, \Phi), (w_0, \nabla w)\rangle\\
                  &= \langle (\Phi_0,\Phi), \mathfrak T(w_0, \Theta, g)\rangle.
                  \end{align*}
         
         The adjointness can be checked by integration by parts: For $\Phi = (\Phi_i)_{1\leq i\leq N}$
         \begin{align*}
         \langle \mathfrak T^*(\Phi_0, \Phi), (w_0, \Theta, g)\rangle &=\langle (z(0), z, z|_{\Gamma}), (w_0, \Theta, g)\rangle
         = \int_{\Omega}z(0)w_0 + \int_{\Omega_T}z\Theta + \int_{\Gamma_T}zg\\
         &= \int_{\Omega}z(0)w_0+\int_{\Omega_T}z(w_t - d_P\Delta w)  + \int_{\Gamma_T}zg\\
         %&= \int_{\Omega}z(0)w_0+\left(\int_{\Omega}z(t)w(t)dx\right)\biggl|_{t=0}^{t=T} - \int_{\Omega_T}wz_t\\
         %&\quad\quad + d_P\int_{\Omega_T}\nabla z \nabla w-\int_{\Gamma_T}zg+ \int_{\Gamma_T}zg\\
         &= -\int_{\Omega_T}wz_t + d_P\int_{\Omega_T}\nabla w\nabla z
         \end{align*}
         by using \eqref{dual} and after integration by parts. Further, with $\partial z/\partial\nu = 0$, we continue
         \begin{align*}
         \langle \mathfrak T^*(\Phi_0, \Phi), (w_0, \Theta, g)\rangle
         &= \int_{\Omega_T}-w(z_t + d_P\Delta z)
         = \int_{\Omega_T}-w(-\Phi_0 + \nabla\cdot\Phi)\\
         &= \int_{\Omega_T}(\Phi_0w + \Phi\nabla w)
         = \langle (\Phi_0, \Phi), (w_0, \nabla w)\rangle\\
         &= \langle (\Phi_0,\Phi), \mathfrak T(w_0, \Theta, g)\rangle.
         \end{align*}
         It is well-known (see e.g. \cite{Lady}) that for $p>N/2+1$, $q>N+2$ and $X = L^p(\Omega_T)\times (L^q(\Omega_T))^N$, the solution $z$ to \eqref{adjoint} satisfies for a small enough $\alpha >0$
         $$
         \|z\|_{C^{\alpha}(\Omega_T)} \leq \kappa\|(\Phi_0, \Phi)\|_X,
         $$
         where $\kappa$ does not depend on $\Phi_0, \Phi$. Thus, due to the dense embedding $C_0^{\infty}\times(C_0^{\infty}(\Omega_T))^{N}\hookrightarrow L^{p}(\Omega)\times(L^{q}(\Omega_T))^N$, we can uniquely extend  $\mathfrak T^*$ to a continuous operator from $X$ into $C^{\alpha}(\Omega)\times C^{\alpha}(\Omega_T)$ and consequently to a compact operator from $X$ into $L^{\infty}(\Omega)\times L^{\infty}(\Omega_T)\times L^{\infty}(\Gamma_T)$. It implies that $\mathfrak T$ can be defined as a compact operator from $L^1(\Omega)\times L^1(\Omega_T)\times L^1(\Gamma_T)$ into $X' = L^r(\Omega_T)\times (L^s(\Omega_T))^N$ for all $r<(N+2)/N$ and $s<(N+2)/(N+1)$. By taking $r = s = 1$, we can complete the proof.
         \end{proof}
         
\subsection{Numerical discretisation}     \label{appnum}    
         
The following first order finite element scheme for system \eqref{f1}--\eqref{f3}
has been used for the numerical simulations in Section \ref{sec:3}
         \subsubsection*{Time-discretisation}
         We apply a first order implicit Euler scheme as time discretisation, which is well known to be stable for such linear problems (see e.g. \cite{Johnson}). 
         More precisely, for a given time step $h$, we shall denote by $L_n(x) := L(nh, x)$ and $L_{n+1}(x) := L((n+1)h, x)$, respectively and analog for $P, l$ and $p$. Thus, we have  for 
         $n\ge0$ the following iteration of semi-discretised systems :
         \begin{equation}
         \begin{cases}
         -hd_L\Delta L_{n+1} + (1+h\beta)L_{n+1} - h\alpha P_{n+1} = L_n, &\qquad x\in\Omega,\\
         -hd_P\Delta P_{n+1} + (1+h\alpha)P_{n+1} -h\beta L_{n+1} = P_n, &\qquad x\in\Omega,\\
         -hd_l\Delta_{\Gamma}l_{n+1} + (1+h(\gamma + \sigma\chi_{\Gamma_2}))l_{n+1} - h\lambda L_{n+1} = l_n, &\qquad x\in\Gamma,\\
         -hd_p\Delta_{\Gamma_2}p_{n+1} + (1+h\xi)p_{n+1} - h\sigma l_{n+1} = p_n, &\qquad x\in\Gamma_2,
         \end{cases}
         \end{equation}
         with boundary condition
         \begin{equation}
         \begin{cases}
         d_L\partial L_{n+1}/\partial\nu = -\lambda L_{n+1} + \gamma l_{n+1}, &\qquad x\in\Gamma,\\
         d_P\partial P_{n+1}/\partial\nu = \xi\chi_{\Gamma_2}p_{n+1}, &\qquad x\in\Gamma,\\
         d_p\partial p_{n+1}/\partial\nu_{\Gamma_2} = 0, &\qquad x\in\partial\Gamma_2.
         \end{cases}
         \end{equation}
         \subsubsection*{Space-discretisation}
         We use a standard finite element method for space discretisation of the cell volume. More precisely, the domain $\Omega$ is approximated by a triangulation mesh $\mathcal T_\eta$ where $\eta$ is the maximum diameter of the triangles. We will use as basis functions the space of continuous, piecewise linear functions on triangles. Although finite element methods are well known for linear reaction-diffusion problems on bounded domains, we ought to remark the following three points specific to this work by addressing surface diffusion as well as the active and nonactive parts of the boundary:
         \begin{itemize}
         \item To deal with the equations on the boundary, $\Gamma$ will be approximated by a polygon $\partial\mathcal T_\eta$. Such a discretisation was already successfully applied for a linear elliptic system featuring mixed volume-surface diffusion in \cite{ER}, where also an error analysis was carried out.
         \item The triangulation is made such that the boundary $\partial\Gamma_2$, which are just two points $(-1,0)$ and $(0,-1)$ in the considered case $\Omega\subset\mathbb R^2$, coincides with the vertices of one or more triangles. Moreover, due to the discontinuity of $\chi_{\Gamma_2}$ we shall significantly refine the mesh in the proximity of these two points as can be seen in Figure \ref{fig:mesh}.
%          The refinement was done by an adaptive strategy implemented in the function "adaptmesh" in Matlab using two point sources at $(-1,0)$ and $(0,-1)$. 
          We remark that, for the sake of clarity, the mesh given in Figure \ref{fig:mesh}, which is obtained after one mesh refinement, has approximately $4000$ elements. Later in this paper, to produce high resolution pictures, we will use a mesh created by five mesh refinements, which contains about $65000$ elements.
         
         \begin{figure}[htp]
         \centering
         \includegraphics[scale=.35]{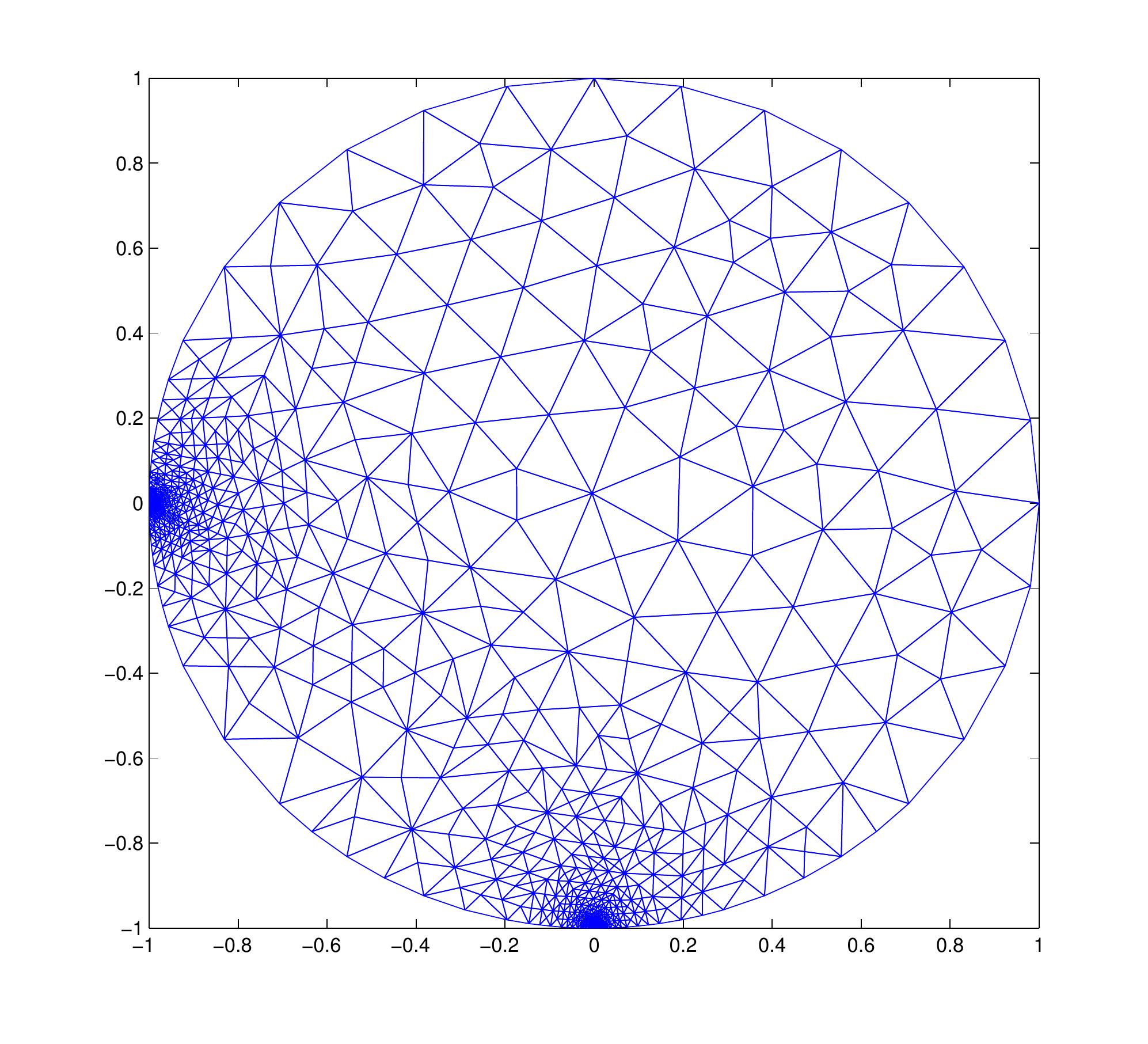}
         \caption{Triangulation mesh and refinement in the proximity of $\partial\Gamma_2$, i.e. the points $(-1,0)$ and $(0,-1)$, 
         which are discontinuity points of the system.}
         \label{fig:mesh}
         \end{figure}
         
         \item The Laplace-Beltrami operator $\Delta_{\Gamma}$ and $\Delta_{\Gamma_2}$ on the boundary, which represent the surface diffusion, $\Gamma$ can be approximated by the Laplace-Beltrami operator on $\partial\mathcal T_{\eta}$. By choosing a polygon as approximation of the boundary $\Gamma$, the operator $\Delta_{\Gamma}$ can itself be approximated by an operator $\Delta_{\partial\mathcal T_{\eta}}$, see e.g. \cite{ER}. Because $\partial\mathcal T_{\eta}$ is a union of disjoint segments, the operator $\Delta_{\partial\mathcal T_{\eta}}$ can be split to act on each segment separately. Moreover, since we use a weak/variational problem formulation, we only have to compute the tangential gradient of affine basis functions on the approximating segments and remark that in this case the tangential gradient coincides with the directional derivative. 
%         As illustrating example, consider a segment $AB$ and the basis function $\varphi: AB \rightarrow [0,1]$ such that $\varphi(A) = 1$ and $\varphi(B) = 0$ then
%         $$
%         \nabla_{\overrightarrow{AB}}\varphi = -\frac{1}{|\overrightarrow{AB}|}
%         $$
%         where $\nabla_{\overrightarrow{AB}}\varphi$ denotes the tangential derivative of $\varphi$ along the direction $\overrightarrow{AB}$.
         
         Note that in the case of a circle or a sphere, we could alternatively use spherical coordinates to discretise the Laplace-Beltrami operator (see e.g. \cite{NGCRSS}). However, the above discretisation has the advantage to work for any sufficiently smooth domain $\Omega$, which can be well approximated by linear segments.
         \end{itemize}

         \vskip 1cm
         \noindent{\bf Acknowledgements}. We would like to thank the reviewers for their very helpful comments and suggestions which improve the paper.
         
         B.Q. Tang is supported by International Research Training Group IGDK 1754. K. Fellner gratefully acknowledges partial support by NAWI Graz. The authors would like to thank very much Herbert Egger for his support concerning the numerical discretisation.

\end{document}